\documentclass[12pt]{article}
\usepackage{mathrsfs,amsthm,graphicx,color,verbatim,bbm,amsmath,amsfonts,amssymb,newclude,nicefrac,graphicx,enumerate,hyperref,bm,geometry,mathabx}
\usepackage[T1]{fontenc}
\usepackage{todonotes}
\usepackage[capitalise]{cleveref} 

\theoremstyle{plain}
\newtheorem{theorem}{Theorem}[section]
\newtheorem{lemma}[theorem]{Lemma}
\newtheorem{prop}[theorem]{Proposition}
\newtheorem{cor}[theorem]{Corollary}

\newtheorem{setting}[theorem]{Setting}

\theoremstyle{remark}

\theoremstyle{definition}

\newcommand{\mb}[1]{\mathbb{#1}}
\newcommand{\mc}[1]{\mathcal{#1}}
\newcommand{\mf}[1]{\mathfrak{#1}}

\newcommand{\A}{\mathbb{A}}
\newcommand{\B}{\mathbb{B}}

\newcommand{\E}{\mathbb{E}}

\renewcommand{\P}{\mathbb{P}}

\newcommand{\R}{\mathbb{R}}
\newcommand{\N}{\mathbb{N}}
\newcommand{\Sym}{\mathbb{S}}

\newcommand{\Z}{\mathbb{Z}}
\newcommand{\Borel}{\mathcal{B}}

\newcommand{\cF}{\mathcal{F}}

\newcommand{\cO}{\mathcal{O}}

\newcommand{\cV}{\mathcal{V}}

\newcommand{\smallsum}{\textstyle\sum}

\newcommand{\Exp}[1]{ \E \! \left[ #1 \right]}

\newcommand{\EXP}[1]{ \E  [ #1 ]}
\newcommand{\EXPP}[1]{ \E \big[ #1 \big]}

\newcommand{\norm}[1]{ \left\| #1 \right\| }
\newcommand{\Norm}[1]{ \| #1 \| }

\newcommand{\HSnorm}[1]{ \left\vvvert #1 \right\vvvert }

\newcommand{\qandq}{\qquad\text{and}\qquad}

\newcommand{\var}[1]{ \operatorname{Var}\!\left[ #1 \right]}
\newcommand{\Var}[1]{ \operatorname{Var} [ #1 ] }
\newcommand{\VAR}[1]{ \operatorname{Var} \big[ #1 \big]}

\newcommand{\VARRRRR}[1]{ \operatorname{Var} \Bigg[ #1 \Bigg]}

\newcommand{\Cost}{\mf C}

\newcommand{\Forall}{\forall\,}
\newcommand{\is}{\leftarrow}

\makeatletter
\newcommand{\vast}{\bBigg@{3.5}}
\newcommand{\Vast}{\bBigg@{4}}
\makeatother

\newcounter{AuthorCount}
\stepcounter{AuthorCount}

\begin{document}

\title{Overcoming the curse of dimensionality \\
	in the numerical approximation \\
	of high-dimensional semilinear \\
	elliptic partial differential equations}
	
	\author{
		Christian Beck$^{\arabic{AuthorCount}}$, 
		\stepcounter{AuthorCount}
		Lukas Gonon$^{\arabic{AuthorCount}\stepcounter{AuthorCount}}$, 
		and 
		Arnulf Jentzen$^{\arabic{AuthorCount}\stepcounter{AuthorCount},\arabic{AuthorCount}\stepcounter{AuthorCount}}$
		\bigskip
		\setcounter{AuthorCount}{1}
		\\
		\small{$^{\arabic{AuthorCount}
				\stepcounter{AuthorCount}}$ 
			Department of Mathematics, 
			ETH Zurich, 
			Z\"urich,}\\
		\small{Switzerland, 
			e-mail: christian.beck@math.ethz.ch} 
		\\
		\small{$^{\arabic{AuthorCount}
				\stepcounter{AuthorCount}}$ 
			Faculty of Mathematics and Statistics, University of St.~Gallen,}\\
		\small{St.~Gallen, 
			Switzerland, 
			e-mail:	lukas.gonon@unisg.ch}
		\\
		\small{$^{\arabic{AuthorCount}
				\stepcounter{AuthorCount}}$ 
			Department of Mathematics, 
			ETH Zurich, 
			Z\"urich,}\\
		\small{Switzerland, 
			e-mail: arnulf.jentzen@sam.math.ethz.ch} 
		\\
		\small{$^{\arabic{AuthorCount}\stepcounter{AuthorCount}}$ 
			Faculty of Mathematics and Computer Science, 
			University of M\"unster, }\\
		\small{M\"unster, Germany, e-mail: ajentzen@uni-muenster.de}		
	}

\maketitle

\begin{abstract}
Recently, so-called full-history recursive multilevel Picard (MLP) approximation schemes have been introduced and shown to overcome the curse of dimensionality in the numerical approximation of semilinear \emph{parabolic} partial differential equations (PDEs) with Lipschitz nonlinearities. 
The key contribution of this article is to introduce and analyze a new variant of MLP approximation schemes for certain semilinear \emph{elliptic} PDEs with Lipschitz nonlinearities and to prove that the proposed approximation schemes overcome the curse of dimensionality in the numerical approximation of such semilinear elliptic PDEs.
\end{abstract}

\tableofcontents

\section{Introduction}

Partial differential equations (PDEs) are widely used as a modelling tool, e.g., in the natural sciences, in the engineering sciences, or in the financial industry. Usually the exact solutions to specific PDE problems in applications can hardly be found. Thus, there is a high demand for approximative solution techniques. In the scientific literature, there are several well-established approximation methods for PDEs like finite difference methods or finite element methods which work well in low dimensions. However, such classical deterministic approximation methods cannot be used in high dimensions as they suffer from the so-called \emph{curse of dimensionality} in the sense that the computational effort for calculating an approximation grows at least exponentially in the PDE dimension $ d \in \N $. Probabilistic approximation methods like Monte Carlo averaging, on the other hand, do not  suffer from the curse of dimensionality in the numerical approximation of linear second-order parabolic PDEs as well as linear second-order elliptic PDEs. 

Recently, several probabilistic approximation methods for high-dimensional nonlinear PDEs have been proposed in hopes of overcoming the curse of dimensionality in the numerical approximation of nonlinear PDEs. We refer, e.g., to 
	\cite{DeepSplitting,
		DeepKolmogorov,
		BeckEJentzen2017MachineLearning,
		becker2018deep,
		becker2019solving,
		chan2018machine,
		chen2019deep,
		dockhorn2019discussion,
		EHanJentzen2017DeepLearning,
		EYu17,
		FujiiTakahashiTakahashi17,
		goudenege2019machine,
		HanEJentzen2017SolvingHighdimensionalPDEs,
		han2018convergence,
		HanNicaStinchcombe2020,
		henry2017deep,
		hure2019some,
		LongLuMaDong17,
		lye2019deep,
		magill2018neural,
		pham2019neural, 
		SirignanoDGM2017} 
for deep learning based approximation methods for possibly nonlinear PDEs, 
e.g., to 
	\cite{agarwal2017branching,
		belak2019probabilistic,
		bouchard2019numerical,
		bouchard2017numerical,
		Chang2016Branching,
		Labordere2012CounterpartyRiskValuation,
		HenryLabordereEtAl2019BranchingDiffusion,
		LabordereTanTouzi2014NumericalAlgorithm,
		McKean1975ApplicationBrownianMotion,
		Rasulov2019MonteCarlo,
		Skorokhod1964BranchingDiffusion,
		Watanabe1965BranchingProcess} 
for approximation methods for nonlinear second-order parabolic PDEs based on branching diffusions, 
and, e.g., to 
	\cite{AllenCahnPaper,
		hutzenthaler2016multilevel,
		EHutzenthalerJentzenKruse2019MLP,
		Welti2019GeneralisedMLP,
		HutzenthalerJentzenKruse2019GradientDependent,
		Overcoming,	
		hutzenthaler2019overcoming,
		hutzenthaler2017multi}
for full-history recursive multilevel Picard (MLP) approximation methods for nonlinear second-order parabolic PDEs. Numerical experiments raise hopes that deep learning based approximation methods are able to approximate solutions of high-dimensional nonlinear PDE problems, but at the moment there are only partial explanations for the good performance of deep learning based approximation methods in numerical experiments for high-dimensional PDEs available (cf., e.g., \cite{BernerGrohsJentzen2018,
	ElbraechterSchwab2018,
	GononGrohsJentzenKoflerSiska2019Uniform,
	GrohsWurstemberger2018,
	Grohs2019ANNCalculus,
	grohs2019deep,
	HutzenthalerJentzenKruse2019,
	JentzenSalimovaWelti2018,
	KutyniokPetersen2019,
	ReisingerZhang2019}). 
In contrast to this, there are, however, complete mathematical analyses in the scientific literature showing that branching diffusion approximation methods can efficiently solve high-dimensional semilinear parabolic PDE problems for sufficiently small time horizons. 
The branching diffusion approximation method, however, breaks down when the time horizon is not sufficiently small anymore. MLP approximation methods are, to the best of our knowledge, up to now the only approximation methods for high-dimensional semilinear parabolic PDEs which are guaranteed by rigorous mathematical proofs to overcome the curse of dimensionality for all time horizons (see \cite{AllenCahnPaper,Welti2019GeneralisedMLP,Overcoming,hutzenthaler2019overcoming}). 
The MLP approximation schemes proposed and studied so far in the scientific literature, 
however, exclusively deal with parabolic PDE problems and none of these schemes and their analyses can be applied to nonlinear elliptic PDEs. 
The key contribution of this article is to introduce and analyze MLP approximation schemes for certain semilinear elliptic PDEs and to prove for the first time that approximations of such semilinear elliptic PDEs can be obtained without suffering from the curse of dimensionality. 
In particular, the main result of this article, \cref{prop:overcoming_the_curse}, proves that the computational effort to obtain an approximation of a desired accuracy $ \varepsilon \in (0,\infty) $ grows at most polynomially in the PDE dimension $ d\in\N $ as well as in the reciprocal of the desired accuracy.
To illustrate the findings of this article in more detail, we present in the following result, \cref{thm:intro} below, a special case of \cref{prop:overcoming_the_curse}. 

\begin{theorem} \label{thm:intro}
	Let $ c,L \in [0,\infty) $, 
		$ \lambda \in (L,\infty) $,
		$ M \in \N \cap ( (\sqrt{\lambda}+\sqrt{L})^2(\sqrt{\lambda}-\sqrt{L})^{-2}, \infty ) $, 
		$ \Theta = \cup_{n\in\N} \Z^n $, 
	let $ u_d \in C^2(\R^d,\R) $, $ d \in \N $, 
	and $ f_d \in C(\R^d\times\R,\R) $, $ d \in \N $, 
	satisfy for all $ d \in \N $, $ x \in \R^d $ that 
		\begin{equation} \label{intro:elliptic_PDEs}
		(\Delta u_d)(x) = f_d(x,u_d(x)),  
		\end{equation} 
	let $ ( \Omega, \mc F, \P ) $ be a probability space, 
	let $ W^{d,\theta} \colon [0,\infty) \times \Omega \to \R^d $, $ \theta\in\Theta$, $ d \in \N $, be i.i.d.~standard Brownian motions, 
	let $ R^{\theta} \colon \Omega \to [0,\infty) $, $ \theta \in \Theta $, be i.i.d.~random variables, 
	assume that 
		$ (R^{\theta})_{\theta\in\Theta} $ and $ (W^{d,\theta})_{(d,\theta)\in\N\times\Theta}$ are independent, 
	assume for all 
		$ d \in \N $, 
		$ x=(x_1,\ldots,x_d) \in \R^d $, 
		$ v,w \in \R $, 
		$ \varepsilon \in (0,\infty) $ 
	that 
		$ | f_d(x,v) - f_d(x,w) - \lambda ( v - w ) | \leq L | v-w | $, 
		$ | f_d(x,0) | \leq c d^c [1+\sum_{j=1}^d |x_j|]^c $,  
		$ \sup_{y=(y_1,\ldots,y_d)\in\R^d} [|u_d(y)|\exp(-\varepsilon \sum_{j=1}^d |y_j|)] < \infty $,  
	and 
		$ \P(R^0 \geq \varepsilon) = e^{-\lambda \varepsilon} $, 
	let $ U^{d,\theta}_{n} = (U^{d,\theta}_{n}(x))_{x\in\R^d}\colon\R^d\times\Omega\to\R$, $\theta\in\Theta$, $d,n\in\N_0$, 
	satisfy for all 
		$ d,n \in \N $, 
		$ \theta\in\Theta $, 
		$ x \in \R^d $
	that 
		$ U^{d,\theta}_{0}(x) = 0 $ 
	and 
		\begin{equation} \label{intro:mlp}
		\begin{split}
		&U^{d,\theta}_{n}(x) 
		=  
		\frac{-1}{\lambda M^n} \left[ 
		\sum_{m=1}^{M^n} f_d\big(x+\sqrt{2}\,W^{d,(\theta,0,m)}_{R^{(\theta,0,m)}},0\big) 
		\right]
		\\
		& 
		+ \sum_{k=1}^{n-1} \frac{1}{\lambda M^{(n-k)}} 
		\Bigg[ 
		\sum_{m=1}^ {M^{(n-k)}} 
		\bigg( 
		\lambda \Big[  U^{d,(\theta,k,m)}_{k}\big(x+\sqrt{2}\,W^{d,(\theta,k,m)}_{R^{(\theta,k,m)}}\big)
		- 
		U^{d,(\theta,k,-m)}_{k-1}\big(x+\sqrt{2}\,W^{d,(\theta,k,m)}_{R^{(\theta,k,m)}}\big) \Big]
		\\& 
		\qquad \qquad \qquad \qquad \qquad - 
		\Big[   
		f_d\!\left( x+\sqrt{2}\,W^{d,(\theta,k,m)}_{R^{(\theta,k,m)}}, 
		U^{d,(\theta,k,m)}_{k}\big(x+\sqrt{2}\,W^{d,(\theta,k,m)}_{R^{(\theta,k,m)}}\big) \right)
		\\
		& \qquad \qquad \qquad \qquad \qquad \qquad  
		- f_d\!\left(  
		x+\sqrt{2}\,W^{d,(\theta,k,m)}_{R^{(\theta,k,m)}}, 
		U^{d,(\theta,k,-m)}_{k-1}\big(x+\sqrt{2}\,W^{d,(\theta,k,m)}_{R^{(\theta,k,m)}}\big)
		\right) 
		\Big]
		\bigg)
		\Bigg], 
		\end{split} 
		\end{equation}
	and let $ \Cost_{d,n} \in \R $, $ d,n\in\N_0$, 
	satisfy for all 
		$ d,n \in \N_0 $ 
	that 
		$ \Cost_{d,n} \leq (d+1) M^n + \sum_{k=1}^{n-1} M^{(n-k)} ( d 	+ 1 + \Cost_{d,k} + \Cost_{d,k-1} ) $. 
	Then there exist $ \kappa \in \R $ and $ \mf N \colon (0,1] \times \N \to \N $
	such that for all 
		$ \varepsilon \in (0,1] $, 
		$ d \in \N $
	it holds that 	
		\begin{equation} \label{intro:claim}
		\begin{split}
		& 
		\Cost_{d,\mf N_{\varepsilon,d}} \leq \kappa d^{\kappa}
		{\varepsilon}^{-\kappa} 
		\qquad \text{and} \qquad
		\sup\nolimits_{ x \in [-c,c]^d }
		\big( \EXPP{ | u_d(x) - U^{d,0}_{\mf N_{\varepsilon,d}}(x) |^2 }\big)^{\!\nicefrac12} 
		\leq \varepsilon . 
		\end{split}
		\end{equation} 
\end{theorem} 

\cref{thm:intro} is an immediate consequence of \cref{prop:overcoming_the_curse_2}. \cref{prop:overcoming_the_curse_2}, in turn, follows from \cref{prop:overcoming_the_curse}, the main result of this article. In the following we add comments on some of the mathematical objects appearing in \cref{thm:intro}. 
The functions $ u_d \colon \R^d \to \R $, $ d \in \N $, in \cref{thm:intro} above describe the solutions of the elliptic PDEs which we intend to solve approximately; see \eqref{intro:elliptic_PDEs} above. 
The functions $ f_d \colon \R^d \times \R \to \R $, $ d \in \N $, represent the nonlinearities in the elliptic PDEs in \eqref{intro:elliptic_PDEs}. The nonlinearities are assumed to satisfy certain Lipschitz conditions which are formulated with the help of the real numbers $ L \in [0,\infty) $ and $ \lambda \in (L,\infty) $ in \cref{thm:intro} above. 
The random fields $ U^{d,\theta}_n \colon \R^d \times \Omega \to \R $, $ d \in \N $, $ n \in \N_0 $, $ \theta \in \Theta $, in \eqref{intro:mlp} above describe the approximation algorithm which we employ in this work to approximately calculate the PDE solutions $ u_d \colon \R^d \to \R $, $ d \in \N $. 
The motivation  for the specific form of the random fields $ U^{d,\theta}_n \colon \R^d \times \Omega \to \R $, $ d \in \N $, $ n \in \N_0 $, $ \theta \in \Theta $, in \eqref{intro:mlp} stems from multilevel Monte Carlo approximations of Picard approximations of certain stochastic fixed point equations (SFPEs) which are satisfied by the solutions $ u_d \colon \R^d \to \R $, $ d \in \N $, of the elliptic PDEs in \eqref{intro:elliptic_PDEs}. 
The different numbers of Monte Carlo samples in \eqref{intro:mlp} are determined by the constant $ M \in \N $ which is restricted by the real numbers $ L \in [0,\infty) $ and $ \lambda \in (L,\infty) $ used to formulate the Lipschitz assumptions for the nonlinearities $ f_d \colon \R^d \times \R \to \R $, $ d \in \N $, in the PDEs in \eqref{intro:elliptic_PDEs} above. 
For every $ d,n \in \N $, $ x \in \R^d $ we think of the quantity $ \Cost_{d,n} \in \R $ in \eqref{intro:claim} in \cref{thm:intro} as the computational cost to sample one realization of $ U^{d,0}_n(x) \in \R $ (cf.~\cref{subsec:computational_effort} below). 
\cref{thm:intro} thus, roughly speaking, proves that the random fields $ U^{d,0}_n \colon \R^d \times \Omega \to \R $, $ d \in \N $, $ n \in \N_0 $, (see \eqref{intro:mlp} in \cref{thm:intro}) can achieve an approximation accuracy of size $ \varepsilon \in (0,\infty) $ with a computational cost which is bounded by a polynomial in the PDE dimension $ d \in \N $ and in the reciprocal of the desired approximation accuracy $ \varepsilon \in (0,\infty) $ (see \eqref{intro:claim} in \cref{thm:intro}). 
The real number $ c \in [0,\infty) $ is an arbitrarily large real constant which we use to express the growth assumption in \cref{thm:intro} that for all $ d \in \N $, $ x=(x_1,\ldots,x_d) \in \R^d $ it holds that $ |f_d(x,0)| \leq c d^c [1+\sum_{j=1}^d |x_j|]^c $ as well as to specify the regions $ [-c,c]^d \subseteq \R^d $, $ d \in \N $, on which we measure the $L^2$-error between the exact solutions $ u_d \colon \R^d \to \R $, $d \in \N$, of the PDEs in \eqref{intro:elliptic_PDEs} and the random approximations $U^{d,0}_{\mf N_{\varepsilon,d}} \colon \R^d \times \Omega \to \R $, $ d \in \N $, $ \varepsilon \in (0,1] $, in \eqref{intro:claim} in \cref{thm:intro} above. 

The remainder of this article is organized as follows. \cref{sec:stochastic_representation} is devoted to the study of stochastic representations for suitable viscosity solutions of certain semilinear elliptic PDEs. 
In particular, we establish in \cref{sec:stochastic_representation} a one-to-one correspondence between suitable viscosity solutions of certain semilinear elliptic PDEs and solutions of SFPEs associated with such semilinear elliptic PDEs. 
\cref{sec:mlp} focusses on the introduction and the analysis of MLP schemes for the numerical approximation of solutions of certain SFPEs. 
Owing to the results in \cref{sec:stochastic_representation} these MLP schemes are thus apt to numerically approximate suitable viscosity solutions of certain semilinear elliptic PDEs. 
The main error estimates for the MLP approximation schemes can be found in \cref{subsec:bias_variance_estimates,subsec:convergence_analysis_mlp} and a computational cost analysis for the MLP approximation schemes is carried out in \cref{subsec:computational_effort}. 
An overall complexity analysis for the MLP approximation schemes is obtained in \cref{subsec:complexity_analysis} by combining the error estimates from \cref{subsec:bias_variance_estimates,subsec:convergence_analysis_mlp} with the computational cost analysis in \cref{subsec:computational_effort}.  

\section{Stochastic representations for elliptic partial differential equations (PDEs)}
\label{sec:stochastic_representation}

In the main result of this section, \cref{cor:existence_and_uniqueness} in \cref{subsec:existence_and_uniqueness} below, we establish a stochastic representation formula of the Feynman--Kac type for suitable viscosity solutions of certain semilinear elliptic PDEs (cf.~also \cref{cor:existence_and_uniqueness_strange_lipschitz} and \cref{cor:existence_and_uniqueness_polynomial_growing_nonlinearity} in \cref{subsec:existence_and_uniqueness}). 
The established Feynman--Kac type formula will be essential in our error analysis for MLP approximations in \cref{sec:mlp} below. 
Our proof of \cref{cor:existence_and_uniqueness} is particularly based on the following three tools: 
	(i) an existence and uniqueness result for SFPEs, see \cref{cor:existence_of_fixpoint} in \cref{subsec:sfpes_and_pdes} below, 
	(ii) a result which identifies solutions of certain SFPEs as viscosity solutions of certain semilinear elliptic PDEs, see \cref{cor:fixpoint_is_viscosity_solution} in \cref{subsec:sfpes_and_pdes}, and 
	(iii) a uniqueness result for suitable viscosity solutions of certain semilinear elliptic PDEs, see \cref{prop:comparison_principle} in  \cref{subsec:comparison_principle} below. 
Our proof of the existence and uniqueness result in \cref{cor:existence_of_fixpoint} in \cref{subsec:sfpes_and_pdes} (see (i) above) is based on an elementary application of Banach's fixed point theorem which is performed in \cref{subsec:sfpes} (cf.~also \cite{StochasticFixedPointEquations}). 
The identification result for certain semilinear elliptic PDEs in \cref{cor:fixpoint_is_viscosity_solution} in \cref{subsec:sfpes_and_pdes} (see (ii) above) follows from an approximation argument which combines a well-known convergence result for viscosity solutions (cf., e.g., Hairer et al.~\cite[Lemma 4.8]{HaHuJe2017_LossOfRegularityKolmogorov} or Barles \& Perthame~\cite[Theorem A.2]{BarlesPerthame}) with the well-known construction result for classical PDE solutions in \cref{lem:classical_solution_smooth_case} below (cf., e.g., Da Prato \& Zabczyk~\cite[Theorems 7.4.5 and 7.5.1]{DaPZab2002_SecondOrderPDEsInHilbertSpaces}). Our proof of the uniqueness result for suitable viscosity solutions of certain semilinear elliptic PDEs in \cref{prop:comparison_principle} in \cref{subsec:comparison_principle} (see (iii) above) is inspired by Crandall et al.~\cite[Theorem 2.1]{CrandallEvansLions}. 

\subsection{Stochastic fixed point equations (SFPEs)} 
\label{subsec:sfpes}

\begin{lemma}\label{lem:integrability} 
	Let $ d \in \N $, $ c,\gamma,\rho \in \R $, $ \lambda \in (\rho,\infty) $, 
	let $ \mc O \subseteq \R^d $ be a non-empty open set, 
	let $ (\Omega,\mc F,\P) $ be a probability space, 
	let $ X = (X_t)_{t\in [0,\infty)}\colon [0,\infty) \times \Omega \to \mc O $ be $(\Borel([0,\infty))\otimes\mc F)$/$\Borel(\mc O)$-measurable, 
	let $ h\colon \mc O \to \R $ be $\Borel(\mc O)$/$\Borel(\R)$-measurable, 
	let $ V\colon \mc O \to (0,\infty) $ be $\Borel(\mc O)$/$\Borel((0,\infty)) $-measurable, and assume for all 
		$ t \in [0,\infty) $, 
		$ x \in \mc O $ 
	that $\EXP{e^{-\rho t}V(X_t)} \leq \gamma $ and $ |h(x)|\leq c V(x) $. 
	Then it holds that 
	\begin{equation}\label{integrability:claim}
	\Exp{\int_0^{\infty} e^{-\lambda t} |h(X_t)|\,dt} 
	\leq \frac{c\gamma}{\lambda - \rho} . 
	\end{equation}
\end{lemma} 

\begin{proof}[Proof of \cref{lem:integrability}] 
	First, observe that 
	the hypothesis that $h\colon\mc O \to \R$ is $\Borel(\mc O)$/$\Borel(\R)$-measurable 
	and 
	the hypothesis that $ X\colon [0,\infty) \times \Omega \to \mc O $ is $ (\Borel([0,\infty))\otimes\mc F) $/$\Borel(\mc O)$-measurable 
	prove that 
	$ [0,\infty) \times \Omega \ni (t,\omega) \mapsto e^{-\lambda t} h(X_t(\omega)) \in \R $ 
	is $(\Borel([0,\infty))\otimes\mc F)$/$\Borel(\R)$-measurable. This, the hypothesis that for all 
		$x\in\mc O$ 
	it holds that $|h(x)| \leq c V(x)$, 
	the hypothesis that 
	for all $ t\in [0,\infty) $ it holds that $ \EXP{e^{-\rho t}V(X_t)} \leq \gamma $, 
	and Fubini's theorem ensure that 
	\begin{equation}
	\begin{split}
	\Exp{\int_0^{\infty} e^{-\lambda t} |h(X_t)| \,dt} 
	& \leq 
	\Exp{\int_0^{\infty} e^{-\lambda t} c V(X_t) \,dt} 
	= 
	c \int_0^{\infty} e^{-\lambda t} \,\Exp{V(X_t)}\!\,dt 
	\\
	& \leq 
	c\gamma \int_0^{\infty} e^{-(\lambda-\rho)t} \,dt
	= \frac{c\gamma}{\lambda-\rho} . 
	\end{split}
	\end{equation}
	This establishes \eqref{integrability:claim}. The proof of 
	\cref{lem:integrability} is thus completed. 
\end{proof} 

\begin{lemma} \label{bounded_case} 
	Let $ d \in \N $, $ c,\gamma,\rho \in \R $, $ \lambda \in ( \rho, \infty ) $, 
	let $ \norm{\cdot}\colon\R^d \to [0,\infty) $ be a norm on $ \R^d $, 
	let $ \mc O \subseteq \R^d $ be a non-empty open set, 
	let $ ( \Omega, \mc F, \P ) $ be a probability space, 
	for every $ n \in \N_0 $ let $ X_n = ( X_{n,t} )_{t\in [0,\infty)} \colon [0,\infty) \times \Omega \to \mc O $ 
	be $ ( \Borel([0,\infty)) \otimes \mc F ) $/$ \Borel( \mc O ) $-measurable, 
	assume for all 
		$ \varepsilon,t \in (0,\infty) $ 
	that 
		$ \limsup_{ n \to \infty } \P(\norm{X_{n,t} - X_{0,t}} \geq \varepsilon ) = 0$, 
	let $ h \in C( \mc O, \R ) $ be bounded, 
	let $ V \colon \mc O \to (0,\infty) $ be $ \Borel(\mc O) $/$ \Borel((0,\infty))$-measurable, 
	and assume for all 
		$ n \in \N_0 $, 
		$ t \in [0,\infty) $, 
		$ x \in \mc O $
	that 
		$ |h(x)| \leq c V(x) $ 
	and 
		$ \Exp{e^{-\rho t} V(X_{n,t})} \leq \gamma $.  
	Then 
	\begin{enumerate}[(i)]
		\item \label{bounded_case:item1}
		it holds for all $ n \in \N_0 $ that 
			$ \EXPP{\int_0^{\infty} e^{-\lambda t} |h(X_{n,t})|\,dt} < \infty $
		and 
		\item \label{bounded_case:item2} 
		it holds that 
		\begin{equation} \label{bounded_case:claim}
		\limsup_{ n \to \infty } 
		\left| 
		\Exp{\int_0^{\infty} e^{-\lambda t} h(X_{n,t})\,dt} 
		- 
		\Exp{\int_0^{\infty} e^{-\lambda t} h(X_{0,t})\,dt}
		\right| = 0. 
		\end{equation} 
	\end{enumerate}
\end{lemma} 

\begin{proof}[Proof of \cref{bounded_case}]
	First, observe that \cref{lem:integrability} (with 
		$ d \is d $, 
		$ c \is c $, 
		$ \gamma \is \gamma $, 
		$ \rho \is \rho $, 
		$ \mc O \is \mc O $, 
		$ ( \Omega, \mc F, \P ) \is ( \Omega, \mc F, \P ) $,
		$ X \is X_n $, 
		$ h \is h $, 
		$ V \is V $
	for $ n \in \N_0 $
	in the notation of \cref{lem:integrability}) establishes Item~\eqref{bounded_case:item1}. 	
	Next note that Kallenberg~\cite[Lemma 4.3]{Kallenberg2002Foundations}, 
	the assumption that for all $ \varepsilon,t\in (0,\infty) $ it holds that $ \limsup_{n\to\infty} \P(\norm{X_{n,t}-X_{0,t}}\geq\varepsilon) = 0$, and the assumption that $ h \in C(\mc O,\R) $ assure that 
	for all 
		$ \varepsilon, t \in (0,\infty) $
	it holds that 
	\begin{equation} 
	\limsup_{ n \to \infty } 
	\left[ \P(|h(X_{n,t})-h(X_{0,t})|\geq\varepsilon) \right] = 0. 
	\end{equation} 
	Combining this with the assumption that $h$ is bounded and Vitali's convergence theorem implies for all 
		$ t \in (0,\infty) $ 
	that 
		\begin{equation} \label{bounded_case:convergence_on_time_slices}
		 \limsup_{n\to\infty} \Exp{\left|h(X_{n,t}) - h(X_{0,t})\right|} = 0 . 
		\end{equation} 
	Moreover, note that for all 
		$ n \in \N_0 $, 
		$ t \in (0,\infty) $ 
	it holds that 
		\begin{equation} 
		e^{-\lambda t} \, \Exp{|h(X_{n,t})|} 
		\leq c e^{-\lambda t} \, \Exp{V(X_{n,t})} 
		\leq c \gamma e^{-(\lambda-\rho) t}. 
		\end{equation} 
	Lebesgue's dominated convergence theorem and \eqref{bounded_case:convergence_on_time_slices} hence 
	establish Item~\eqref{bounded_case:item2}. The proof of \cref{bounded_case} is thus completed. 
\end{proof}

\begin{lemma} \label{lem:continuity} 
	Let $ d \in \N $, $ \rho \in \R $, $ \lambda \in (\rho,\infty) $, 
	let $ \norm{\cdot} \colon \R^d \to [0,\infty) $ be a norm on $ \R^d $, 
	let $ \mc O \subseteq \R^d $ be a non-empty open set, 
	for every $ r \in (0,\infty) $ let $ O_r \subseteq \mc O $ 
	satisfy 
		$ O_r = \{x \in \mc O \colon \norm{x} \leq r~\text{and}~\{y\in\R^d\colon \norm{y-x}<\nicefrac{1}{r}\}\subseteq \mc O \} $, 
	let $ ( \Omega, \mc F, \P ) $ be a probability space, 
	for every $ x \in \mc O $ let $ X^x = (X^x_t)_{t\in [0,\infty) } \colon [0,\infty) \times \Omega \to \mc O $ be $(\Borel([0,\infty))\otimes \mc F)$/$\Borel(\mc O)$-measurable, 
	assume for all 
		$ \varepsilon,t \in (0,\infty) $
	and all 
		$ x_n \in \mc O $, $ n \in \N_0$, 
	with 
		$ \limsup_{n\to\infty} \norm{x_n-x_0} = 0$ 
	that 
		$ \limsup_{n \to \infty} \P(\norm{X^{x_n}_t - X^{x_0}_t} \geq \varepsilon ) = 0$, 
	let 
		$ h \in C( \mc O, \R) $, 
		$ V \in C( \mc O, (0,\infty) ) $ 
	satisfy for all 
		$ t \in [0,\infty) $, 
		$ x \in \mc O $
	that 
		$ \EXP{e^{-\rho t}V(X^x_t)} \leq V(x) $
	and 
		$ \inf_{r\in (0,\infty)} [\sup_{y\in\mc O\setminus O_r} (\frac{|h(y)|}{V(y)})] = 0 $, 
	and let $ u \colon \mc O \to \R $ satisfy for all 
		$ x \in \mc O $ 
	that 
		\begin{equation} 
		u(x) = \Exp{\int_0^{\infty} e^{-\lambda t} h(X^x_t)\,dt}
		\end{equation} 
	(cf.\ \cref{lem:integrability}). Then 
	\begin{enumerate}[(i)]
		\item \label{continuity:item1} it holds that $ u \in C(\mc O,\R) $ and 
		
		\item \label{continuity:item2} it holds in the case of $\sup_{r\in (0,\infty)}[\inf_{x\in\mc O\setminus O_r} V(x)] = \infty $ that 
		\begin{equation} 
		\limsup_{r\to\infty} \left[ 
		\sup_{x\in\mc O\setminus O_r} \left(  
		\frac{|u(x)|}{V(x)}
		\right)
		\right] = 0. 
		\end{equation} 
	\end{enumerate}
\end{lemma}

\begin{proof}[Proof of \cref{lem:continuity}] 
	Throughout this proof let 
		$ \mf h_n\colon \mc O \to \R $, $ n \in \N $, 
	be compactly supported continuous functions which satisfy 
		\begin{equation} \label{continuity:convergence_of_hs}
		\limsup_{n\to\infty} 
		\left[ 
		\sup_{x \in \mc O} \left( \frac{|\mf h_n(x) - h(x)|}{V(x)} \right) 
		\right] 
		= 0
		\end{equation} 
	(cf.~\cite[Corollary 2.4]{StochasticFixedPointEquations}) and let $ \mf u_n \colon \mc O \to \R $, $ n \in \N $, satisfy for all 
		$ n \in \N $, 
		$ x \in \mc O $ 
	that 
		\begin{equation} \label{continuity:definition_of_uns}
		\mf u_n(x) = \Exp{\int_0^{\infty} e^{-\lambda t} \mf h_n(X^x_t) \,dt} 
		\end{equation} 
	(cf.~\cref{lem:integrability}). Note that the assumption that $ V $ is continuous implies that $ V $ is locally bounded.  \cref{bounded_case} (with 
		$ d \is d $, 
		$ c \is \sup_{y\in\mc O} (\frac{|\mf h_n(y)|}{V(y)}) $, 
		$ \gamma \is \sup_{k\in\N_0} V(x_k) $, 
		$ \rho \is \rho $, 
		$ \lambda \is \lambda $, 
		$ \mc O \is \mc O $, 
		$ (\Omega,\mc F,\P) \is (\Omega,\mc F,\P) $, 
		$ (X_k)_{ k \in \N_0 } \is (X^{x_k})_{ k \in \N_0 } $, 
		$ h \is \mf h_n $, 
		$ V \is V $ 
	for 
		$ n \in \N $, 
		$ (x_k)_{k\in\N_0} \subseteq \mc O $ 
	with 
		$ \limsup_{k\to\infty} \norm{x_k-x_0} = 0 $
	in the notation of \cref{bounded_case}) hence ensures for all 	
		$ n \in \N $, 
		$ x_k \in \mc O $, $ k \in \N_0 $, 
	with 
		$ \limsup_{k\to\infty} \norm{x_k-x_0} = 0 $ 
	that 
		\begin{equation} 
		\limsup_{k\to\infty} | \mf u_n(x_k) - \mf u_n(x_0) | 
		=
		\limsup_{ k \to \infty} 
		\left|   \Exp{\int_0^{\infty} e^{-\lambda t} \mf h_n(X^{x_k}_t) \,dt } 
		- \Exp{\int_0^{\infty} e^{-\lambda t} \mf h_n(X^{x_0}_t) \,dt }
		\right| = 0.  
		\end{equation} 	
	Hence, we obtain for all $ n \in \N $ that $\mf u_n$ is continuous. Next observe that for all 
		$ n \in \N $, 
		$ x \in \cO $ 
	it holds that 
		\begin{equation} \label{continuity:locally_uniform_convergence}
		\begin{split}
		&
		|\mf u_n(x) - u(x)| 
		= 
		\left| \Exp{\int_0^{\infty} e^{-\lambda t} \mf h_n(X^x_t)\,dt} 
		- 
		\Exp{\int_0^{\infty} e^{-\lambda t} h(X^x_t)\,dt} 
		\right|
		\\
		& = 
		\left| 
		\Exp{\int_0^{\infty} e^{-\lambda t} (\mf h_n(X^x_t) - h(X^x_t)) \,dt}
		\right| 
		\leq 
		\Exp{\int_0^{\infty} e^{-\lambda t} |\mf h_n(X^x_t) - h(X^x_t)| \,dt }
		\\
		& = 
		\int_0^{\infty} 
		e^{-\lambda t} \, \Exp{|\mf h_n(X^x_t)-h(X^x_t)|}\!\,dt 
		= 
		\int_0^{\infty} 
		e^{-\lambda t} \,\Exp{\frac{|\mf h_n(X^x_t) - h(X^x_t)|}{V(X^x_t)}V(X^x_t)}\!\,dt 
		\\
		& 
		\leq 
		\int_0^{\infty} e^{-\lambda t} \left[ \sup_{y\in\mc O} \left(\frac{|\mf h_n(y)-h(y)|}{V(y)}\right) \right] 
		e^{\rho t}V(x) \,dt 
		= 
		\frac{V(x)}{\lambda-\rho} \left[ \sup_{y\in\mc O} \left(\frac{|\mf h_n(y)-h(y)|}{V(y)}\right) \right]\!.
		\end{split} 
		\end{equation} 
	This, \eqref{continuity:convergence_of_hs}, the assumption that $V\in C(\mc O,(0,\infty))$, and the fact that $ (\mf u_n)_{n\in\N} \subseteq C(\mc O,\R) $ imply that $u$ is continuous. This establishes Item~\eqref{continuity:item1}. Next we prove Item~\eqref{continuity:item2}. For this we assume that $\sup_{r\in (0,\infty)} [\inf_{x\in\mc O\setminus O_r} V(x)] = \infty$. The assumption that $ V \in C(\mc O,(0,\infty)) $ hence assures that $ \{x\in\mc O\colon V(x) = \inf_{y\in\mc O} V(y) \} \neq \emptyset $. The assumption that for all $ t \in [0,\infty) $, $ x \in \mc O $ it holds that $ e^{-\rho t}\EXP{V(X^x_t)} \leq V(x) $ therefore implies for all $ t \in (0,\infty) $ that $ 0 < e^{-\rho t} \min_{y\in\mc O} V(y) \leq \min_{y\in\mc O} V(y) $. Hence, we obtain that $ 0 \leq \rho < \lambda $. 
	This, the fact that for every $ n \in \N $ it holds that $ \mf h_n \colon \mc O \to \R $ is bounded, the assumption that $\sup_{r\in (0,\infty)} [\inf_{x\in\mc O\setminus O_r} V(x)] = \infty$, and \eqref{continuity:definition_of_uns}
	demonstrate for all $ n \in \N $ that 
		\begin{equation} 
		\begin{split}
		\limsup_{r\to\infty} \left[ \sup_{x\in\mc O\setminus O_r} \left(\frac{|\mf u_n(x)|}{V(x)}\right) \right] 
		& \leq 
		\left[ \sup_{x\in\mc O} |\mf u_n(x)|\right] 
		\limsup_{r\to\infty} \left[  \sup_{x\in\mc O\setminus O_r} \left( \frac{1}{V(x)} \right) \right] 
		\\
		& \leq \frac{1}{\lambda} \left( \sup_{x\in\mc O} |\mf h_n(x)| \right)  \left( \limsup_{r\to\infty} \left[  \sup_{x\in\mc O\setminus O_r} \left( \frac{1}{V(x)} \right) \right] \right) = 0. 
		\end{split}
		\end{equation} 
	Combining this with \eqref{continuity:locally_uniform_convergence} and \eqref{continuity:convergence_of_hs} ensures that 
		\begin{equation} 
		\limsup_{r\to\infty} \left[ \sup_{x\in\mc O\setminus O_r} 
		\left( \frac{|u(x)|}{V(x)} \right) \right] 
		= 0. 
		\end{equation} 
	This establishes Item~\eqref{continuity:item2}. The proof of \cref{lem:continuity} is thus completed. 
\end{proof} 

\begin{cor}\label{mapping_welldefinedness} 
	Let $ d \in \N $, 
		$ L,\rho \in \R $, 
		$ \lambda \in (\rho,\infty) $, 
	let $ \norm{\cdot}\colon\R^d \to [0,\infty) $ be a norm on $\R^d$, 
	let $ \mc O \subseteq \R^d $ be a non-empty open set, 
	for every $ r \in (0,\infty) $ let $ O_r \subseteq \mc O $ satisfy 
		$ O_r = \{ x \in \cO\colon \norm{x}\leq r~\text{and}~\{y\in\R^d\colon \norm{y-x} < \nicefrac{1}{r}\} \subseteq \cO\} $, 
	let $ ( \Omega, \mc F, \P ) $ be a probability space, 
	for every $ x \in \mc O $ let 
		$ X^x = (X^x_t)_{t\in [0,\infty)}\colon [0,\infty) \times \Omega \to \mc O $ be $ (\Borel([0,\infty))\otimes \mc F)$/$\Borel(\mc O)$-measurable, 
	assume for all 
		$ \varepsilon,t \in (0,\infty) $ 
	and all $ x_n \in \mc O $, $ n \in \N_0 $, with $ \limsup_{n\to\infty} \norm{x_n-x_0} = 0 $ that $ \limsup_{n\to\infty} \P(\norm{X^{x_n}_t-X^{x_0}_t}\geq\varepsilon) = 0 $, 
	let $ f \in C( \mc O\times\R, \R ) $, 
		$ u \in C( \mc O, \R )$, 
		$ V \in C( \mc O, (0,\infty) ) $ 
	satisfy for all 
	 	$ t \in [0,\infty) $, 
		$ x \in \mc O $ 
	that 
		$ \EXP{e^{-\rho t}V(X^x_t)} \leq V(x) $,
	and assume for all 
		$ x \in \mc O $, 
		$ v,w \in \R $
	that 
		$ \inf_{r\in (0,\infty)} [\sup_{y\in\mc O\setminus O_r} (\frac{|f(y,0)|+|u(y)|}{V(y)})] = 0 $
	and 
		$ |f(x,v)-f(x,w)|\leq L|v-w| $. 
	Then 
	\begin{enumerate}[(i)]
		\item \label{mapping_welldefinedness:item1} it holds for all 
			$ x \in \mc O $ 
		that 
			$\EXP{\int_0^{\infty} e^{-\lambda t} |f(X^x_t,u(X^x_t))|\,dt}<\infty$, 
		\item \label{mapping_welldefinedness:item2} it holds that 
			\begin{equation} 
			\mc O \ni x \mapsto \Exp{\int_0^{\infty} e^{-\lambda t} f(X^x_t,u(X^x_t))\,dt} \in \R
			\end{equation} 
		is continuous, and 
		\item \label{mapping_welldefinedness:item3} it holds in the case of $ \sup_{r\in (0,\infty)}[\inf_{x\in\mc O\setminus O_r} V(x)] = \infty $ that 
		\begin{equation} 
		\lim_{r\to\infty} \left[ 
		\sup_{x\in\mc O\setminus O_r} \left(  
		\frac{|\Exp{\int_0^{\infty} e^{-\lambda t} f(X^x_t,u(X^x_t))\,dt}|}{V(x)}
		\right) 
		\right] = 0.
	\end{equation}
	\end{enumerate}
\end{cor}

\begin{proof}[Proof of \cref{mapping_welldefinedness}] 
	First, observe that $ \mc O \ni x \mapsto f(x,u(x)) \in \R $ is a continuous function which satisfies for all $ x \in \mc O $ that 
		\begin{equation}
		|f(x,u(x))| 
		\leq |f(x,0)| + |f(x,u(x))-f(x,0)| 
		\leq |f(x,0)| + L|u(x)|. 
		\end{equation} 
	The hypothesis that $\inf_{r\in (0,\infty)} [\sup_{x\in\mc O\setminus O_r}( \frac{|f(x,0)|+|u(x)|}{V(x)})] = 0$ hence ensures that 
		\begin{equation} \label{mapping_welldefinedness:growth_condition_check}
		\inf_{r\in (0,\infty)} 
		\left[ 
		\sup_{x\in\mc O\setminus O_r} 
		\left(  
		\frac{|f(x,u(x))|}{V(x)} 
		\right) 
		\right] 
		\leq 
		\inf_{r\in (0,\infty)} 
		\left[ 
		\sup_{x\in\mc O\setminus O_r} 
		\left( 
		\frac{|f(x,0)|}{V(x)} 
		+ 
		L \frac{|u(x)|}{V(x)}
		\right) 
		\right] = 0. 
		\end{equation} 
	This, the hypothesis that 
		$ f \in C(\mc O\times\R,\R) $, 
	the hypothesis that	
		$ u \in C(\mc O,\R)$, 
	and the hypothesis that
		$ V \in C(\mc O,(0,\infty)) $
	imply that 
		$ \sup_{x\in\mc O} (\frac{|f(x,u(x))|}{V(x)}) < \infty $.  
	\cref{lem:integrability} (with 
		$ d \is d $, 
		$ c \is \sup_{y\in\mc O} (\frac{|f(y,u(y))|}{V(y)}) $,
		$ \gamma \is V(x) $, 
		$ \rho \is \rho $, 
		$ \lambda \is \lambda $, 
		$ \mc O \is \mc O $, 
		$ ( \Omega, \mc F, \P ) \is ( \Omega, \mc F, \P ) $, 
		$ X \is X^x $, 
		$ h \is ( \mc O \ni y \mapsto f(y,u(y)) \in \R ) $ 
	for 
		$ x \in \mc O $
	in the notation of \cref{lem:integrability}) therefore establishes Item~\eqref{mapping_welldefinedness:item1}. In addition, note that \eqref{mapping_welldefinedness:growth_condition_check} and \cref{lem:continuity} (with 
		$ d \is d$, 
		$ \rho \is \rho$, 
		$ \lambda \is \lambda$, 
		$ \norm{\cdot} \is \norm{\cdot}$, 
		$ \mc O \is \mc O$, 
		$ (\Omega,\mc F,\P) \is (\Omega,\mc F,\P)$, 
		$ (X^x)_{ x \in \mc O } \is (X^x)_{ x \in \mc O }$, 
		$ h \is (\mc O \ni y \mapsto f(y,u(y)) \in \R)$, 
		$ V \is V$ 
	in the notation of \cref{lem:continuity}) establish Items~\eqref{mapping_welldefinedness:item2} and \eqref{mapping_welldefinedness:item3}. The proof of \cref{mapping_welldefinedness} is thus completed.  
\end{proof} 

\begin{lemma} \label{lipschitz_property}
	Let 
		$d\in\N$, 
		$L,\rho\in\R$,
	let 
		$\cO \subseteq  \R^d$ 
	be a non-empty open set, 
	let 
		$(\Omega,\mathcal{F},\P)$ 
	be a probability space, 
	for every 
		$x\in\cO$	
	let 
		$
		X^{x}=(X^{x}_t)_{t\in [0,\infty)}\colon [0,\infty)\times\Omega\to\cO$ 
	be 
		$(\Borel([0,\infty)) \otimes \mathcal{F})$/$\Borel(\cO)$-measurable, 
	let 
		$V\colon \cO\to (0,\infty) $ 
	be 
		$\Borel(\cO)$/$\Borel((0,\infty))$-measurable, 
	assume for all 
		$t\in [0,\infty)$, 
		$x\in\cO$ 
	that 
		$
		\Exp{e^{-\rho t}V(X^{x}_{t})}
		\leq 
		V(x)
		$, 
	let 
		$f\colon \cO\times\R\to\R$ 
	be  		
		$\Borel(\cO\times\R)$/$\Borel(\R)$-measurable, 
	assume for all 
		$x\in\cO$, 
		$v,w\in \R$ 
	that 
		$|f(x,v)-f(x,w)| \leq L | v -  w |$, 
	let 
		$v,w\colon \cO\to\R$ be 
		$\Borel(\cO)$/$\Borel(\R)$-measurable, 
	and assume that
		\begin{equation}\label{lipschitz_property:growth_assumption}
		\sup_{x\in \cO} 
		\left[
		\frac{ |v(x)| + |w(x)| }{V(x)} 
		\right] 
		< 
		\infty. 
		\end{equation}
	Then it holds for all 
		$\lambda \in (\rho,\infty)$, 
		$x\in\cO$ 
	that 
		\begin{equation}
		\label{lipschitz_property:claim}
		\begin{split}
		& 
		\frac{1}{V(x)}
		\Exp{
			\int_0^{\infty} e^{-\lambda s}\left| 
			f\big( X^{x}_{s}, v(X^{x}_{s}) \big) - 
			f\big( X^{x}_{s}, w(X^{x}_{s}) \big) 
			\right| \,ds
		}
		\leq 
		\frac{L}{\lambda-\rho} 
		\left[ 
		\sup_{y\in \cO} 
		\left( 
		\frac{|v(y) - w(y)|}{V(y)}
		\right) 
		\right]. 
		\end{split}
		\end{equation}
\end{lemma}

\begin{proof}[Proof of \cref{lipschitz_property}]
	First, note that the fact that 
	$f\colon \cO\times\R\to\R$ 
	is $\Borel(\cO\times\R)$/$\Borel(\R)$-measurable, 
	the fact that 
	$v,w\colon \cO\to\R$ 
	are $\Borel(\cO)$/$\Borel(\R)$-measurable, 
	and the fact that for all 
	$x\in\cO$ 
	it holds that 
	$X^{x} \colon [0,\infty)\times\Omega\to\cO$ 
	is $(\Borel([0,\infty))\otimes\cF)$/$\Borel(\cO)$-measurable ensure that for all 
	$x\in\cO$ 
	it holds that
	\begin{equation}
	[0,\infty) \times \Omega 
	\ni (t,\omega)  
	\mapsto 
	\left| 
	f\big( 
	X^{x}_{t}(\omega),
	v(X^{x}_{t}(\omega))
	\big) 
	- 
	f\big( 
	X^{x}_{t}(\omega),
	w(X^{x}_{t}(\omega))
	\big) 
	\right|\in\R
	\end{equation}
	is $(\Borel([0,\infty)) \otimes \mathcal{F})$/$\Borel(\R)$-measurable. Next observe that the hypothesis that for all 
		$t\in [0,\infty)$, 
		$x\in \cO$ 
	it holds that 
		$\EXP{e^{-\rho t}V(X^{x}_t)} \leq V(x)$, 
	the hypothesis that for all 
		$x\in \cO$, 
		$a,b\in\R$ 
	it holds that 
		$
		|f(x,a)-f(x,b)| \leq L|a-b|
		$, 
	and Fubini's theorem ensure that for all
		$\lambda\in (\rho,\infty)$, 
		$x\in\cO$
	it holds that 
		\begin{equation}
		\begin{split}
		& 
		\frac{1}{V(x)}\Exp{ 
			\int_0^{\infty} e^{-\lambda t}\left| 
			f\big( X^{x}_{t}, v(X^{x}_{t}) \big) 
			- 
			f\big( X^{x}_{t}, w(X^{x}_{t}) \big) 
			\right| \,dt
		}
		\\
		& \leq 
		\Exp{ 
			\int_0^\infty 
			Le^{-\lambda t} \frac{\left| v(X^{x}_{t}) - w(X^{x}_{t})\right|}{V(x)} \,dt 
		}
		= 
		L 
		\int_0^{\infty} e^{-\lambda t}\, \Exp{
			\left( \frac{ |v(X^{x}_{t})-w(X^{x}_{t})| }{ V(X^{x}_{t}) } \right) \!
			\left( \frac{ V(X^{x}_{t}) }{ V(x) } \right)
		}\!\,dt
		\\
		& 
		\leq 
		L
		\int_0^{\infty} 
		e^{-\lambda t}
		\left[ 
		\sup_{y\in\cO} 
			\left( 
		\frac{|v(y)-w(y)|}{V(y)}
			\right) 
		\right]
		\left[\frac{\Exp{V(X^{x}_{t})}}{V(x)}\right]\!\,dt
		\\
		& 
		\leq 
		L \left[ 
		\sup_{y\in\cO} 
			\left( 
		\frac{|v(y)-w(y)|}{V(y)}
			\right) 
		\right]
		\left[ 
		\int_0^{\infty} 
		e^{-(\lambda-\rho) t}
		\,dt \right]
		= \frac{L}{\lambda-\rho} 
		\left[ 
		\sup_{y\in\cO} 
		\left( 
		\frac{|v(y)-w(y)|}{V(y)}
		\right) 
		\right]\!. 
		\end{split}
		\end{equation}
	This establishes \eqref{lipschitz_property:claim}. 
	The proof of \cref{lipschitz_property} is thus completed.
\end{proof}

\begin{prop} \label{prop:existence_of_abstract_fixpoint} 
	Let $ d \in \N $, $ L, \rho \in \R $, $ \lambda \in (L+\rho,\infty) $,  
	let $\norm{\cdot}\colon \R^d \to [0,\infty) $ be a norm on $\R^d$, 
	let $ \mc O \subseteq \R^d $ be a non-empty open set, 
	for every $ r \in (0,\infty) $ let $ O_r \subseteq \mc O $ satisfy $ O_r = \{x \in \mc O \colon \norm{x}\leq r~\text{and}~\{y\in\R^d\colon \norm{y-x}<\nicefrac{1}{r}\}\subseteq\mc O\} $, 
	let $ ( \Omega, \mc F, \mb P ) $ be a probability space, 
	for every $ x \in \mc O $ let $ X^x = (X^x_t)_{t\in [0,\infty)}\colon [0,\infty) \times \Omega \to \mc O $ 
	be $(\Borel([0,\infty))\otimes\mc F)$/$\Borel(\mc O)$-measurable, 
	assume for all 
		$ \varepsilon,t \in (0,\infty) $ 
	and all 
		$ x_n \in \mc O$, $ n \in \N_0 $, 
	with 
		$ \limsup_{n\to\infty} \norm{x_n-x_0} = 0 $ 
	that 
		$ \limsup_{n\to\infty} \P(\norm{X^{x_n}_t-X^{x_0}_t}\geq\varepsilon) = 0$, 
	let $ V \in C( \mc O, (0,\infty) ) $, $ f \in C( \mc O \times \R, \R) $ satisfy for all 
		$ t\in [0,\infty) $, 
		$ x \in \mc O $, 
		$ v,w \in \R $
	that 
		$ \EXP{e^{-\rho t} V(X^x_t)} \leq V(x) $ 
	and 
		$ |f(x,v) - f(x,w)| \leq L|v-w| $, 
	and assume that 
		$ \inf_{r\in (0,\infty)} [\sup_{x\in\mc O\setminus O_r} (\frac{|f(x,0)|}{V(x)})] = 0 $ 
	and 
		$ \sup_{r \in (0,\infty) } [\inf_{x\in\mc O\setminus O_r} V(x)] = \infty $. 
	Then there exists a unique $u \in C(\mc O,\R)$ such that 
	\begin{enumerate}[(i)] 
		\item\label{existence_of_abstract_fixpoint:item1} it holds that 
		\begin{equation} 
		\limsup_{r\to\infty} \left[ \sup_{x\in \mc O \setminus O_r} \left( \frac{|u(x)|}{V(x)} \right) \right] = 0  
		\end{equation} 
		and 
		\item\label{existence_of_abstract_fixpoint:item2} it holds for all $ x \in \mc O $ that 
		$ \EXPP{ \int_0^{\infty} e^{-\lambda t} |f(X^x_t,u(X^x_t))|\,dt} < \infty $  and 
		\begin{equation}
		u(x) = \Exp{\int_0^{\infty} e^{-\lambda t} f(X^x_t,u(X^x_t))\,dt}\!. 
		\end{equation}
	\end{enumerate}
\end{prop} 

\begin{proof}[Proof of \cref{prop:existence_of_abstract_fixpoint}] 
	Throughout this proof let $\cV$	be the set given by 
		\begin{equation} 
		\cV = \left\{ u \in C(\mc O,\R)\colon \limsup_{r\to\infty} \left[ \sup_{x\in\cO\setminus O_r} \left(\frac{|u(x)|}{V(x)}\right) \right] = 0 \right\}
		\end{equation} 
	and let $\norm{\cdot}_{\cV}\colon\cV\to [0,\infty)$ 	satisfy for all 
		$ v \in \cV $ 
	that 
		\begin{equation} 
		\norm{v}_{\cV} = \sup_{x\in\cO} \left( \frac{|v(x)|}{V(x)} \right)\!. 
		\end{equation} 
	Note that 
	$(\cV,\norm{\cdot}_{\cV})$ 
	is an $\R$-Banach space. Moreover, observe that \cref{mapping_welldefinedness} (with 
		$ d \is d $,  
		$ L \is L $, 
		$ \rho \is \rho $, 
		$ \lambda \is \lambda $, 
		$ \norm{\cdot} \is \norm{\cdot} $, 
		$ \mc O \is \mc O $, 
		$ (\Omega,\mc F,\P) \is (\Omega,\mc F,\P) $, 
		$ (X^x)_{ x \in \mc O } \is (X^x)_{ x \in \mc O } $ ,
		$ f \is f $, 
		$ V \is V $, 
		$ u \is v $ 
	for $ v \in \mc V $
	in the notation of \cref{mapping_welldefinedness}) demonstrates that there exists a unique function 
	$\Phi\colon\cV\to\cV$ 
	which satisfies for all 
		$ x \in \mc O $, 
		$ v \in \cV $ 
	that 
		\begin{equation} 
		[\Phi(v)](x) 
		= 
		\Exp{\int_0^{\infty} e^{-\lambda t} f(X^x_t,v(X^x_t))\,dt}\!. 
		\end{equation} 
	Moreover, note that \cref{lipschitz_property} (with 
		$ d \is d $, 
		$ L \is L $, 
		$ \rho \is \rho $, 
		$ \mc O \is \mc O $, 
		$ (\Omega, \mc F, \P) \is (\Omega,\mc F,\P) $, 
		$ (X^x)_{ x \in \mc O } \is (X^x)_{ x \in \mc O } $, 
		$ V \is V $, 
		$ f \is f $
	in the notation of \cref{lipschitz_property}) ensures for all 
		$v,w\in\cV$
	that 
	\begin{equation} 
	\norm{\Phi(v)-\Phi(w)}_{\cV} \leq \frac{L}{\lambda - \rho} \norm{v-w}_{\cV}. 
	\end{equation} 
	This, the hypothesis that $\lambda>L+\rho$, the fact that $(\cV,\norm{\cdot}_{\cV})$ is an $\R$-Banach space, and Banach's fixed point theorem demonstrate that there exists a unique $u\in\cV$ which satisfies $\Phi(u)=u$. This establishes Items~\eqref{existence_of_abstract_fixpoint:item1} and \eqref{existence_of_abstract_fixpoint:item2}. The proof of \cref{prop:existence_of_abstract_fixpoint} is thus completed. 
\end{proof}

\subsection{On the Feynman--Kac connection between SFPEs and  semilinear elliptic PDEs}
\label{subsec:sfpes_and_pdes}

\begin{cor} \label{cor:existence_of_fixpoint}
Let $d,m\in\N$, 
	$ B \in \R^{d\times m} $, 
	$ L, \rho \in \R $, 
	$ \lambda \in (\rho+L,\infty) $, 
let $\norm{\cdot}\colon\R^d\to [0,\infty)$ be a norm on $\R^d$, 
let 
	$f\in C(\R^d\times\R,\R)$, 
	$V\in C^2(\R^d,(0,\infty))$	
satisfy for all 
	$x\in\R^d$, 
	$v,w\in\R$ 
that 
	$|f(x,v)-f(x,w)|\leq L|v-w|$
and 
	\begin{equation} \label{existence_of_fixpoint:lyapunov_condition}
	 \tfrac12\operatorname{Trace}(  
	 BB^{*}(\operatorname{Hess} V)(x)
	 ) 
	 \leq 
	 \rho V(x), 
	\end{equation} 
assume that
	$
	\limsup_{r\to\infty} 
	[\sup_{\norm{x}>r} (\frac{|f(x,0)|}{V(x)})] 
	= 0
	$
and 
	$\sup_{r\in (0,\infty)}
	[\inf_{\norm{x}>r} V(x)] 
	= \infty$,
let 
	$(\Omega,\cF,\P)$ 
be a probability space, 
and let 
	$W\colon [0,\infty)\times\Omega\to\R^m$ 
be a standard Brownian motion. 
Then there exists a unique 
	$u\in C(\R^d,\R)$ 
such that 
\begin{enumerate}[(i)] 
	\item \label{existence_of_fixpoint:item1} 
	it holds that 
		\begin{equation}
		\limsup_{r\to\infty} 
		\left[\sup_{\norm{x}>r} \left(\frac{|u(x)|}{V(x)}\right)\right] = 0 
		\end{equation} 
	and 
	\item\label{existence_of_fixpoint:item2} for all $ x \in \R^d $ it holds that  
		$ \EXPP{\int_0^{\infty} e^{-\lambda s} |f(x+BW_s,u(x+BW_s))|\,ds} < \infty $
	and 
		\begin{equation}\label{existence_of_fixpoint:claim}
		u(x) = \Exp{\int_0^{\infty} e^{-\lambda s} f(x+BW_s,u(x+BW_s))\,ds}\!.
		\end{equation} 
\end{enumerate}
\end{cor}

\begin{proof}[Proof of \cref{cor:existence_of_fixpoint}] 
	First, observe that combining 
		\cite[Lemmas 3.1 and 3.2]{StochasticFixedPointEquations}  
	with \eqref{existence_of_fixpoint:lyapunov_condition}
	ensures that for all 
		$ t\in [0,\infty) $, 
		$ x \in \R^d $ 
	it holds that 
		\begin{equation} \label{existence_of_fixpoint:supermartingale_type_control}
		\Exp{e^{-\rho t}V(x+BW_t)} \leq V(x). 
		\end{equation} 
	Moreover, note that for all 
		$ \varepsilon, t \in (0,\infty) $ 
	and all 
		$ x_n \in \R^d $, $ n \in \N_0 $, 
	with 
		$ \limsup_{ n \to \infty } \norm{ x_n - x_0 } = 0 $ 
	it holds that 
		$ \limsup_{ n \to \infty } \P(\norm{(x_n+BW_t)-(x_0+BW_t)} \geq \varepsilon) = 0 $. 
	\cref{prop:existence_of_abstract_fixpoint} (with 
		$ d \is d $, 
		$ L \is L $, 
		$ \rho \is \rho $, 
		$ \lambda \is \lambda $, 
		$ \norm{\cdot} \is \norm{\cdot} $,  
		$ \mc O \is \R^d $, 
		$ (\Omega,\mc F, \P) \is (\Omega,\mc F,\P) $, 
		$ (X^x)_{x\in\R^d} \is ([0,\infty)\times\Omega\ni (t,\omega) \mapsto x + BW_t(\omega) \in \R^d)_{x\in\R^d} $, 
		$ V \is V $, 
		$ f \is f $ 
	in the notation of \cref{prop:existence_of_abstract_fixpoint}) and \eqref{existence_of_fixpoint:supermartingale_type_control} hence assure that there exists a unique $ u \in C(\R^d,\R) $ 
	such that  
		\begin{enumerate}[(I)] 
			\item it holds that 
			\begin{equation} 
				\limsup_{ r \to \infty } \left[ \sup_{\norm{x}>r} \left( \frac{|u(x)|}{V(x)} \right) \right] 
				= 0
			\end{equation} 
			and 
			\item for all $ x \in \R^d $ it holds that 
				$ \EXPP{\int_0^{\infty} e^{-\lambda s} |f(x+BW_s,u(x+BW_s))|\,ds} < \infty $ 
			and 
			\begin{equation}
				u(x) = \Exp{\int_0^{\infty} e^{-\lambda t} f(x+BW_t,u(x+BW_t))\,dt}\!.
			\end{equation}
		\end{enumerate}
	This establishes Items~\eqref{existence_of_fixpoint:item1} and \eqref{existence_of_fixpoint:item2}. The proof of \cref{cor:existence_of_fixpoint} is thus completed. 
\end{proof} 

\begin{lemma} \label{lem:classical_solution_smooth_case}
	Let $ T \in (0,\infty) $, 
		$ d,m \in \N $, 
		$ B \in \R^{d\times m} $, 
		$ \varphi \in C^{2}(\R^d,\R) $ 
	satisfy  
		$ \sup_{x\in\R^d} 
		[ 
		\sum_{i,j=1}^d 
		( | \varphi(x) | + | (\frac{\partial }{\partial x_i}\varphi)(x) | + |(\frac{\partial^2}{\partial x_i\partial x_j} \varphi)(x)| ) ] < \infty $, 
	let $ (\Omega,\mc F,\P) $ be a probability space, 
	let $ Z \colon \Omega \to \R^m $ be a standard normal random variable, 
	and let $ u \colon [0,T] \times \R^d \to \R $ satisfy for all 
		$ t \in [0,T] $, 
		$ x \in \R^d $ 
	that 
		\begin{equation} \label{classical_solution_smooth_case:solution_formula}
		u(t,x) = \EXPP{\varphi(x+\sqrt{t}B Z)}. 
		\end{equation} 
	Then  
	\begin{enumerate} [(i)]
		\item\label{classical_solution_smooth_case:item1}
		it holds that 
			$ u \in C^{1,2}([0,T] \times \R^d, \R) $ 
		and 
		\item\label{classical_solution_smooth_case:item2}
		it holds for all 
			$ t \in [0,T] $, 
			$ x \in \R^d $ 
		that 
			\begin{equation} \label{classical_solution_smooth_case:claim}
			(\tfrac{\partial u}{\partial t})(t,x) = \tfrac12 \operatorname{Trace}\!\big(BB^{*}(\operatorname{Hess}_x u)(t,x)\big).  
			\end{equation} 
	\end{enumerate}
\end{lemma}

\begin{proof}[Proof of \cref{lem:classical_solution_smooth_case}]
	Throughout this proof 
	let $ e_1 = ( 1, 0, \ldots, 0), e_2 = ( 0, 1, \ldots, 0), \ldots, e_m = (0, \ldots, 0, 1) \in \R^m $, 
	let $ \langle\cdot,\cdot\rangle \colon (\cup_{k\in\N} (\R^k\times\R^k) ) \to \R $ satisfy for all $ k \in \N $, $ x = (x_1,x_2,\ldots,x_k), y = (y_1,y_2,\ldots,y_k) \in \R^k $ that $ \langle x,y \rangle = \sum_{i=1}^k x_iy_i $, 
	let $ \norm{\cdot}\colon \R^m \to [0,\infty) $ be the standard norm on $ \R^m $, 
	and let 
		$ \psi_{t,x}=(\psi_{t,x}(y))_{y\in\R^m} \colon \R^m \to \R $, $ t\in [0,T] $, $ x \in \R^d $, 
	satisfy for all 
		$ t \in [0,T] $, 
		$ x \in \R^d $, 
		$ y \in \R^m $
	that 
		$ \psi_{t,x}(y) = \varphi(x+\sqrt{t}By) $. 
	Note that the assumption that $ \varphi \in C^2(\R^d,\R) $, the assumption that $ \sup_{x\in\R^d} [\sum\nolimits_{i,j=1}^d (|\varphi(x)| + |(\frac{\partial}{\partial x_i}\varphi)(x)| + | (\frac{\partial^2}{\partial x_i\partial x_j}\varphi)(x) |)] < \infty $, 
	the chain rule, and Lebesgue's dominated convergence theorem ensure that 
	\begin{enumerate}[(I)] 
		\item \label{classical_solution_smooth_case:proof_item1} 
		for all 
			$ x \in \R^d $ 
		it holds that 
			$ (0,T] \ni t \mapsto u(t,x) \in \R $ 
		is differentiable, 
		\item \label{classical_solution_smooth_case:proof_item2} 
		for all  
			$ t \in [0,T] $
		it holds that 
			$ \R^d \ni x \mapsto u(t,x) \in \R $ is twice differentiable, 		
		\item \label{classical_solution_smooth_case:proof_item3} 
		for all 
			$ t \in (0,T] $, 
			$ x \in \R^d $ 
		it holds that 
			\begin{equation} 
			(\tfrac{\partial u}{\partial t})(t,x) 
			= 
			\EXPP{\langle (\nabla\varphi)(x+\sqrt{t}BZ),\tfrac{1}{2\sqrt{t}}BZ\rangle}, 
			\end{equation} 
		and 
		\item \label{classical_solution_smooth_case:proof_item4} 
		for all 
			$ t \in [0,T] $, 
			$ x \in \R^d $
		it holds that 
			\begin{equation} 
			(\operatorname{Hess}_x u)(t,x) 
			= 
			\EXPP{(\operatorname{Hess} \varphi)(x+\sqrt{t}BZ) }.   
			\end{equation} 
		\end{enumerate}
	Observe that Items~\eqref{classical_solution_smooth_case:proof_item3} and \eqref{classical_solution_smooth_case:proof_item4}, the assumption that $ \varphi \in C^2(\R^d,\R) $, the assumption that $ \sup_{x\in\R^d} [\sum\nolimits_{i,j=1}^d (|\varphi(x)| + |(\frac{\partial}{\partial x_i}\varphi)(x)| + | (\frac{\partial^2}{\partial x_i\partial x_j}\varphi)(x) |)] < \infty $, the fact that $ \EXP{\Norm{Z}} < \infty $, and Lebesgue's dominated convergence theorem prove that $ (0,T] \times \R^d \ni (t,x) \mapsto (\frac{\partial u}{\partial t})(t,x) \in \R $ and $ [0,T] \times \R^d \ni (t,x) \mapsto (\operatorname{Hess}_x u)(t,x) \in \R^{d\times d} $ are continuous. 
	Next note that Item~\eqref{classical_solution_smooth_case:proof_item4} and the fact that for all 
		$ X \in \R^{m \times d} $, 
		$ Y \in \R^{d \times m} $ 
	it holds that 
		$ \operatorname{Trace}(X Y) = \operatorname{Trace}(Y X) $	
	imply that for all 
		$ t \in (0,T] $, 
		$ x \in \R^d $ 
	it holds that 
		\begin{equation} 
		\begin{split}
		& 
		\tfrac12 \operatorname{Trace}\!\big(B B^{*} (\operatorname{Hess}_x u)(t,x)\big) 
		= 
		\Exp{\tfrac12\operatorname{Trace}\!\big( B B^{*} (\operatorname{Hess} \varphi)(x+\sqrt{t}BZ) \big)}
		\\
		& = 
		\tfrac{1}{2} \, \Exp{\operatorname{Trace}\!\big( B^{*} (\operatorname{Hess} \varphi)(x+\sqrt{t}BZ) B \big) } 
		= 
		\tfrac{1}{2} \, \Exp{\smallsum\limits_{k=1}^m \displaystyle\langle e_k, B^{*}  (\operatorname{Hess} \varphi)(x+\sqrt{t}BZ) B e_k \rangle }
		\\		
		& = 
		\tfrac{1}{2} \, \Exp{\smallsum\limits_{k=1}^m \displaystyle\langle B e_k ,  (\operatorname{Hess} \varphi)(x+\sqrt{t}BZ) B e_k \rangle }
		= 
		\tfrac{1}{2} \, \Exp{\smallsum\limits_{k=1}^m \displaystyle\varphi^{\prime\prime}( x + \sqrt{t}BZ )(B e_k, B e_k) }
		\\
		& = 
		 \tfrac{1}{2t} \, \Exp{\smallsum\limits_{k=1}^m \displaystyle(\psi_{t,x})^{\prime\prime}(Z)(e_k,e_k) } 
		= 
		\tfrac{1}{2t} \, \Exp{\smallsum\limits_{k=1}^m \displaystyle(\tfrac{\partial^2 }{\partial y_k^2}\psi_{t,x})(Z) }
		= 
		\tfrac{1}{2t} \,
		\Exp{ (\Delta \psi_{t,x})(Z)}\!.
		\end{split}
		\end{equation} 
	The assumption that $ Z \colon \Omega \to \R^m $ is a standard normal random variable and integration by parts hence ensure that for all 
		$ t \in (0,T] $, 
		$ x \in \R^d $ 
	it holds that  
		\begin{equation} 
		\begin{split} 
		& \tfrac12 \operatorname{Trace}\!\big(B B^{*} (\operatorname{Hess}_x u)(t,x)\big) 
		\\
		& = 
		\frac{1}{2t} 
		\int_{\R^m} (\Delta\psi_{t,x})(y)  \left[ \frac{\exp(-\tfrac{\langle y,y\rangle}{2})}{(2\pi)^{\nicefrac{m}{2}}}\right]\!\,dy 
		= 
		\frac{1}{2t} 
		\int_{\R^m} \langle (\nabla\psi_{t,x})(y), y \rangle \left[ \frac{\exp(-\frac{\langle y,y \rangle}{2})}{(2\pi)^{\nicefrac{m}{2}}} \right]\!\,dy 
		\\
		& = 
		\frac{1}{2\sqrt{t}} 
		\int_{\R^m} \left\langle B^{*}(\nabla \varphi)(x+\sqrt{t}By), y \right\rangle \left[\frac{\exp(-\frac{\langle y, y\rangle}{2})}{(2\pi)^{\nicefrac{m}{2}}}\right]\!\,dy 
		\\
		& =
		\frac{1}{2\sqrt{t}} \, 
		\EXPP{ \langle B^{*}(\nabla\varphi)(x+\sqrt{t}BZ),Z \rangle }
		= 
		\EXPP{ \langle (\nabla\varphi)(x+\sqrt{t}BZ), \tfrac{1}{2\sqrt{t}} B Z\rangle }.
		\end{split}
		\end{equation} 
	Item~\eqref{classical_solution_smooth_case:proof_item3} therefore proves for all 
		$ t \in (0,T] $, 
		$ x \in \R^d $ 
	that 
		\begin{equation} \label{classical_solution_smooth_case:t_positive}
		(\tfrac{\partial u}{\partial t})(t,x) = \tfrac12 \operatorname{Trace}\!\big(B B^{*} (\operatorname{Hess}_x u)(t,x)\big). 
		\end{equation} 
	The fundamental theorem of calculus hence implies that for all 
		$ t,s \in (0,T] $, 
		$ x \in \R^d $ 
	it holds that 
		\begin{equation} 
		\begin{split}
		u(t,x) - u(s,x) 
		& = 
		\int_{s}^t (\tfrac{\partial u}{\partial t})(r,x) \,dr 
		= \int_{s}^t \tfrac12 \operatorname{Trace}\!\big(BB^{*}(\operatorname{Hess}_x u)(r,x)\big) \,dr . 
		\end{split}
		\end{equation} 
	The fact that $ [0,T] \times \R^d \ni (t,x) \mapsto (\operatorname{Hess}_x u)(t,x) \in \R^{d\times d} $ 
	is continuous therefore ensures for all 
		$ t \in (0,T] $, 
		$ x \in \R^d $ 
	that 	
		\begin{equation} 
		\frac{u(t,x) - u(0,x)}{t} 
		= 
		\lim_{ s \searrow 0 } 
		\left[ \frac{u(t,x)-u(s,x)}{t} \right] 
		= 
		\frac{1}{t} 
		\int_0^t \tfrac12 \operatorname{Trace}\!\big( B B^{*} (\operatorname{Hess}_x u)(r,x) \big) \,dr. 
		\end{equation} 
	This and again the fact that 
		$ [0,T]\times\R^d \ni (t,x) \mapsto (\operatorname{Hess}_x u)(t,x) \in \R^{d\times d} $ 
	is continuous imply for all 
		$ x \in \R^d $ 
	that 	
		\begin{equation} \label{classical_solution_smooth_case:t_0}
		\begin{split}
		& \limsup_{ t \searrow 0 } \left| \frac{u(t,x) - u(0,x)}{t} - \tfrac12 \operatorname{Trace}\!\big( B B^{*} (\operatorname{Hess}_x u)(0,x) \big) \right| 
		\\
		& 
		\leq  
		\limsup_{ t \searrow 0 } \left[ \frac{1}{t} \int_0^t \left| \tfrac12 \operatorname{Trace}\!\big( B B^{*} (\operatorname{Hess}_x u)(s,x) \big) - \tfrac12 \operatorname{Trace}\!\big( B B^{*} (\operatorname{Hess}_x u)(0,x) \big) \right| \,ds \right]
		\\
		& 
		\leq 
		\limsup_{ t \searrow 0 } \left[ \sup_{s\in [0,t]} \left| \tfrac12 \operatorname{Trace}\!\Big( B B^{*} \big((\operatorname{Hess}_x u)(s,x) - (\operatorname{Hess}_x u)(0,x) \big) \Big) \right| \right] 
		= 0. 
		\end{split}
		\end{equation} 
	Item~\eqref{classical_solution_smooth_case:proof_item1} hence establishes that for all $ x \in \R^d $ it holds that $ [0,T] \ni t \mapsto u(t,x) \in \R $ is differentiable. Combining this with \eqref{classical_solution_smooth_case:t_0} and \eqref{classical_solution_smooth_case:t_positive} ensures that for all 
		$ t \in [0,T] $, 
		$ x \in \R^d $ 
	it holds that 
		\begin{equation} \label{classical_solution_smooth_case:end_of_proof}
		(\tfrac{\partial u}{\partial t})(t,x) = \tfrac12 \operatorname{Trace}\!\big( B B^{*} (\operatorname{Hess}_x u)(t,x) \big). 
		\end{equation}
	This and the fact that $ [0,T] \times \R^d \ni (t,x) \mapsto (\operatorname{Hess}_x u)(t,x) \in \R^{d\times d} $ is continuous establish Item~\eqref{classical_solution_smooth_case:item1}. In addition, note that \eqref{classical_solution_smooth_case:end_of_proof} establishes Item~\eqref{classical_solution_smooth_case:item2}. The proof of \cref{lem:classical_solution_smooth_case} is thus completed. 
\end{proof} 

\begin{cor} \label{cor:classical_solution_brownian_motion}
	Let $ T \in (0,\infty) $, 
		$ d,m \in \N $, 
		$ B \in \R^{d\times m} $, 
		$ \varphi \in C^{2}(\R^d,\R) $ 
	satisfy  
		$ \sup_{x\in\R^d}
		[ 
		\sum_{i,j=1}^d
		( | \varphi(x) | + | (\frac{ \partial }{\partial x_i}\varphi)(x) | + |(\frac{\partial^2}{\partial x_i\partial x_j} \varphi)(x)| ) ] < \infty $, 
	let $ (\Omega,\mc F,\P) $ be a probability space, 
	let $ W \colon [0,T] \times \Omega \to \R^m $ be a standard Brownian motion, 
	and let $ u \colon [0,T] \times \R^d \to \R $ satisfy for all 
		$ t \in [0,T] $, 
		$ x \in \R^d $ 
	that 
		\begin{equation} 
		u(t,x) = \EXPP{\varphi(x+BW_t)}. 
		\end{equation} 
	Then  
	\begin{enumerate} [(i)]
		\item \label{classical_solution_brownian_motion:item1}
		it holds that 
			$ u \in C^{1,2}([0,T] \times \R^d, \R) $ 
		and 
		\item \label{classical_solution_brownian_motion:item2}
		it holds for all 
			$ t \in [0,T] $, 
			$ x \in \R^d $ 
		that 
			\begin{equation} 
			(\tfrac{\partial u}{\partial t})(t,x) = \tfrac12 \operatorname{Trace}\!\big(BB^{*}(\operatorname{Hess}_x u)(t,x)\big).  
			\end{equation} 
	\end{enumerate}
\end{cor} 

\begin{proof}[Proof of \cref{cor:classical_solution_brownian_motion}]
	First, observe that the assumption that $ W\colon [0,T]\times\Omega\to\R^m $ is a standard Brownian motion ensures  
	for all 
		$ t \in [0,T] $, 
		$ x \in \R^d $
	that 
		\begin{equation} 
		u(t,x) = \Exp{ \varphi(x+BW_t) } = \Exp{ \varphi(x+\sqrt{t}B \frac{W_T}{\sqrt{T}}) }\!.
		\end{equation} 
	The fact that $ \frac{W_T}{\sqrt{T}} \colon \Omega \to \R^m $ is standard normally distributed and \cref{lem:classical_solution_smooth_case} hence establish Items~\eqref{classical_solution_brownian_motion:item1} and \eqref{classical_solution_brownian_motion:item2}. The proof of \cref{cor:classical_solution_brownian_motion} is thus completed. 
\end{proof} 

\begin{lemma} \label{lem:viscosity_solution_inhomogeneous_elliptic_equation}
	Let $ d,m \in \N $,
		$ B \in \R^{d\times m} $, 
		$ \lambda \in (0,\infty) $, 
		$ \rho \in (-\infty,\lambda) $, 
	let $ \norm{\cdot} \colon \R^d \to [0,\infty) $ be a norm on $ \R^d $, 
	let $ h \in C( \R^d, \R ) $, 
		$ V \in C^2( \R^d, (0,\infty) ) $, 
	assume for all 
		$ x \in \R^d $ 
	that 	
		\begin{equation} \label{viscosity_solution_inhomogeneous_elliptic_equation:lyapunov_assumption}
		\tfrac12\operatorname{Trace}(BB^{*}(\operatorname{Hess} V)(x)) \leq \rho V(x), 
		\end{equation} 
	assume that 
		$ \sup_{ r \in (0,\infty) } [ \inf_{\norm{x}>r} V(x) ] = \infty $ 
	and 
		$ \inf_{ r \in (0,\infty) } [ \sup_{ \norm{x}> r } ( \frac{| h(x) |}{ V(x) } ) ] = 0 $, 
	let 
		$ ( \Omega, \mc F, \P ) $ be a probability space, 
	let 
		$ W \colon [0,\infty) \times \Omega \to \R^m $ be a standard Brownian motion, 
	and let $ u \colon \R^d \to \R $ satisfy for all 
		$ x \in \R^d $ 
	that 
		\begin{equation} \label{viscosity_solution_inhomogeneous_elliptic_equation:fixed_point_equation}
		u(x) = \Exp{\int_0^{\infty} e^{-\lambda s} h(x+BW_s) \,ds }
		\end{equation}   
	(cf.~Item~\eqref{existence_of_fixpoint:item2} in \cref{cor:existence_of_fixpoint}). 
	Then it holds that $ u $ is a viscosity solution of 
		\begin{equation} \label{viscosity_solution_inhomogeneous_elliptic_equation:claim}
		\lambda u(x) - \tfrac12 \operatorname{Trace}(BB^{*}(\operatorname{Hess} u)(x)) = h(x)  
		\end{equation} 
	for $ x \in \R^d $. 
\end{lemma}

\begin{proof}[Proof of \cref{lem:viscosity_solution_inhomogeneous_elliptic_equation}]
	Throughout this proof let 
		$ \mf h_n \in C^{\infty}(\R^d,\R) $, 
		$ n \in \N $,  
	be compactly supported functions which satisfy  
		\begin{equation} \label{viscosity_solution_inhomogeneous_elliptic_equation:approximations}
		\limsup_{n \to \infty} 
		\left[ 
		\sup_{x\in\R^d} \left( \frac{| \mf h_n(x) - h(x) |}{V(x)}  \right)
		\right] 
		= 0,   
		\end{equation} 
	let 
		$ F_n \colon \R^d \times \R \times \R^d \times \Sym_d \to \R $, $ n \in \N_0 $, 
	satisfy for all 
		$ n \in \N $, 
		$ x,p \in \R^d $, 
		$ r \in \R $, 
		$ A \in \Sym_d $ 
	that 
		\begin{equation} \label{viscosity_solution_inhomogeneous_elliptic_equation:Fn_functions}
		F_n(x,r,p,A) 
		= 
		\lambda r - \tfrac12 \operatorname{Trace}(BB^{*}A) - \mf h_n(x) 
		\,\,\text{and}\,\,
		F_0(x,r,p,A) 
		= 
		\lambda r - \tfrac12 \operatorname{Trace}(BB^{*}A) - h(x),  
		\end{equation} 	
	let $ \mf u_n \colon \R^d \to \R $, $ n \in \N $, 
	satisfy for all 
		$ n \in \N $, 
		$ x \in \R^d $ 
	that 
		\begin{equation} \label{viscosity_solution_inhomogeneous_elliptic_equations}
		\mf u_n(x) = \Exp{ \int_0^{\infty} e^{-\lambda s} \, \mf h_n(x+BW_s)\,ds}\!,  
		\end{equation}
	and let $ \mf v_n \colon [0,\infty) \times \R^d \to \R $, $n \in \N $, satisfy for all 
		$ n \in \N $, 
		$ t \in [0,\infty) $, 
		$ x \in \R^d $ 
	that $ \mf v_n(t,x) = \EXP{\mf h_n(x+BW_t)} $. 
	Observe that \cref{cor:classical_solution_brownian_motion} ensures for all 
		$ n \in \N $, 
		$ t \in [0,\infty) $, 
		$ x \in \R^d $
	that 
		$ \mf v_n \in C^{1,2}([0,\infty)\times\R^d,\R) $
	and 
		\begin{equation} 
		(\tfrac{\partial}{\partial t}\mf v_n)(t,x) 
		= 
		\tfrac12 \operatorname{Trace}(BB^{*}(\operatorname{Hess}_x \mf v_n)(t,x) ). 
		\end{equation} 
	This, \eqref{viscosity_solution_inhomogeneous_elliptic_equations}, the fact that for all 
		$ n \in \N $ 
	it holds that 
		$ \sup_{(t,x)\in [0,\infty) \times \R^d} [ \sum_{i,j=1}^d (|\mf v_n(t,x)|+|(\tfrac{\partial}{\partial x_i}\mf v_n)(t,x)|+|(\tfrac{\partial^2}{\partial x_i\partial x_j}\mf v_n)(t,x)|)] < \infty $, 
	integration by parts, and 
	Lebesgue's dominated convergence theorem guarantee that for all 
		$ n \in \N $, 
		$ x \in \R^d $ 
	it holds that 
		\begin{equation} 
		\begin{split}
		\mf u_n(x) 
		&= 
		\int_0^{\infty} e^{-\lambda t} \, \mf v_n(t,x) \,dt 
		= 
		\lim_{R\to\infty} 
		\left[  
		\int_0^R e^{-\lambda t} \, \mf v_n(t,x) \,dt \right] 
		\\
		&= 
		\lim_{R \to \infty} 
		\left[ 
		\frac{1}{\lambda} \mf v_n(0,x) 
		-
		\frac{e^{-\lambda R}}{\lambda}\mf v_n(R,x)
		+ 
		\frac{1}{\lambda} 
		\int_0^R e^{-\lambda t} (\tfrac{\partial}{\partial t}\mf v_n)(t,x) \,dt 
		\right] 
		\\
		&=
		\frac{1}{\lambda} 
		\left[ 
		\mf v_n(0,x) 
		+ 
		\int_0^{\infty} e^{-\lambda t} (\tfrac{\partial}{\partial t}\mf v_n)(t,x)\,dt 
		\right]
		\\
		&= 
		\frac{1}{\lambda} 
		\left[ 
		\mf h_n(x) 
		+ 
		\int_0^{\infty} e^{-\lambda t} 
		\tfrac12\operatorname{Trace}(BB^{*}(\operatorname{Hess}_x \mf v_n)(t,x))\,dt
		\right] 
		\\
		&= 
		\frac{1}{\lambda} \left[ 
		\mf h_n(x) 
		+ 
		\tfrac12 \operatorname{Trace}(BB^{*}(\operatorname{Hess} \mf u_n)(x))
		\right]\!.
		\end{split}
		\end{equation}
	This shows for every $ n \in \N $ that $ \mf u_n $ is a viscosity solution of 
		\begin{equation} \label{viscosity_solution_inhomogeneous_elliptic_equations:equation_for_u_n}
		\lambda \mf u_n(x) - \tfrac12 \operatorname{Trace}(BB^{*}(\operatorname{Hess} \mf u_n)(x)) = \mf h_n(x) 
		\end{equation} 
	for $ x \in \R^d $. 
	Next note that \eqref{viscosity_solution_inhomogeneous_elliptic_equation:approximations} and \eqref{viscosity_solution_inhomogeneous_elliptic_equation:Fn_functions} ensure for every non-empty compact set
		$ K \subseteq \R^d\times \R \times \R^d \times \Sym_d $ 
	that 
		\begin{equation} \label{viscosity_solution_inhomogeneous_elliptic_equation:convergence_rhs}
		\begin{split}
		& \limsup_{n \to \infty} 
		\left[ 
		\sup_{(x,r,p,A) \in K} \left|
		F_n(x,r,p,A) - F_0(x,r,p,A) 
		\right|
		\right] 
		= 
		\limsup_{n \to \infty} 
		\left[ 
		\sup_{(x,r,p,A) \in K} |\mf h_n(x) - h(x)| 
		\right]
		\\
		& \leq 
		\limsup_{n \to \infty} 
		\left( 
		\left[ 
		\sup_{y\in\R^d} \left( \frac{|\mf h_n(y) - h(y)|}{V(y)}  \right) 
		\right] 
		\left[ 
		\sup_{(x,r,p,A) \in K} V(x) 
		\right]
		\right) 
		= 0 . 
		\end{split} 
		\end{equation}
	Moreover, note that
	\eqref{viscosity_solution_inhomogeneous_elliptic_equation:lyapunov_assumption}, 
	\eqref{viscosity_solution_inhomogeneous_elliptic_equation:fixed_point_equation},
	and \eqref{viscosity_solution_inhomogeneous_elliptic_equations} guarantee for all 
		$ n \in \N $,
		$ x \in \R^d $ 
	that 
		\begin{equation} 
		\begin{split} 
		& 
		\frac{ | \mf u_n(x) - u(x) |}{ V(x) } 
		= 
		\frac{1}{V(x)} 
		\left| 
		\Exp{ \int_0^{\infty} e^{-\lambda s} \left( h(x+BW_s) - \mf h_n(x+BW_s) \right)\!\,ds }
		\right|  
		\\
		& \leq 
		\int_0^{\infty} e^{-\lambda s} \left[ \sup_{y\in\R^d} \left(  \frac{|h(y)-\mf h_n(y)|}{V(y)} \right) \right] \frac{\Exp{V(x+BW_s)}}{V(x)} \, ds 
		\leq 
		\frac{1}{\lambda-\rho}	\left[ \sup_{y\in\R^d} \left(  \frac{|h(y)-\mf h_n(y)|}{V(y)} \right) \right]	\!.
		\end{split} 
		\end{equation}  
	This and \eqref{viscosity_solution_inhomogeneous_elliptic_equation:approximations} imply for every non-empty compact set $ K \subseteq \R^d $ that 
		\begin{equation}
			\limsup_{n\to\infty} 
			\left[ 
			\sup_{x\in K} 
			|\mf u_n(x) - u(x) |
			\right] 
			= 0. 
		\end{equation} 
	This, \eqref{viscosity_solution_inhomogeneous_elliptic_equation:convergence_rhs}, 
	the fact that for all 
		$ n \in \N $ 
	it holds that 
		$ \mf u_n $ 
	is a viscosity solution of 
		\begin{equation}  
		F_n(x,\mf u_n(x),(\nabla\mf u_n)(x),(\operatorname{Hess} \mf u_n)(x)) = 0 
		\end{equation} 
	for $ x \in \R^d $ (cf.~\eqref{viscosity_solution_inhomogeneous_elliptic_equation:Fn_functions} and \eqref{viscosity_solution_inhomogeneous_elliptic_equations:equation_for_u_n}), 
	and Hairer et al.~\cite[Lemma 4.8]{HaHuJe2017_LossOfRegularityKolmogorov} (see also Barles \& Perthame~\cite{BarlesPerthame}) imply that $ u $ is a viscosity solution of
		\begin{equation} 
		F_0(x,u(x),(\nabla u)(x),(\operatorname{Hess} u)(x)) = 0 
		\end{equation} 
	for $ x \in \R^d $. 
	This establishes \eqref{viscosity_solution_inhomogeneous_elliptic_equation:claim}. The proof of \cref{lem:viscosity_solution_inhomogeneous_elliptic_equation} is thus completed. 
\end{proof} 

\begin{cor}\label{cor:fixpoint_is_viscosity_solution}
	Let $ d,m \in \N $, 
		$ B \in \R^{d\times m} $, 
		$ L,\rho \in \R $, 
		$ \lambda \in (L+\rho,\infty) $, 
	let $ \norm{\cdot}\colon\R^d\to [0,\infty) $ be a norm on $\R^d $, 
	let $ f \in C( \R^d\times\R,\R) $, $ u \in C(\R^d,\R)$, $V \in C^2(\R^d,(0,\infty)) $ 
	satisfy 
		$ \sup_{r\in (0,\infty)} [ \inf_{\norm{x}>r} V(x)] = \infty $
	and 
		$ \inf_{r\in (0,\infty)} [ \sup_{\norm{x}>r} (\frac{|f(x,0)|+|u(x)|}{V(x)})] = 0 $, 
	assume for all 
		$ x \in \R^d $ 
	that 
		\begin{equation} \label{fixpoint_is_viscosity_solution:supersolution}
		\tfrac12 \operatorname{Trace}\!\big( BB^{*}(\operatorname{Hess} V)(x) \big) 
		\leq \rho V(x), 
		\end{equation}  
	let $ ( \Omega, \mc F, \P ) $ be a probability space, 
	let $ W \colon [0,\infty) \times \Omega \to \R^m $ be a standard Brownian motion, 
	and assume for all 
		$ x \in \R^d $, 
		$ v,w \in \R $ 
	that $ | f(x,v) - f(x,w) | \leq L |v-w| $ and 
		\begin{equation}\label{fixpoint_is_viscosity_solution:sfpe}
		u(x) = \Exp{\int_0^{\infty} e^{-\lambda t} f(x+BW_t,u(x+BW_t))\,dt} 
		\end{equation} 
	(cf.\ \cref{cor:existence_of_fixpoint}).
Then it holds that $u$ is a viscosity solution of 
	\begin{equation}\label{fixpoint_is_viscosity_solution:claim}
	\lambda u(x) - 
	\tfrac12 \operatorname{Trace}\!\big(BB^{*}(\operatorname{Hess}u)(x)\big) 
	=  
	f(x,u(x)) 
	\end{equation}
for $ x \in \R^d $. 
\end{cor} 

\begin{proof}[Proof of \cref{cor:fixpoint_is_viscosity_solution}]
	Throughout this proof let $ h\colon \R^d \to \R $ satisfy for all 
		$ x \in \R^d $ 
	that 
		$ h(x) = f(x,u(x)) $. 
	Observe that the assumptions that $V\in C(\R^d,(0,\infty)) $ and $ \sup_{r\in (0,\infty)} [\inf_{ \norm{x}>r } V(x)] = \infty $ imply that $ \{ x \in \R^d \colon V(x) = \inf_{y\in\R^d} V(y) \} \neq \emptyset $. This, the fact that for all 
		$ x \in \{ y \in \R^d \colon V(y) = \inf_{z\in\R^d} V(z) \} $ it holds that 
		$ (\operatorname{Hess} V)(x) \geq 0 $, and \eqref{fixpoint_is_viscosity_solution:supersolution}
	ensure that $ \rho \geq 0 $. Hence, we obtain that $ \lambda \in (0,\infty) $. 
	Next note that \eqref{fixpoint_is_viscosity_solution:sfpe} and \cref{lem:viscosity_solution_inhomogeneous_elliptic_equation} prove that 
		$ u $ 
	is a viscosity solution of 
		\begin{equation} \label{viscosity_solution_inhomogeneous_elliptic_equation:equation}
		\lambda u(x) - 
		\tfrac12 \operatorname{Trace}\!\big(BB^{*}(\operatorname{Hess}u)(x)\big) 
		=  
		h(x) 
		\end{equation} 
	for $ x \in \R^d $. 
	This implies for all 
		$ x \in \R^d $, 
		$ \varphi \in C^2(\R^d,\R) $ 
	with 
		$ \varphi \geq u $ and $ \varphi(x) = u(x) $ 
	that 
		\begin{equation} \label{viscosity_solution_inhomogeneous_elliptic_equation:subsolution}
		\begin{split}
		0 
		& 
		\geq 
		\lambda \varphi(x) 
		- 
		\tfrac12 \operatorname{Trace}\!\big(BB^{*}(\operatorname{Hess}\varphi)(x)\big) 
		-
		h(x)
		\\
		& 
		= 
		\lambda \varphi(x) 
		- 
		\tfrac12 \operatorname{Trace}\!\big(BB^{*}(\operatorname{Hess}\varphi)(x)\big) 
		- 
		f(x,\varphi(x))
		. 
		\end{split}
		\end{equation} 
	Moreover, note that \eqref{viscosity_solution_inhomogeneous_elliptic_equation:equation} implies for all 
		$ x \in \R^d $, 
		$ \varphi \in C^2(\R^d,\R) $ 
	with 
		$ \varphi \leq u $ and $ \varphi(x) = u(x) $ 
	that 
		\begin{equation} 
		\begin{split}
		0 
		& 
		\leq 
		\lambda \varphi(x) 
		- 
		\tfrac12 \operatorname{Trace}\!\big(BB^{*}(\operatorname{Hess}\varphi)(x)\big) 
		- 
		h(x)
		\\
		& = 
		\lambda \varphi(x) 
		- 
		\tfrac12 \operatorname{Trace}\!\big(BB^{*}(\operatorname{Hess}\varphi)(x)\big) 
		-
		f(x,\varphi(x)).   
		\end{split}
		\end{equation} 
	This and \eqref{viscosity_solution_inhomogeneous_elliptic_equation:subsolution} demonstrate \eqref{fixpoint_is_viscosity_solution:claim}. The proof of \cref{cor:fixpoint_is_viscosity_solution} is thus completed. 
\end{proof} 

\subsection{On a comparison principle for viscosity solutions of semilinear elliptic PDEs}
\label{subsec:comparison_principle}

\begin{prop}[Comparison principle] \label{prop:comparison_principle}
	Let 
		$ d,m \in \N $, 
		$ B \in \R^{d\times m} $,
		$ L, \rho \in \R $, 
		$ \lambda \in (\rho+L,\infty) $, 
	let 
		$\norm{\cdot}\colon\R^d\to [0,\infty)$ 
	be the standard norm on $\R^d$, 
	let
		$ f \in C(\R^d\times\R,\R)$, 
		$ g,h,u,v \in C(\R^d,\R)$, 
		$ V \in C^2(\R^d,(0,\infty))$ 
	satisfy for all 
		$x\in\R^d$, 
		$a,b\in\R$ 
	that 
		$ [f(x,a)-f(x,b)-L(a-b)](a-b)\leq 0 $
	and 
		\begin{equation} \label{comparison_principle:lyapunov}
		\tfrac12\operatorname{Trace}\!\left(BB^{*}(\operatorname{Hess} V)(x)\right) 
		\leq 
		\rho V(x), 
		\end{equation} 
	assume that 
		$
		\limsup_{r\to\infty} 
		[\sup_{\norm{x}>r} (\frac{|f(x,0)|+|g(x)|+|h(x)|+|u(x)|+|v(x)|}{V(x)})] 
		= 0
		$,
	assume that $u$ is viscosity supersolution of 
		\begin{equation} \label{comparison_principle:supersolution} 
			\lambda u(x) 
			- 
			\tfrac12 \operatorname{Trace}\!\left( BB^{*}(\operatorname{Hess} u)(x) \right) 
			= 
			f(x,u(x)) 
			+
			g(x) 
		\end{equation} 
	for $x\in\R^d$, 
	and assume that $v$ is a viscosity subsolution of 
	\begin{equation} \label{comparison_principle:subsolution} 
		\lambda v(x)
		- 
		\tfrac12 
		\operatorname{Trace}\!\left( 
		BB^{*}(\operatorname{Hess} v)(x)
		\right) 
		=  
		f(x,v(x)) 
		+ 
		h(x) 
		\end{equation} 
	for $x\in\R^d$. Then it holds that 
		\begin{equation} \label{comparison_principle:claim}
		\sup_{x\in\R^d}
		\left[ \max\left\{\frac{v(x)-u(x)}{V(x)},0\right\} \right]
		\leq 
		\frac{1}{\lambda-(L+\rho)}
		\left[
		\sup_{x\in\R^d} \left(
		\max\left\{\frac{h(x)-g(x)}{V(x)},0\right\} \right)\right]\!.
		\end{equation} 	
	\end{prop} 

\begin{proof} [Proof of \cref{prop:comparison_principle}]
	Throughout this proof let 
		$ \langle\cdot,\cdot\rangle\colon\R^d\times\R^d\to\R $ 
	be the standard scalar product on $\R^d$, 
	let 
		$w_1,w_2\in C(\R^d,\R)$ 
	satisfy for all 
		$x\in\R^d$ 
	that 
		$w_1(x)=\frac{u(x)}{V(x)}$ 
	and 
		$w_2(x)=\frac{v(x)}{V(x)}$,  
	and let 
		$\eta_{\alpha}\colon\R^d\times\R^d\to\R$, $\alpha\in (0,\infty)$, 
	satisfy for all 
		$\alpha\in (0,\infty)$, 
		$x,y\in\R^d$ 
	that 
		$\eta_{\alpha}(x,y) 
		= 
		w_2(x)-w_1(y)-\frac{\alpha}{2}\norm{x-y}^2$. 
	Note that \eqref{comparison_principle:claim} is clear in the case of $v\leq u$. 
	Therefore, we assume in the following that there exists 
		$x\in\R^d$ 
	such that 
		$v(x)-u(x)>0$. 
	This implies that there exists 
		$x\in\R^d$ 
	such that 
		$w_2(x)-w_1(x)>0$. 
	Next note that the hypothesis that
		$\limsup_{r\to\infty}[\sup_{\norm{x}>r} 
		(\frac{|u(x)|+|v(x)|}{V(x)})] = 0$
	demonstrates that 
		\begin{equation} \label{comparison_principle:w_1_and_w_2_vanishing_at_infinity}
		\limsup_{r\to\infty} 
		\left[ 
		\sup_{\norm{x}>r} 
		\Big( 
		|w_1(x)|+|w_2(x)|
		\Big) 
		\right] 
		= 0. 
		\end{equation}
	Combining this with the fact that 
		$w_1, w_2\in C(\R^d,\R)$ 
	and the fact that 
		$\sup_{x\in\R^d} (\max\{w_2(x)-w_1(x),0\})\in (0,\infty]$ 
	implies that there exists 
		$x_0\in\R^d$ 
	which satisfies that 
		\begin{equation} \label{comparison_principle:positive_maximum}
		w_2(x_0)-w_1(x_0) 
		= 
		\sup_{y\in\R^d} 
		\Big(w_2(y)-w_1(y)\Big) 
		> 0. 
		\end{equation} 
	In addition, note that \eqref{comparison_principle:w_1_and_w_2_vanishing_at_infinity} and the fact that 	
		$ w_1, w_2 \in C(\R^d,\R) $ 
	ensure that 
		$ \sup_{x\in\R^d} (|w_1(x)|+|w_2(x)|) < \infty $. 
	This and \eqref{comparison_principle:positive_maximum} imply for all 
		$ \alpha \in (0,\infty)$, 
		$ \beta \in (\alpha,\infty) $
	that 
		\begin{equation} \label{comparison_principle:chain_of_inequalities_for_maxs}
		\begin{split}
		\infty 
		& 
		> 
		\left[ \sup_{ x \in \R^d } |w_1(x)| \right] 
		+
		\left[ \sup_{ x \in \R^d } |w_2(x)| \right]  
		\geq 
		\sup_{ x,y \in \R^d } (w_2(x) - w_1(y)) 
		\geq 
		\sup_{ z \in \R^d \times \R^d } \eta_{\alpha}(z)
		\\
		& \geq 
		\sup_{ z \in \R^d \times \R^d } \eta_{\beta}(z)
		\geq 
		\sup_{ x \in \R^d } \eta_{\beta}(x,x) 
		= 
		\sup_{ x \in \R^d } (w_2(x)-w_1(x)) > 0. 
		\end{split}
		\end{equation} 
	Next let 
		$ r_{\alpha} \in (0,\infty) $, $ \alpha \in (0,\infty) $, 
	satisfy for all 
		$ \alpha \in (0,\infty) $ 
	that 
		$ r_{\alpha} = [ \frac{2}{\alpha}([\sup_{x\in\R^d} |w_1(x)|] + [\sup_{x \in \R^d} |w_2(x)|]) ]^{\nicefrac12} $
	and let 
		$ R \in (0,\infty) $ 
	satisfy that for all 
		$ x \in \R^d $ 
	with 
		$ \norm{x}>R $ 
	it holds that 
		$ |w_1(x)| + |w_2(x)| < \frac14 \sup_{y\in\R^d} (w_2(y)-w_1(y)) $.
	Furthermore, observe that for all 
		$ \alpha \in (0,\infty) $,  
		$ x,y \in \R^d $ 
	with 
		$ \norm{x-y} > r_{\alpha} $ 
	it holds that $ w_2(x)-w_1(y)-\frac{\alpha}{2}\norm{x-y}^2 \leq 0 $. Hence, we obtain for all 
		$ \alpha \in (0,\infty) $, 
		$ x,y \in \R^d $ 
	with 
		$ \max\{\norm{x},\norm{y}\} > R+r_{\alpha} $ 
	that 
		\begin{equation} 
		\eta_{\alpha}(x,y) 
		\begin{cases} 
		\leq 0 & \colon \norm{x-y} > r_{\alpha}, \\
		\leq \frac12 \sup_{z\in\R^d} (w_2(z)-w_1(z)) & \colon \norm{x-y} \leq r_{\alpha}.
		\end{cases}
		\end{equation} 
	Combining this with 
		\eqref{comparison_principle:chain_of_inequalities_for_maxs}
	demonstrates for all 
		$ \alpha \in (0,\infty) $ 
	that 
		\begin{equation} 
		\sup_{z\in\R^d\times\R^d} \eta_{\alpha}(z) = \sup_{\max\{\norm{x},\norm{y}\}\leq R+r_{\alpha}} \eta_{\alpha}(x,y).
		\end{equation}  
	Hence, we obtain that for every 
		$\alpha\in (0,\infty)$ 
	there exist
		$ x_{\alpha},y_{\alpha}\in\R^d$ 		
	which satisfy that 
		$ \max\{\norm{x_{\alpha}},\norm{y_{\alpha}}\} \leq R + r_{\alpha} $ 
	and 
	\begin{equation} \label{comparison_principle:minimizers_of_regularized_problems}
	w_2(x_{\alpha})-w_1(y_{\alpha})-
	\frac{\alpha}{2}\norm{x_{\alpha}-y_{\alpha}}^2 
	= 
	\sup_{z\in\R^d\times\R^d} \eta_{\alpha}(z)
	> 0. 
	\end{equation} 
	Crandall et al.~\cite[Theorem 3.2]{CIL1992UsersGuide} 
	(with 
		$ k \is 2 $, 
		$ N_1 \is d $, 
		$ N_2 \is d $, 
		$ \mc O_1 \is \R^d $, 
		$ \mc O_2 \is \R^d $, 
		$ u_1 \is w_2 $, 
		$ u_2 \is -w_1 $, 
		$ \varphi \is (\R^d\times\R^d\ni (x,y)\mapsto \frac{\alpha}{2}\norm{x-y}^2\in\R)$, 
		$ \hat{x} \is (x_{\alpha},y_{\alpha}) $ 
	for 
		$ \alpha \in (0,\infty) $
	in the notation of \cite[Theorem 3.2]{CIL1992UsersGuide}) therefore guarantees that there exist 
		$ X_{\alpha},Y_{\alpha}\in\Sym_{d} $, $ \alpha \in (0,\infty) $, 
	which satisfy for all 
		$ \alpha\in (0,\infty) $ 
	that 
	$ (\alpha(x_{\alpha}-y_{\alpha}),X_{\alpha}) \in (\hat{J}^{2}_{+}w_2)(x_{\alpha}) $, 
	$ (\alpha(x_{\alpha}-y_{\alpha}),Y_{\alpha}) \in (\hat{J}^{2}_{-}w_1)(y_{\alpha}) $, 
	and 
	\begin{equation} \label{comparison_principle:jensen_ishii} 
		-3 \alpha 
		\begin{pmatrix} 
			I & 0 \\
			0 & I
		\end{pmatrix} 
		\leq 
		\begin{pmatrix}
			X_{\alpha} & 0 \\
			0 & -Y_{\alpha} 
		\end{pmatrix} 
		\leq 
		 3 \alpha 
		\begin{pmatrix} 
			I & -I \\
		 	-I & I
		\end{pmatrix}
		\end{equation}
	(see Hairer et al.~\cite[Definition 4.3]{HaHuJe2017_LossOfRegularityKolmogorov} for definitions of $\hat{J}^{2}_{+}w_2$ and $\hat{J}^{2}_{-}w_1$).  
	Next observe that \eqref{comparison_principle:supersolution} implies that $w_1$ is a viscosity supersolution of 
	\begin{equation} 
	\begin{split}
	\lambda w_1(x) & - \bigg[ 
	\tfrac12\operatorname{Trace}\left( 
	BB^{*}(\operatorname{Hess} w_1)(x)
	\right) 
	+ 
	\left\langle BB^{*}\tfrac{(\nabla V)(x)}{V(x)}, (\nabla w_1)(x) \right\rangle \\
	& \qquad + 
	\tfrac12\operatorname{Trace}\!\left( 
	BB^{*}\tfrac{(\operatorname{Hess} V)(x)}{V(x)}
	\right) w_1(x) 
	+ 
	\tfrac{1}{V(x)}f\big(x,V(x)w_1(x)\big)
	+ \tfrac{g(x)}{V(x)}
	\bigg] 
	= 
	0
	\end{split}
	\end{equation} 
	for $x\in\R^d$. Combining this and \eqref{comparison_principle:jensen_ishii} assures for all 
		$ \alpha \in (0,\infty) $ 
	that
		\begin{equation} \label{comparison_principle:first_inequality}
		\begin{split}
		 &  
		 \lambda w_1(y_{\alpha})
		 - \bigg[ 
		 \tfrac12\operatorname{Trace}\left( 
		 BB^{*}Y_{\alpha}
		 \right) 
		 + 
		 \left\langle BB^{*}\tfrac{(\nabla V)(y_{\alpha})}{V(y_{\alpha})}, \alpha(x_{\alpha}-y_{\alpha}) \right\rangle 
		 \\
		 & \qquad + 
		 \tfrac12\operatorname{Trace}\!\left( 
		 BB^{*}\tfrac{(\operatorname{Hess} V)(y_{\alpha})}{V(y_{\alpha})}
		 \right) w_1(y_{\alpha}) 
		 +
		 \tfrac{1}{V(y_{\alpha})}f\big(y_{\alpha},V(y_{\alpha})w_1(y_{\alpha})\big)
		 +
		 \tfrac{g(y_{\alpha})}{V(y_{\alpha})} 
		 		 \bigg] 
		 \geq 0	 
		 . 
		\end{split}
		\end{equation} 
	In addition, note that \eqref{comparison_principle:subsolution} ensures that $w_2$ 	is a viscosity subsolution of 
		\begin{equation} 
		\begin{split}
		\lambda w_2(x) & - \bigg[ 
		\tfrac12\operatorname{Trace}\left( 
		BB^{*}(\operatorname{Hess} w_2)(x)
		\right) 
		+
		\left\langle BB^{*}\tfrac{(\nabla V)(x)}{V(x)}, (\nabla w_2)(x) \right\rangle\\
		& + 
		\tfrac12\operatorname{Trace}\!\left( 
		BB^{*}\tfrac{(\operatorname{Hess} V)(x)}{V(x)}
		\right) w_2(x) 
		+ 
		\tfrac{1}{V(x)}f\big(x,V(x)w_2(x)\big)
		+
		\tfrac{h(x)}{V(x)} 
		\bigg] 
		= 0
		\end{split}
		\end{equation} 
	for $ x \in \R^d $. Combining this and \eqref{comparison_principle:jensen_ishii} implies for all 
		$\alpha\in (0,\infty)$ 
	that 
		\begin{equation} 
		\begin{split}
		& 
		\lambda w_2(x_{\alpha}) 
		- 
		\bigg[
		\tfrac12\operatorname{Trace}\left( 
		BB^{*}X_{\alpha}
		\right) 
		+ 
		\left\langle BB^{*}\tfrac{(\nabla V)(x_{\alpha})}{V(x_{\alpha})}, \alpha(x_{\alpha}-y_{\alpha}) \right\rangle
		\\
		& \qquad 
		+ 
		\tfrac12\operatorname{Trace}\!\left( 
		BB^{*}\tfrac{(\operatorname{Hess} V)(x_{\alpha})}{V(x_{\alpha})}
		\right) w_2(x_{\alpha}) 
		+ 
		\tfrac{1}{V(x_{\alpha})}f\big(x_{\alpha},V(x_{\alpha})w_2(x_{\alpha})\big)
		+ 
		\tfrac{h(x_{\alpha})}{V(x_{\alpha})}
		\bigg] 
		\leq 0
		. 
		\end{split}
		\end{equation} 
	This and \eqref{comparison_principle:first_inequality} assure for all 
		$ \alpha \in (0,\infty) $ 
	that 
		\begin{equation} \label{comparison_principle:inequality}
		\begin{split}
		\lambda (w_2(x_{\alpha}) - w_1(y_{\alpha}))
		& 
		\leq 
		\tfrac12 \operatorname{Trace}\!\left( 
		BB^{*}(X_{\alpha}-Y_{\alpha})
		\right)
		+ 
		\left\langle 
		BB^{*}\left(
			\tfrac{(\nabla V)(x_{\alpha})}{V(x_{\alpha})}
			-
			\tfrac{(\nabla V)(y_{\alpha})}{V(y_{\alpha})}\right),
		\alpha(x_{\alpha}-y_{\alpha})	
		\right\rangle 
		\\
		& +
		\tfrac12 \operatorname{Trace}\!\left( 
		BB^{*}\left(\tfrac{(\operatorname{Hess} V)(x_{\alpha})}{V(x_{\alpha})} w_2(x_{\alpha})
		- 
		\tfrac{(\operatorname{Hess} V)(y_{\alpha})}{V(y_{\alpha})} w_1(y_{\alpha})\right)
		\right)
		\\ 
		& + 
		\tfrac{1}{V(x_{\alpha})}f\big(x_{\alpha},V(x_{\alpha})w_2(x_{\alpha})\big)
		- 
		\tfrac{1}{V(y_{\alpha})}f\big(y_{\alpha},V(y_{\alpha})w_1(y_{\alpha})\big)
		+ 
		\tfrac{h(x_{\alpha})}{V(x_{\alpha})}-\tfrac{g(y_{\alpha})}{V(y_{\alpha})}.  
		\end{split}
		\end{equation} 
	Next note that \eqref{comparison_principle:jensen_ishii} ensures for all 
		$ \alpha \in (0,\infty) $ 
	that	 
		$ X_{\alpha} \leq Y_{\alpha} $. 
	This and \eqref{comparison_principle:inequality} imply for all 
		$ \alpha \in (0,\infty) $ 
	that 
		\begin{equation} \label{comparison_principle:one_simplification}
		\begin{split} 
		\lambda (w_2(x_{\alpha})-w_1(y_{\alpha}) )
		& \leq 
		\left\langle 
		BB^{*}\left(
		\tfrac{(\nabla V)(x_{\alpha})}{V(x_{\alpha})}
		-
		\tfrac{(\nabla V)(y_{\alpha})}{V(y_{\alpha})}\right),
		\alpha(x_{\alpha}-y_{\alpha})	
		\right\rangle 
		\\
		& +
		\tfrac12 \operatorname{Trace}\!\left( 
		BB^{*}\!\left(\tfrac{(\operatorname{Hess} V)(x_{\alpha})}{V(x_{\alpha})} w_2(x_{\alpha})
		- 
		\tfrac{(\operatorname{Hess} V)(y_{\alpha})}{V(y_{\alpha})} w_1(y_{\alpha})\right)
		\right)
		\\ 
		& + 
		\tfrac{1}{V(x_{\alpha})}f\big(x_{\alpha},V(x_{\alpha})w_2(x_{\alpha})\big)
		- 
		\tfrac{1}{V(y_{\alpha})}f\big(y_{\alpha},V(y_{\alpha})w_1(y_{\alpha})\big)
		+ 
		\tfrac{h(x_{\alpha})}{V(x_{\alpha})}-\tfrac{g(y_{\alpha})}{V(y_{\alpha})}.  
		\end{split}
		\end{equation} 
	Moreover, observe that 		
		\eqref{comparison_principle:chain_of_inequalities_for_maxs}
	implies that 
		$ \lim_{\alpha\to\infty} [\sup_{z\in\R^d\times\R^d} \eta_{\alpha}(z)] \in \R $ exists.  
	Hairer et al.~\cite[Lemma 4.9]{HaHuJe2017_LossOfRegularityKolmogorov} (with 
		$ d \is 2d $, 
		$ O \is \R^d\times\R^d $, 
		$ \eta \is ( \R^d\times\R^d \ni (x,y) \mapsto w_2(x) - w_1(y) \in \R ) $, 
		$ \phi \is ( \R^d\times\R^d \ni (x,y) \mapsto \frac12 \norm{x-y}^2 \in [0,\infty) ) $, 
		$ x \is ( (0,\infty) \ni \alpha \mapsto (x_{\alpha},y_{\alpha}) \in \R^d\times\R^{d} ) $
	in the notation of Hairer et al.~\cite[Lemma 4.9]{HaHuJe2017_LossOfRegularityKolmogorov}) and 
	\eqref{comparison_principle:minimizers_of_regularized_problems}
	therefore ensure that 
		$\limsup_{\alpha\to\infty} [\alpha\norm{x_{\alpha}-y_{\alpha}}^2] = 0$. 
	This, the fact that 
		$ \limsup_{\alpha\to\infty} r_{\alpha} = 0 $, 
	the fact that 
		$\frac{\nabla V}{V} \colon \R^d \to \R^d $ 
	is locally Lipschitz continuous,
	and  \eqref{comparison_principle:minimizers_of_regularized_problems} 	
	imply that 
		\begin{equation} \label{comparison_principle:gradient_term_converging_to_zero}
		\limsup_{\alpha\to\infty} 
		\left| 
		\left\langle 
		BB^{*}\!\left( \tfrac{(\nabla V)(x_{\alpha})}{V(x_{\alpha})} 
		- 
		\tfrac{(\nabla V)(y_{\alpha})}{V(y_{\alpha})}
		\right), \alpha(x_{\alpha}-y_{\alpha}) 
		\right\rangle 
		\right| 
		= 0.  
		\end{equation} 
	In addition, note that the fact that 
		$ \limsup_{\alpha\to\infty} r_{\alpha} = 0 $ 
	and 
		\eqref{comparison_principle:minimizers_of_regularized_problems} 
	assure that there exist 
		$\hat{x} \in \R^d $ 
	and 
		$\alpha_n \in (0,\infty)$, $n\in\N$, 
	which satisfy that  
		$\liminf_{n\to\infty} \alpha_n=\infty$ 
	and 
		$\limsup_{n\to\infty} \norm{x_{\alpha_{n}}-\hat{x}} = 0$. 
	This, the fact that 
		$ (\operatorname{Hess} V) \in C(\R^d,\Sym_d) $,
	the fact that 
		$ f \in C(\R^d \times \R, \R) $, 
	the fact that 
		$ V \in C(\R^d, (0,\infty) ) $, 
	the fact that 
		$ u, v, g, h \in C(\R^d,\R) $,
	and the fact that 
		$ \limsup_{\alpha\to\infty} [\frac{\alpha}{2}\norm{x_{\alpha}-y_{\alpha}}^2] = 0 $
	prove that 
	\begin{enumerate}[(i)] 
		\item it holds that 
			\begin{equation} 
			\begin{split} 
			\limsup_{n\to\infty} & \Bigg[ \left| \tfrac12 \operatorname{Trace}\!\left(BB^{*} \tfrac{(\operatorname{Hess} V)(x_{\alpha_n})}{V(x_{\alpha_n})}w_2(x_{\alpha_n}) \right) 
			- \tfrac12 \operatorname{Trace}\!\left(BB^{*} \tfrac{(\operatorname{Hess} V)(\hat{x})}{V(\hat{x})} w_2(\hat{x}) \right) \right| 
			\\
			& 
			+ \left| \tfrac12 \operatorname{Trace}\!\left(BB^{*} \tfrac{(\operatorname{Hess} V)(y_{\alpha_n})}{V(y_{\alpha_n})}w_1(y_{\alpha_n}) \right) 
			- \tfrac12 \operatorname{Trace}\!\left(BB^{*} \tfrac{(\operatorname{Hess} V)(\hat{x})}{V(\hat{x})} w_1(\hat{x}) \right) \right| \Bigg] 
			= 0,
			\end{split}
			\end{equation} 
		\item it holds that 
			\begin{equation} 
			\begin{split}
			\limsup_{n\to\infty} & \Bigg[ \left| \tfrac{1}{V(x_{\alpha_n})} f\big(x_{\alpha_n},V(x_{\alpha_n})w_2(x_{\alpha_n})\big) -  \tfrac{1}{V(\hat{x})} f\big(\hat{x},V(\hat{x})w_2(\hat{x})\big) \right| 
			\\
			& \qquad + \left| \tfrac{1}{V(y_{\alpha_n})} f\big(y_{\alpha_n},V(y_{\alpha_n})w_1(y_{\alpha_n})\big) - \tfrac{1}{V(\hat{x})} f\big(\hat{x},V(\hat{x})w_1(\hat{x})\big) \right| \Bigg] = 0,
			\end{split} 
			\end{equation} 
		\item and it holds that 
			\begin{equation}
			\limsup_{n\to\infty} \left[ \left| \tfrac{h(x_{\alpha_n})}{V(x_{\alpha_n})} - \tfrac{h(\hat{x})}{V(\hat{x})} \right| 
			+ \left| \tfrac{g(y_{\alpha_n})}{V(y_{\alpha_n})} - \tfrac{g(\hat{x})}{V(\hat{x})} \right| \right] = 0 . 
			\end{equation} 
	\end{enumerate}  
	Combining this with \eqref{comparison_principle:one_simplification} and \eqref{comparison_principle:gradient_term_converging_to_zero} 
	shows that 
		\begin{equation} \label{comparison_principle:almost_there}
		\begin{split}
		& \lambda ( w_2(\hat{x}) - w_1(\hat{x}) ) 
		\\
		& \leq 
		\tfrac12 \operatorname{Trace}\!\left( 
		BB^{*}\tfrac{(\operatorname{Hess} V)(\hat{x})}{V(\hat{x})} 
		\right) 
		(w_2(\hat{x})-w_1(\hat{x}))
		+
		\tfrac{f(\hat{x},V(\hat{x})w_2(\hat{x}))-f(\hat{x},V(\hat{x})w_1(\hat{x}))}{V(\hat{x})}
		+ 
		\tfrac{h(\hat{x})}{V(\hat{x})} - \tfrac{g(\hat{x})}{V(\hat{x})}.  
		\end{split} 
		\end{equation} 
	Next note that the second part of the statement of Hairer et al.~\cite[Lemma 4.9]{HaHuJe2017_LossOfRegularityKolmogorov} (with 
		$ d \is 2d $, 
		$ O \is \R^d\times\R^d $, 
		$ \eta \is ( \R^d\times\R^d \ni (x,y) \mapsto w_2(x) - 	w_1(y) \in \R ) $, 
		$ \phi \is ( \R^d\times\R^d \ni (x,y) \mapsto \frac12 \norm{x-y}^2 \in [0,\infty) ) $, 
		$ x \is ( (0,\infty) \ni \alpha \mapsto (x_{\alpha},y_{\alpha}) \in \R^d\times\R^{d} ) $, 
		$ (\alpha_n)_{n\in\N} \is (\alpha_n)_{n\in\N} $, 
		$ x_0 \is (\hat{x},\hat{x}) $
	in the notation of Hairer et al.~\cite[Lemma 4.9]{HaHuJe2017_LossOfRegularityKolmogorov})		
	demonstrates that  
		$ w_2(\hat{x})-w_1(\hat{x}) = \sup_{x\in\R^d} (w_2(x)-w_1(x)) > 0 $. 	
	This, \eqref{comparison_principle:almost_there}, \eqref{comparison_principle:lyapunov}, 
	and the assumption that for all 
		$ x \in \R^d $, 
		$ a,b \in \R $ 
	it holds that 
		$[f(x,a)-f(x,b)](a-b) \leq L |a-b|^2 $ 
	ensure that 
		\begin{equation} 
		\begin{split} 
		& \lambda \left[ \sup_{z\in\R^d} ( \max\{ w_2(z)-w_1(z) , 0 \} ) \right] 
		\\[1ex]
		& =
		\lambda \left( w_2(\hat{x}) - w_1(\hat{x}) \right)
		\leq 
		\rho 
		\left( w_2(\hat{x})-w_1(\hat{x})\right)
		+ 
		\tfrac{ L [V(\hat{x})w_2(\hat{x}) - V(\hat{x})w_1(\hat{x})]}{V(\hat{x})} 
		+ 
		\tfrac{h(\hat{x})}{V(\hat{x})} - \tfrac{g(\hat{x})}{V(\hat{x})} 
		\\[1ex]
		& 		
		\leq 
		(\rho + L) (w_2(\hat{x})-w_1(\hat{x})) 
		+ 
		\sup_{z\in\R^d} \left[ \max\left\{ \tfrac{h(z)-g(z)}{V(z)},0 \right\} \right] 
		\\
		& = 
		(\rho + L) \left[ \sup_{z\in\R^d} ( \max\{ w_2(z)-w_1(z) , 0 \} ) \right]
		+ 
		\sup_{z\in\R^d} \left[ \max\left\{ \tfrac{h(z)-g(z)}{V(z)},0 \right\} \right] \!.
		\end{split}
		\end{equation} 
	Hence, we obtain that 
		\begin{equation}
		\begin{split} 
		\Big[\lambda - (L+\rho)\Big]
		\left[\sup_{x\in\R^d}\left( 
		\max\left\{\tfrac{v(x) - u(x)}{V(x)},0\right\}\right) 
		\right] \leq 
		\sup_{x\in\R^d}\left[\max\left\{ \tfrac{h(x)-g(x)}{V(x)},0 \right\} \right]\!.
		\end{split}
		\end{equation} 
	This establishes \eqref{comparison_principle:claim}. The proof of \cref{prop:comparison_principle} is thus completed. 
\end{proof} 

\subsection{Existence and uniqueness results for viscosity solutions of semilinear elliptic PDEs}
\label{subsec:existence_and_uniqueness}

\begin{prop}[Existence and uniqueness of viscosity solutions] \label{cor:existence_and_uniqueness} 
Let 
	$d,m\in\N$, 
	$B\in\R^{d\times m}$,
	$L,\rho\in \R$, 
	$\lambda\in (\rho+L,\infty)$, 
let 
	$\norm{\cdot}\colon\R^d\to [0,\infty)$ 
be a norm on $\R^d$, 
let
	$f\in C(\R^d\times\R,\R)$, 
	$V\in C^2(\R^d,(0,\infty))$ 
satisfy for all 
	$x\in\R^d$, 
	$v,w\in\R$  
that 
	$|f(x,v)-f(x,w)|\leq L|v-w|$
and 
	\begin{equation} 
	\tfrac12\operatorname{Trace}\!\left(BB^{*}(\operatorname{Hess} V)(x)\right) 
	\leq 
	\rho V(x), 
	\end{equation} 
assume that 
	$
	\inf_{r\in (0,\infty)} 
	[\sup_{\norm{x}>r} (\frac{|f(x,0)|}{V(x)})] 
	= 0
	$
and 
	$ 
	\sup_{r\in (0,\infty)} 
	[ \inf_{\norm{x}>r} V(x) ]
	= \infty
	$, 
let 
	$ ( \Omega, \mc F, \P ) $ be a probability space, 
and let 
	$ W \colon [0,\infty) \times \Omega \to \R^m $ 
be a standard Brownian motion.  
Then 
\begin{enumerate}[(i)] 
	\item \label{existence_and_uniqueness:item1}
	there exists a unique $ u \in \{ v \in C(\R^d,\R)\colon \inf_{r\in (0,\infty)} [\sup_{\norm{x}>r}(\frac{|v(x)|}{V(x)})] = 0 \} $
	which satisfies that $ u $ is a viscosity solution of 
		\begin{equation} \label{existence_and_uniqueness:claim}
		\lambda u( x ) 
		- 
		\tfrac12 \operatorname{Trace}\!\big(BB^{*}(\operatorname{Hess}u)(x)\big) 
		= 
		f(x,u(x))
		\end{equation} 
	for $x\in\R^d$ and 
	\item \label{existence_and_uniqueness:item2}
	it holds for all 
		$ x \in \R^d $ 
	that $ \EXPP{ \int_0^{\infty} e^{-\lambda t} |f(x+BW_t,u(x+BW_t))|\,dt} < \infty $ and
		\begin{equation} 
		u(x) = \Exp{ \int_0^{\infty} e^{-\lambda t} f(x+BW_t,u(x+BW_t))\,dt}\!.
		\end{equation} 
	\end{enumerate}
\end{prop} 

\begin{proof}[Proof of \cref{cor:existence_and_uniqueness}]
	First, observe that \cref{cor:existence_of_fixpoint} (with  
		$ d \is d $, 
		$ m \is m $, 
		$ B \is B $, 
		$ L \is L $, 
		$ \rho \is \rho $, 
		$ \lambda \is \lambda $, 
		$ \norm{\cdot} \is \norm{\cdot} $, 
		$ f \is f $, 
		$ V \is V $, 
		$ (\Omega,\mc F,\P) \is (\Omega,\mc F,\P) $, 
		$ W \is W $
	in the notation of \cref{cor:existence_of_fixpoint})
	guarantees that there exists $ u \in C(\R^d, \R) $ which satisfies for all 
 		$ x \in \R^d $ 
 	that 
 		$ \limsup_{r\to\infty} [\sup_{\norm{y}>r} (\frac{|u(y)|}{V(y)})] = 0 $, 
 		$ \EXPP{ \int_0^{\infty} e^{-\lambda t} |f(x+BW_t,u(x+BW_t))|\,dt} < \infty $,  
 	and 
 		\begin{equation} \label{existence_and_uniqueness:sfpe}
 		u(x) = \Exp{\int_0^{\infty} e^{-\lambda t} f(x+BW_t,u(x+BW_t))\,dt}\!. 
 		\end{equation} 
 	\cref{cor:fixpoint_is_viscosity_solution} (with 
 		$ d \is d $, 
 		$ m \is m $, 
 		$ B \is B $, 
 		$ L \is L $, 
 		$ \rho \is \rho $, 
 		$ \lambda \is \lambda $, 
 		$ \norm{\cdot} \is \norm{\cdot} $, 
 		$ f \is f $, 
 		$ u \is u $, 
 		$ V \is V $, 
 		$ (\Omega,\mc F,\P) \is (\Omega,\mc F,\P) $, 
 		$ W \is W $
 	in the notation of \cref{cor:fixpoint_is_viscosity_solution}) 
 	therefore implies that 
 		$ u $ 
 	is a viscosity solution of 
 		\begin{equation} \label{existence_and_uniqueness:viscosity_solution}
 		\lambda u(x) 
 		- 
 		\tfrac{1}{2}\operatorname{Trace}\!\big( BB^{*}(\operatorname{Hess} u)(x) 
 		\big)  
		= 
		f(x,u(x))
 		\end{equation} 
 	for $ x \in \R^d $. Furthermore, observe that \cref{prop:comparison_principle} (with 
 		$ d \is d $, 
 		$ m \is m $, 
 		$ B \is B $, 
 		$ L \is L $, 
 		$ \rho \is \rho $, 
 		$ \lambda \is \lambda $, 
 		$ f \is f $, 
 		$ g \is ( \R^d \ni x \mapsto 0 \in \R ) $, 
 		$ h \is ( \R^d \ni x \mapsto 0 \in \R ) $, 
 		$ V \is V $
 	in the notation of \cref{prop:comparison_principle}) demonstrates that for every $ v \in \{ w \in C(\R^d,\R) \colon  \inf_{ r \in (0,\infty) } [\sup_{\norm{x}>r} (\frac{|w(x)|}{V(x)})] = 0 \} $ which satisfies that $ v $ is a viscosity solution of 
 		\begin{equation} \label{existence_and_uniqueness:uniqueness_viscosity_solution}
 		\lambda v(x) 
 		- 
 		\tfrac{1}{2}\operatorname{Trace}\!\big( BB^{*}(\operatorname{Hess} v)(x) 
 		\big)  
		= 
		f(x,v(x))
 		\end{equation} 
 		for $ x \in \R^d $ 
  	it holds that $ u = v $. Combining this with \eqref{existence_and_uniqueness:sfpe} and  \eqref{existence_and_uniqueness:viscosity_solution} establishes Item~\eqref{existence_and_uniqueness:item1}. In addition, note that \eqref{existence_and_uniqueness:sfpe}, \eqref{existence_and_uniqueness:viscosity_solution}, and \eqref{existence_and_uniqueness:uniqueness_viscosity_solution} establish Item~\eqref{existence_and_uniqueness:item2}. The proof of \cref{cor:existence_and_uniqueness} is thus completed. 
\end{proof}

\begin{cor} \label{cor:existence_and_uniqueness_strange_lipschitz} 
	Let $ d,m \in \N $, 
		$ B \in \R^{d\times m} $, 
		$ L, \rho \in \R $, 
		$ \lambda \in (\rho + L,\infty) $, 
	let $ \norm{\cdot} \colon \R^d \to [0,\infty) $ be a norm on $ \R^d $, 
	let $ f \in C(\R^d\times\R,\R) $, 
		$ V \in C^2(\R^d,(0,\infty)) $ 
	satisfy for all 
		$ x \in \R^d $, 
		$ v,w \in \R $ 
	that 
		$ | f(x,v) - f(x,w) - \lambda (v-w) | \leq L | v-w | $ 
	and 
		$ \frac12 \operatorname{Trace}(BB^{*}(\operatorname{Hess }V)(x)) \leq \rho V(x) $, 
	assume that 
		$ \inf_{r\in (0,\infty)} [\sup_{\norm{x}>r} (\frac{|f(x,0)|}{V(x)})] = 0 $ 
	and 
		$ \sup_{r\in (0,\infty)} [\inf_{\norm{x}>r} V(x)] = \infty $, 
	let $ (\Omega,\mc F,\P) $ be a probability space, 
	and let 
		$ W \colon [0,\infty) \times \Omega \to \R^m $ 
	be a standard Brownian motion. 
	Then 
	\begin{enumerate}[(i)] 
		\item \label{existence_and_uniqueness_strange_lipschitz:item1}
		there exists a unique $ u \in \{ v \in C(\R^d,\R) \colon \inf_{r\in (0,\infty)} [\sup_{\norm{x}>r}(\frac{|v(x)|}{V(x)})] = 0 \} $ which satisfies that $u$ is a viscosity solution of 
			\begin{equation}
			\tfrac12 \operatorname{Trace}\!\big(BB^{*}(\operatorname{Hess}u)(x)\big)=f(x,u(x))
			\end{equation} 
		for $x\in\R^d$
		and 
		\item \label{existence_and_uniqueness_strange_lipschitz:item2}
		it holds for all 
			$ x \in \R^d $ 
		that 
			$ \EXPP{ \int_0^{\infty} e^{-\lambda t} | \lambda u(x+BW_t) - f(x+BW_t,u(x+BW_t)) | \,dt} < \infty $ 
		and  
			\begin{equation} 
			u(x) = \Exp{ \int_0^{\infty} e^{-\lambda t} ( \lambda u(x+BW_t) - f(x+BW_t,u(x+BW_t)) ) \,dt}\!. 
			\end{equation}
	\end{enumerate} 
\end{cor} 

\begin{proof}[Proof of \cref{cor:existence_and_uniqueness_strange_lipschitz}]
	Throughout this proof let 
		$ g \colon \R^d \times \R \to \R $ 
	satisfy for all 
		$ x \in \R^d $, 
		$ v \in \R $ 
	that 
		$ g(x,v) = \lambda v - f(x,v) $. 
	Note that the assumption that for all 
		$ x \in \R^d $, 
		$ v,w \in \R $ 
	it holds that 
		$ | f(x,v) - f(x,w) - \lambda (v-w) | \leq L | v - w | $ 
	ensures that for all 
		$ x \in \R^d $, 
		$ v,w \in \R $ 
	it holds that 
		$ | g(x,v) - g(x,w) | \leq L | v - w | $. 
	Moreover, observe that the assumption that 
		$ \inf_{r\in (0,\infty)} [\sup_{\norm{x}>r}(\frac{|f(x,0)|}{V(x)})] = 0 $ 
	implies that 
		$ \inf_{r\in (0,\infty)} 
		[\sup_{\norm{x}>r}(\frac{|g(x,0)|}{V(x)})] = 0 $. 
	In addition, note that for all 
		$ u \in C(\R^d,\R) $ 
	it holds that 
	\begin{equation} 
		\begin{gathered}
		\left( 
		\begin{array}{c} 
		u~\text{is a viscosity solution of}\\
		\tfrac12\operatorname{Trace}\!\big(BB^{*}(\operatorname{Hess} u)(x)\big) = f(x,u(x))\\
		\text{for}~x\in\R^d
		\end{array}
		\right) 
		\\
		\Leftrightarrow
		\\
		\left( 
		\begin{array}{c} 
		u~\text{is a viscosity solution of}\\
		\lambda u(x) - \tfrac12\operatorname{Trace}\!\big(BB^{*}(\operatorname{Hess} u)(x)\big) = g(x,u(x)) \\
			\text{for}~x\in\R^d
		\end{array}
		\right).   
		\end{gathered}
	\end{equation}	
	\cref{cor:existence_and_uniqueness} (with 
		$ d \is d $, 
		$ m \is m $, 
		$ B \is B $, 
		$ L \is L $, 
		$ \rho \is \rho $, 
		$ \lambda \is \lambda $,
		$ \norm{\cdot} \is \norm{\cdot} $, 
		$ f \is g $, 
		$ V \is V $, 
		$ (\Omega,\mc F,\P) \is (\Omega,\mc F,\P) $, 
		$ W \is W $
	in the notation of \cref{cor:existence_and_uniqueness}) therefore establishes Items~\eqref{existence_and_uniqueness_strange_lipschitz:item1} and \eqref{existence_and_uniqueness_strange_lipschitz:item2}. 
	The proof of \cref{cor:existence_and_uniqueness_strange_lipschitz} is thus completed. 	
\end{proof}

\begin{lemma} \label{lem:special_lyapunov_functions}
	Let $ d,m \in \N $, 
		$ B \in \R^{d\times m} $, 
		$ \varepsilon \in (0,\infty) $, 
		$ \norm{\cdot} \colon (\cup_{k\in\N} \R^k) \to [0,\infty) $ 
	satisfy for all 
		$ k \in \N $, 
		$ x = (x_1,\ldots,x_k) \in \R^k $ 
	that 
		$ \norm{x} = [\sum_{i=1}^k |x_i|^2]^{\nicefrac{1}{2}} $ 
	and let 
		$ V \colon \R^d \to (0,\infty) $ 
	satisfy for all 
		$ x \in \R^d $ 
	that 
		\begin{equation} 
		V(x) = \exp\!\left( \varepsilon (1+\norm{x}^2)^{\!\nicefrac12} \right)\!.
		\end{equation}  
	Then 
		\begin{enumerate}[(i)]
		\item \label{special_lyapunov_functions:item1}
		it holds for all 
			$ x \in \R^d $ 
		that 
			\begin{equation} 
			\frac{\norm{B^{*}(\nabla V)(x)}^2}{V(x)} 
			\leq \varepsilon^2 \left[\sup_{y\in\R^m\setminus\{0\}} \left(\frac{\norm{By}}{\norm{y}}\right)\right]^2 V(x) 
			\end{equation} 
		and
		\item \label{special_lyapunov_functions:item2}
		it holds for all 
			$ x \in \R^d $ 
		that 
			\begin{equation} 
			\begin{split} 
			\operatorname{Trace}\!\big( BB^{*}(\operatorname{Hess} V)(x) \big) 
			&\leq 
			(\varepsilon^2 + \varepsilon d)
			\left[\sup_{y\in\R^m\setminus\{0\}} \left(\frac{\norm{By}}{\norm{y}}\right)\right]^2 V(x). 
			\end{split} 
			\end{equation} 
	\end{enumerate}
\end{lemma}

\begin{proof}[Proof of \cref{lem:special_lyapunov_functions}]
	Throughout this proof let 	
		$ \HSnorm{\cdot} \colon \R^{d\times m} \to [0,\infty) $ 
	be the Frobenius norm on $ \R^{d\times m} $. 
	Observe that for all 
		$ x \in \R^d $ 
	it holds that 
		\begin{equation} \label{special_lyapunov_functions:derivatives}
		\begin{split}
		(\nabla V)(x) &= \frac{\varepsilon x}{(1+\norm{x}^2)^{\nicefrac12}} V(x)
		\qandq 
		\\
		(\operatorname{Hess} V)(x) &= \left[ \frac{\varepsilon^2 x\otimes x}{1+\norm{x}^2} 
		+ \frac{\varepsilon \operatorname{Id}_{\R^{d}}}{(1+\norm{x}^2)^{\nicefrac12}} 
		- \frac{\varepsilon x\otimes x}{(1+\norm{x}^2)^{\nicefrac32}}
		\right] V(x). 
		\end{split}
		\end{equation} 
	Hence, we obtain for all 
		$ x \in \R^d $ 
	that 
		\begin{equation} 
		\frac{\norm{B^{*}(\nabla V)(x)}^2}{V(x)} 
		= \frac{\varepsilon^2 \norm{B^{*}x}^2}{1+\norm{x}^2} V(x)
		\leq \varepsilon^2 \left[\sup_{y\in\R^m\setminus\{0\}} \left(\frac{\norm{By}}{\norm{y}}\right)\right]^2 V(x) .
		\end{equation} 	
	This establishes Item~\eqref{special_lyapunov_functions:item1}. 
	Moreover, note that \eqref{special_lyapunov_functions:derivatives} demonstrates for all 
		$ x \in \R^d $
	that 
		\begin{equation} 
		\begin{split} 
		\operatorname{Trace}\!\big( BB^{*}(\operatorname{Hess} V)(x) \big) 
		&= 
		\left[  
		\frac{\varepsilon^2\norm{B^{*}x}^2}{1+\norm{x}^2} + \frac{\varepsilon\HSnorm{B}^2}{(1+\norm{x}^2)^{\nicefrac12}} - 
		\frac{\varepsilon\norm{B^{*}x}^2}{(1+\norm{x}^2)^{\nicefrac32}}
		\right] V(x)
		\\
		&\leq 
		(\varepsilon^2 + \varepsilon d)
		\left[\sup_{y\in\R^m\setminus\{0\}} \left(\frac{\norm{By}}{\norm{y}}\right)\right]^2 V(x). 
		\end{split}
		\end{equation}
	This establishes Item~\eqref{special_lyapunov_functions:item2}. The proof of \cref{lem:special_lyapunov_functions} is thus completed. 
\end{proof}

\begin{cor} \label{cor:existence_and_uniqueness_polynomial_growing_nonlinearity}
	Let $ d,m \in \N $, 
		$ B \in \R^{d\times m} $, 
		$ L \in \R $, 
		$ \lambda \in (L,\infty) $,
		$ f \in C(\R^d\times\R,\R) $ 
	satisfy for all 
		$ x \in \R^d $, 
		$ v,w \in \R $ 
	that 
		$ | f(x,v) - f(x,w) - \lambda ( v - w ) | \leq L | v-w | $, 
	assume that $ f $ is at most polynomially growing, 
	let $ (\Omega,\mc F,\P) $ be a probability space, 
	and let 
	$ W \colon [0,\infty) \times \Omega \to \R^m $ 
	be a standard Brownian motion. 
	Then 
	\begin{enumerate}[(i)] 
		\item \label{existence_and_uniqueness_polynomial_growth_nonlinearity:item1}
		there exists a unique $ u \in \{ v \in C(\R^d,\R) \colon (\Forall\varepsilon\in (0,\infty)\colon
		[\sup_{x=(x_1,\ldots,x_d)\in\R^d} (\frac{|v(x)|}{\exp(\varepsilon \sum_{i=1}^d |x_i|)})] < \infty) \} $ which satisfies that 
		$ u $ is a viscosity solution of 
		\begin{equation} 
		\tfrac12\operatorname{Trace}\!\big(BB^{*}(\operatorname{Hess} u)(x)\big) = f(x,u(x)) 
		\end{equation}  
		for $ x \in \R^d $ and
		\item \label{existence_and_uniqueness_polynomial_growth_nonlinearity:item2}
		it holds for all 
			$ x \in \R^d $ 
		that 
			$ \EXPP{ \int_0^{\infty} e^{-\lambda t} | \lambda u(x+BW_t) - f(x+BW_t,u(x+BW_t)) | \,dt} < \infty $ 
		and 
			\begin{equation} 
			u(x) = \Exp{ \int_0^{\infty} e^{-\lambda t} ( \lambda u(x+BW_t) - f(x+BW_t,u(x+BW_t)) ) \,dt}\!. 
			\end{equation} 
	\end{enumerate} 
\end{cor} 

\begin{proof}[Proof of \cref{cor:existence_and_uniqueness_polynomial_growing_nonlinearity}] 
	Throughout this proof let $ \norm{\cdot} \colon \R^d \to [0,\infty) $ be the standard norm on $ \R^d $, 
	let 
		$ \beta, p \in [0,\infty) $ 
	satisfy  
		\begin{equation} 
		\sup_{x\in\R^d} \left(\frac{|f(x,0)|}{1+\norm{x}^p}\right) < \infty  
		\qandq 
		\beta = \tfrac12 \left[\sup_{x=(x_1,\ldots,x_m)\in\R^m\setminus\{0\}}\left(\frac{\norm{Bx}^2}{\sum_{i=1}^m |x_i|^2}\right)\right]\!,
		\end{equation}
	and let 
		$ V_{\varepsilon} \colon \R^d \to (0,\infty) $, $ \varepsilon \in (0,\infty) $, 
	satisfy for all 
		$ \varepsilon \in (0,\infty) $, 
		$ x \in \R^d $ 
	that 
		\begin{equation} \label{existence_and_uniqueness_polynomial_growth_nonlinearity:introducing_V_epsilon}
		V_{\varepsilon}(x) = \exp\!\left( \varepsilon (1+\norm{x}^2)^{\!\nicefrac12} \right)\!.
		\end{equation} 
	Observe that Item~\eqref{special_lyapunov_functions:item2} in \cref{lem:special_lyapunov_functions} (with 
		$ d \is d $, 
		$ m \is m $, 
		$ B \is B $, 
		$ \varepsilon \is \varepsilon $ 
	for $ \varepsilon \in (0,\infty) $ 
	in the notation of \cref{lem:special_lyapunov_functions}) demonstrates for all 
		$ \varepsilon \in (0,\infty) $, 
		$ x \in \R^d $ 
	that 
		\begin{equation} \label{existence_and_uniqueness_polynomial_growing_nonlinearity:lyapunov}
		\tfrac12\operatorname{Trace}\!\big( BB^{*} (\operatorname{Hess} V_{\varepsilon})(x) \big) 
		\leq 
		(\varepsilon^2+\varepsilon d)\beta V_{\varepsilon}(x). 
		\end{equation} 
	In addition, note that for all 
		$ \varepsilon \in (0,\infty) $, 
		$ x \in \R^d $  
	it holds that 
		\begin{equation} \label{existence_and_uniqueness_polynomial_growing_nonlinearity:equivalence}	
		\exp(\varepsilon \Norm{x}) 
		\leq 
		V_{\varepsilon}(x) 
		\leq 
		e^{\varepsilon} \exp(\varepsilon \Norm{x}). 
		\end{equation}		
	Moreover, note that for all 
		$ \varepsilon \in (0,\infty) $ 
	it holds that $ \sup_{r\in (0,\infty)} [ \inf_{\norm{x}>r} V_{\varepsilon}(x) ] = \infty $ and 
	\begin{equation} 
	\begin{split} 
	& \limsup_{r\to\infty} \left[ \sup_{\norm{x}>r} \left( \frac{|f(x,0)|}{V_{\varepsilon}(x)} \right) \right] 
	= 
	\limsup_{r\to\infty} \left[ \sup_{\norm{x}>r} \left( \frac{|f(x,0)|}{1+\norm{x}^p} \frac{1+\norm{x}^p}{V_{\varepsilon}(x)} \right) \right] 
	\\
	& \leq 
	\limsup_{r\to\infty} \left( \left[ \sup_{x\in\R^d}  \left( \frac{|f(x,0)|}{1+\norm{x}^p} \right) \right]  \left[ \sup_{\norm{x}>r} \left( \frac{1+\norm{x}^p}{V_{\varepsilon}(x)} \right) \right]\right) 
	\\
	& \leq 
	\left[ \sup_{x\in\R^d}  \left( \frac{|f(x,0)|}{1+\norm{x}^p} \right) \right]
	\left[ \limsup_{r\to\infty} \left( \sup_{\norm{x}>r} \left( \frac{1+\norm{x}^p}{\exp(\varepsilon\Norm{x})} \right) \right) \right] 
	= 0.  
	\end{split} 
	\end{equation} 
	This, \eqref{existence_and_uniqueness_polynomial_growing_nonlinearity:lyapunov}, \eqref{existence_and_uniqueness_polynomial_growing_nonlinearity:equivalence}, and \cref{cor:existence_and_uniqueness_strange_lipschitz} (with 
		$ d \is d $, 
		$ m \is m $, 
		$ B \is B $, 
		$ L \is L $, 
		$ \rho \is (\varepsilon^2+\varepsilon d) \beta $, 
		$ \lambda \is \lambda $, 
		$ f \is f $, 
		$ V \is V_{\varepsilon} $, 
		$ (\Omega,\mc F,\P) \is (\Omega,\mc F,\P) $, 
		$ W \is W $ 
	for $ \varepsilon \in (0,\infty) $
	in the notation of \cref{cor:existence_and_uniqueness_strange_lipschitz}) 
	ensure for every 
		$ \varepsilon \in (0,\infty) $ 
	with 
		$ (\varepsilon^2+\varepsilon d) \beta < \lambda - L $ 
	that there exists a unique 
		$ u_{\varepsilon} \in \{ v \in C(\R^d,\R) \colon \inf_{r\in (0,\infty)} [\sup_{\norm{x}>r} (\frac{|v(x)|}{V_{\varepsilon}(x)})] = 0 \} $ 
	which satisfies that 
		$ u_{\varepsilon} $ is a viscosity solution of 
	\begin{equation} \label{existence_and_uniqueness_polynomial_growth_nonlinearity:u_epsilon_viscosity_solution}
		\tfrac12 \operatorname{Trace}(BB^{*}(\operatorname{Hess} u_{\varepsilon})(x)\big) = f(x,u_{\varepsilon}(x))
	\end{equation} 
	for $ x\in \R^d $. \cref{cor:existence_and_uniqueness_strange_lipschitz} hence assures for all 
		$ \varepsilon \in (0,\infty) $, 
		$ x \in \R^d $
	with 
		$ (\varepsilon^2+\varepsilon d) \beta < \lambda - L $ 
	that 
		\begin{equation} \label{existence_and_uniqueness_polynomial_growth_nonlinearity:fixpoint_equation}
		u_{\varepsilon}(x) = \Exp{ \int_0^{\infty} e^{-\lambda t} ( \lambda u_{\varepsilon}(x+BW_t) - f(x+BW_t,u_{\varepsilon}(x+BW_t)) ) \,dt}\!. 
		\end{equation} 
	Next note that the fact that for all 
		$ \varepsilon \in (0,\infty) $, 
		$ \eta \in (0,\varepsilon) $, 
		$ x \in \R^d $ 
	it holds that 
		$ V_{\eta}(x) \leq V_{\varepsilon}(x) $
	and \eqref{existence_and_uniqueness_polynomial_growth_nonlinearity:u_epsilon_viscosity_solution} ensure that for all 
		$ \varepsilon \in (0,\infty) $, 
		$ \eta \in (0,\varepsilon) $, 
		$ x \in \R^d $ 
	with 
		$ (\varepsilon^2+\varepsilon d) \beta < \lambda - L $ 
	it holds that 
		$ u_{\varepsilon}(x) = u_{\eta}(x) $. 
	Hence, we obtain that there exists 
		$ \mf u \colon \R^d \to \R $ 
	which satisfies for all 
		$ \varepsilon \in (0,\infty) $, 
		$ x \in \R^d $ 
	with 
		$ (\varepsilon^2+\varepsilon d) \beta < \lambda - L $ 
	that 
		$ \mf u(x) = u_{\varepsilon}(x) $. 
	This and \eqref{existence_and_uniqueness_polynomial_growth_nonlinearity:u_epsilon_viscosity_solution} ensure that for every 
		$ v  \in \{ w \in C(\R^d,\R) \colon (\Forall\varepsilon\in (0,\infty) \colon \sup_{x\in\R^d}[(\frac{|w(x)|}{V_{\varepsilon}(x)})] < \infty) \} $ 
	which satisfies that $ v $ is a viscosity solution of 
		\begin{equation} 
		\tfrac12 \operatorname{Trace}(BB^{*}(\operatorname{Hess} v)(x)) = f(x,v(x)) 
		\end{equation}  
	for $ x \in \R^d $ 
	it holds 
	that 
		$ v = \mf u $. 
	This and \eqref{existence_and_uniqueness_polynomial_growth_nonlinearity:u_epsilon_viscosity_solution} establish Item~\eqref{existence_and_uniqueness_polynomial_growth_nonlinearity:item1}. Next note that  \eqref{existence_and_uniqueness_polynomial_growth_nonlinearity:fixpoint_equation} and the fact that there exist 
		$ \varepsilon_n \in (0,\infty) $, $ n \in \N $, 
	with $ \limsup_{n\to\infty} \varepsilon_n = 0 $ such that $ \mf u = u_{\varepsilon_n} $ establish Item~\eqref{existence_and_uniqueness_polynomial_growth_nonlinearity:item2}. The proof of \cref{cor:existence_and_uniqueness_polynomial_growing_nonlinearity} is thus completed. 	
\end{proof} 

\subsection{A priori estimates for solutions of SFPEs}
\label{subsec:a_priori_estimate}

\begin{prop}[A priori estimate] \label{lem:a_priori_estimate_exact_solution} 
	Let $ d,m \in \N $, 
		$ B \in \R^{d\times m} $, 
		$ L \in \R $, 
		$ \lambda \in ( L, \infty) $, 
		$ \varepsilon \in (\nicefrac{L^2}{\lambda},\infty) $, 
	let $(\Omega,\mathcal{F},\P)$ be a probability space, 
	let $ W \colon [0,\infty) \times \Omega \to \R^m $ be a standard Brownian motion,  
	let 
		$ f \colon \R^d\times\R \to \R $ 
	be  $ \Borel(\R^d\times\R)$/$\Borel(\R)$-measurable, 
	assume for all 
		$ x \in \R^d $, 
		$ v,w \in \R $ 
	that 
		$ |f(x,v)-f(x,w)| \leq L | v -  w | $ 
	and 
		$ \int_0^{\infty} e^{-\lambda t} \,
		\EXP{|f(x+BW_t,0)|}\,dt < \infty $,
	let $ u \colon \R^d \to \R $ be $ \Borel(\R^d) $/$ \Borel(\R) $-measurable, 
	and assume for all 
		$ x \in \R^d $ 
	that 
		$ \int_0^{\infty} e^{(\varepsilon-\lambda) t}\,\EXP{|u(x+BW_t)|^2}\,dt < \infty $
	and 		
	\begin{equation} \label{a_priori_estimate:fixed_point_equation}
	u(x) 
	=
	\int_0^{\infty} e^{-\lambda t} \, \Exp{f(x+BW_t, u(x+BW_t) )}\!\,dt .
	\end{equation} 
	Then 
	\begin{enumerate} [(i)]
		\item \label{a_priori_estimate:item1} 
		it holds for all 
			$ t \in [0,\infty) $, 
			$ x \in \R^d $ 
		that
			\begin{equation} 
			e^{-\lambda t} \, \Exp{|u(x+BW_t)|^2}
			\leq 
			\frac{1}{\lambda}
			\int_t^{\infty} e^{-\lambda s} \, \Exp{|f(x+BW_s,u(x+BW_s))|^2}\!\,ds, 
			\end{equation}
		\item \label{a_priori_estimate:item2}
		it holds for all 
			$ x \in \R^d $ 
		that  
			\begin{equation} \label{a_priori_estimate:claim}
			\begin{split} 
			& \left[ 
			\int_0^{\infty} e^{(\varepsilon-\lambda)t} \, \Exp{|u(x+BW_t)|^2}\!\,dt \right]^{\nicefrac12}
			\\
			& \leq 
			\frac{1}{\sqrt{\varepsilon\lambda}-L}
			\left[\int_0^{\infty}
			e^{(\varepsilon-\lambda)s} \,	\Exp{|f(x+BW_{s},0)|^2}\!\,ds\right]^{\nicefrac12}, 
			\end{split} 
			\end{equation} 
		and 
		\item \label{a_priori_estimate:item3} 
		it holds for all 
			$ x \in \R^d $ 
		that 	
			\begin{equation} \label{a_priori_estimate:addition}
			| u(x) | 
			\leq 
			\frac{\sqrt{\varepsilon}}{\sqrt{\varepsilon\lambda} - L} \left[ 
			\int_0^{\infty} e^{(\varepsilon-\lambda) t } \, \Exp{ | f(x+BW_t,0) |^2 }\!\,dt
			\right]^{\nicefrac12}. 
			\end{equation}
	\end{enumerate} 
\end{prop}

\begin{proof}[Proof of \cref{lem:a_priori_estimate_exact_solution}]
	Throughout this proof let 
		$ \nu_{t,x} \colon \Borel(\R^d) \to [0,1] $, $ t \in [0,\infty) $, $ x \in \R^d $, 
	satisfy for all 
		$ t \in [0,\infty) $, 
		$ x \in \R^d $, 
		$ A \in \Borel(\R^d) $
	that 
		$ \nu_{t,x}(A) = \P(x + BW_t \in A) $. 
	Observe that \eqref{a_priori_estimate:fixed_point_equation} and the Cauchy-Schwarz inequality ensure for all 
		$ x \in \R^d $ 
	that 
		\begin{equation} 
		\begin{split}
		|u(x)|^2 
		& \leq 
		\left[ \int_0^{\infty} e^{-\lambda t} \,dt \right]
		\left[ 
		\int_0^{\infty} e^{-\lambda t}
		\left| 
		\Exp{ f(x+BW_t,u(x+BW_t)) }
		\right|^2 \,dt \right]
		\\
		& \leq 
		\frac{1}{\lambda} 
		\int_0^{\infty} e^{-\lambda t} \,
		\Exp{ | f(x+BW_t,u(x+BW_t)) |^2}\!\,dt. 
		\end{split}
		\end{equation} 
	This, Fubini's theorem, and the fact that for all 
		$ t, s\in [0,\infty) $, 
		$ x \in \R^d $, 
		$ A \in \Borel(\R^d) $ 
	it holds that 
		$ 
		\nu_{t+s,x}(A) = \int_{\R^d} \nu_{s,y}(A)\,\nu_{t,x}(dy) $	
	ensure for all 
		$ t \in [0,\infty) $ 
		$ x \in \R^d $
	that 
		\begin{equation} 
		\begin{split}
		\Exp{|u(x+BW_t)|^2} 
		& = 
		\int_{\R^d} |u(y)|^2\,\nu_{t,x}(dy) 
		\\
		& 
		\leq 
		\frac{1}{\lambda} \int_{\R^d} \int_0^{\infty} e^{-\lambda s}\, \Exp{|f(y+BW_{s},u(y+BW_s))|^2}\!\,ds \,\nu_{t,x}(dy)
		\\
		& = 
		\frac{1}{\lambda} \int_0^{\infty} 
		e^{-\lambda s} \, \Exp{|f(x+BW_{t+s},u(x+BW_{t+s}))|^2}\!\,ds 
		\\
		& = 
		\frac{e^{\lambda t}}{\lambda} \int_t^{\infty} 
		e^{-\lambda s} \, \Exp{|f(x+BW_{s},u(x+BW_{s}))|^2}\!\,ds. 
		\end{split}
		\end{equation}
	This establishes Item~\eqref{a_priori_estimate:item1}. 
	Combining Item~\eqref{a_priori_estimate:item1} with Fubini's theorem ensures for all 
		$ x \in \R^d $ 
	that 
		\begin{equation} 
		\begin{split}
		\int_0^{\infty} e^{(\varepsilon-\lambda)t} \,
		\Exp{|u(x+BW_t)|^2} \!\,dt 
		& \leq 
		\int_0^{\infty} \frac{e^{\varepsilon t} }{\lambda} \, \int_t^{\infty} e^{-\lambda s} \, \Exp{| f(x+BW_s,u(x+BW_s))|^2 } \!\,ds \,dt
		\\
		& = 
		\int_0^{\infty} \left[ \int_0^s 
		\frac{e^{\varepsilon t} }{\lambda}
		\,dt \right] e^{-\lambda s}
		\,\Exp{|f(x+BW_s,u(x+BW_s))|^2} \!\,ds
		\\
		& \leq 
		\frac{1}{\varepsilon\lambda}
		\int_0^{\infty}  
		e^{(\varepsilon-\lambda)s} \,
		\Exp{|f(x+BW_s,u(x+BW_s))|^2}
		\!\,ds . 
		\end{split}
		\end{equation} 
	Minkowski's inequality hence shows for all 
		$ x \in \R^d $ 
	that 
		\begin{align}
		\nonumber
		& \left[\int_0^{\infty} e^{(\varepsilon-\lambda)t} \,
		\Exp{|u(x+BW_t)|^2} \!\,dt \right]^{\nicefrac12} 
		\\ 
		& \leq 
		\frac{1}{\sqrt{\varepsilon\lambda}} 
		\left[ \int_0^{\infty}  
		e^{(\varepsilon-\lambda)s} \,
		\Exp{|f(x+BW_s,u(x+BW_s))|^2}
		\!\,ds \right]^{\nicefrac12} 
		\\ \nonumber
		& \leq 
		\frac{1}{\sqrt{\varepsilon\lambda}} 
		\left[ \int_0^{\infty}  
		e^{(\varepsilon-\lambda)s} \,
		\Exp{|f(x+BW_s,0)|^2}
		\!\,ds \right]^{\nicefrac12} 
		+ 
		\frac{L}{\sqrt{\varepsilon\lambda}} 
		\left[ \int_0^{\infty}  
		e^{(\varepsilon-\lambda)s} \,
		\Exp{|u(x+BW_s)|^2}
		\!\,ds \right]^{\nicefrac12}.  
		\end{align}
	This implies for all 
		$ x \in \R^d $ 
	that 
		\begin{equation} 
		\left[ 
		\int_0^{\infty} e^{(\varepsilon-\lambda)t}\,\Exp{|u(x+BW_t)|^2}\!\,dt
		\right]^{\nicefrac{1}{2}} 
		\leq 
		\frac{1}{\sqrt{\varepsilon\lambda}-L} 
		\left[ 
		\int_0^{\infty} e^{(\varepsilon-\lambda)t} \, 
		\Exp{|f(x+BW_t,0)|^2}\!\,dt
		\right]^{\nicefrac12}. 
		\end{equation} 
	This establishes Item~\eqref{a_priori_estimate:item2}. 
	Next observe that the triangle inequality, Fubini's theorem, the assumption that for all 
		$ x \in \R^d $, 
		$ a,b \in \R $ 
	it holds that 
		$ | f(x,a) - f(x,b) | \leq L |a-b| $, 
	and \eqref{a_priori_estimate:fixed_point_equation} ensure for all 
		$ x \in \R^d $ 
	that 
		\begin{equation} 
		| u(x) | \leq \int_0^{\infty} e^{-\lambda t } \, \Exp{ | f(x+BW_t,0) | }\!\,dt 
		+ L \int_0^{\infty} e^{-\lambda t} \, \Exp{ |u(x+BW_t)| } \!\,dt.  
		\end{equation} 
	The Cauchy-Schwarz inequality and Item~\eqref{a_priori_estimate:item2} hence prove for all 
		$ x \in \R^d $ 
	that 
		\begin{equation} 
		\begin{split} 
		| u(x) | & \leq 
		\frac{1}{\sqrt{\lambda}} \left[ \int_0^{\infty} e^{-\lambda t } \, \Exp{ | f(x+BW_t,0) |^2 }\!\,dt \right]^{\nicefrac12} 
		+ \frac{L}{\sqrt{\lambda}} \left[ \int_0^{\infty} e^{-\lambda t} \, \Exp{ | u(x+BW_t) |^2 }\!\,dt \right]^{\nicefrac12} 
		\\
		& \leq 
		\frac{1}{\sqrt{\lambda}} \left[ 1 + \frac{L}{\sqrt{\varepsilon\lambda}-L} \right] 
		\left[  \int_0^{\infty} e^{(\varepsilon-\lambda) t } \, \Exp{ | f(x+BW_t,0) |^2 }\!\,dt \right]^{\nicefrac12}
		\\
		& = 
		\frac{\sqrt{\varepsilon}}{\sqrt{\varepsilon\lambda} - L} \left[ 
		\int_0^{\infty} e^{(\varepsilon-\lambda) t } \, \Exp{ | f(x+BW_t,0) |^2 }\!\,dt
		\right]^{\nicefrac12}.  
		\end{split}
		\end{equation}  
	This establishes Item~\eqref{a_priori_estimate:item3}. 
	The proof of \cref{lem:a_priori_estimate_exact_solution} is thus completed. 
\end{proof} 

\section{Full-history recursive multilevel Picard (MLP) approximations}
\label{sec:mlp}

In this section we introduce and analyze MLP approximation schemes for SFPEs related to semilinear elliptic PDEs (see \cref{setting} in \cref{subsec:setting_mlp_scheme} below). 
In \crefrange{subsec:mlp_measurability}{subsec:mlp_integrability} we establish some rather technical results concerning measurability, distributional, and integrability properties of MLP approximations. 
\cref{subsec:bias_variance_estimates} contains fundamental estimates for the biases and variances of MLP approximations (see \cref{lem:bias,lem:variance} below). 
These fundamental estimates for the biases and variances of MLP approximations are merged in \cref{prop:combined_estimate,cor:L_zero_case} in \cref{subsec:convergence_analysis_mlp} below to obtain a full error analysis for MLP approximations in \cref{cor:all_cases} in \cref{subsec:convergence_analysis_mlp} below. 
In \cref{subsec:computational_effort} we provide an upper bound for the computational cost to sample realizations of MLP approximations. 
In \cref{subsec:complexity_analysis} we finally relate the error analysis for MLP approximations established in \cref{cor:all_cases} in \cref{subsec:convergence_analysis_mlp} with the computational cost analysis for MLP approximations established in \cref{subsec:computational_effort} to obtain in \cref{prop:overcoming_the_curse} in \cref{subsec:complexity_analysis} an overall complexity analysis for the proposed MLP approximation schemes.

\subsection{MLP approximations} \label{subsec:setting_mlp_scheme}

\begin{setting} \label{setting}
Let $ d, M \in \N $, 
	$ \lambda \in (0,\infty) $, 
	$ \Theta = \cup_{n\in\N} \Z^n $, 
	$ B \in \R^{d\times d} $,
	$ f \in C( \R^d \times \R, \R ) $, 
let $ ( \Omega, \mc F, \P ) $ be a probability space, 
let	$ W^{\theta} \colon [0,\infty) \times \Omega \to \R^d $, $ \theta \in \Theta $, be i.i.d.~standard Brownian motions, 
let $R^{\theta}\colon\Omega\to [0,\infty)$, $\theta\in\Theta$, 
	be i.i.d.~random variables which satisfy for all 
	$t\in [0,\infty)$ that $\P(R^{0}\geq t)=e^{-\lambda t}$, 
assume that $(R^{\theta})_{\theta\in\Theta}$ and  	
	$(W^{\theta})_{\theta\in\Theta}$ are independent, 
and let $ U^{\theta}_{n} = ( U^{\theta}_n (x) )_{ x \in \R^d } \colon \R^d \times \Omega \to \R $, 
	$ \theta \in \Theta $, $ n \in \N_0 $, satisfy for all 
	$ \theta \in \Theta $, $ n \in \N $, $ x \in \R^d $
	that $ U^{\theta}_{0}(x)=0 $ and 
	\begin{equation} \label{mlp_scheme}
	\begin{split}
		& 
		U^{\theta}_{n}(x) 
		= 
		\sum_{k=1}^{n-1} \frac{1}{\lambda M^{(n-k)}}
		\Bigg[ 
		\sum_{m=1}^{M^{(n-k)}}
		\Big( 
			f\big(x+BW^{(\theta,k,m)}_{R^{(\theta,k,m)}},U^{(\theta,k,m)}_{k}(
				x+BW^{(\theta,k,m)}_{R^{(\theta,k,m)}})\big) 
		\\
		& 	
		- 
		f\big(x+BW^{(\theta,k,m)}_{R^{(\theta,k,m)}},U^{(\theta,k,-m)}_{k-1}(x+BW^{(\theta,k,m)}_{R^{(\theta,k,m)}})\big)
		\Big)
		\Bigg]
		+
		\frac{1}{\lambda M^n}\left[ 
		\sum_{m=1}^{M^n} f\big(x+BW^{(\theta,0,m)}_{R^{(\theta,0,m)}},0\big)
		\right]\!.
	\end{split}	
	\end{equation} 
\end{setting}

\subsection{Measurability properties of MLP approximations}
\label{subsec:mlp_measurability}
	
\begin{lemma} \label{lem:properties_of_mlp} 
	Assume \cref{setting}. Then 
	\begin{enumerate}[(i)] 
	\item \label{properties_of_mlp:item1}
		it holds for all 
			$ n \in \N_0 $, 
			$ \theta \in \Theta $ 
		that 
			$ U^{\theta}_n \colon \R^d \times \Omega \to \R $ 
		is a continuous random field, 
	\item \label{properties_of_mlp:item2}
		it holds for all 
			$ n \in \N_0 $, 
			$ \theta \in \Theta $ 
		that 
			$ \sigma_{\Omega}(U^{\theta}_n) \subseteq \sigma_{\Omega}((R^{(\theta,\vartheta)})_{\vartheta\in\Theta},(W^{(\theta,\vartheta)})_{\vartheta\in\Theta})$,
	\item \label{properties_of_mlp:item3}
		it holds that 
			$ \R^d\times\Omega \ni (x,\omega) \mapsto  x+BW^{\theta}_{R^{\theta}(\omega)}(\omega) \in \R^d $, $ \theta \in \Theta $, 
		are identically distributed random fields, and
	\item \label{properties_of_mlp:item4}		
		it holds for all 
			$ n\in\N_0 $ 
		that 
			$ U^{\theta}_n $, $ \theta \in \Theta $, 
		are identically distributed random fields. 
	\end{enumerate} 
\end{lemma} 

\begin{proof}[Proof of \cref{lem:properties_of_mlp}]
	First, we prove Item~\eqref{properties_of_mlp:item1}. For this let $ \A $ and $ \B $ be the sets given by
	\begin{equation} 
	\A = \left\{n \in \N_0\colon 
	\left( 
	\begin{array}{c}
	\text{For all}~\theta\in\Theta~\text{it holds that} \\
	~U^{\theta}_n\colon\R^d\times\Omega\to\R~\text{is a continuous random field}
	\end{array}
	\right)\right\}
	\end{equation} 
	and $\B = \left\{ n \in \N \colon \{0,1,\ldots,n-1\} \subseteq \A \right\}$. 
	Observe that the assumption that for all 
		$ \theta \in \Theta $, $ x \in \R^d $ 
	it holds that 
		$ U^{\theta}_0(x) = 0 $ 
	shows that $ 0 \in \A $. Hence, we obtain that $ 1 \in \B $. Next note that the assumption that $ W^{\theta} $, $ \theta \in \Theta $, are Brownian motions, the assumption that $ f \in C(\R^d\times\R,\R) $, and \eqref{mlp_scheme} demonstrate for all 
		$ n \in \B $, $ \theta \in \Theta $ 
	that 
		$ U^{\theta}_n \colon \R^d \times \Omega \to \R $ 
	is a continuous random field. Hence, we obtain for every 
		$ n \in \B $ 
	that 
		$ n \in \A $. 
	This shows for all 
		$ n \in \N $ 
	that $(n \in \B \Rightarrow n+1 \in \B)$. 
	Combining this with the fact that 
		$ 1 \in \B $ 
	and induction demonstrates that $ \N \subseteq \B $. 
	Hence, we obtain that $ \N_0 \subseteq \A $. 
	This establishes Item~\eqref{properties_of_mlp:item1}. 
	Next we prove Item~\eqref{properties_of_mlp:item2}. For this let $ \A $ and $ \B $ be the sets given by 
		\begin{equation} 
		\A = \left\{ n \in \N_0 \colon 
		\left( 
		\begin{array}{c} 
		\text{For all}~\theta\in\Theta~\text{it holds that} \\
		\sigma_{\Omega}(U^{\theta}_n)\subseteq\sigma_{\Omega}((R^{(\theta,\vartheta)})_{\vartheta\in\Theta},(W^{(\theta,\vartheta)})_{\vartheta\in\Theta}) 
		\end{array}  
		\right) \right\}
		\end{equation}
	and $ \B = \{ n \in \N\colon \{0,1,\ldots,n-1\} \subseteq \A \} $. Observe that the assumption that for all 
		$ \theta \in \Theta $, 
		$ x \in \R^d $ 
	it holds that 
		$ U^{\theta}_0(x) = 0 $ 
	shows that $ 0 \in \A $. This implies that $ 1 \in \B $. Next  note that Item~\eqref{properties_of_mlp:item1} ensures for all 
		$ n \in \N_0 $, 
		$ \theta \in \Theta $ 
	that 
		$ U^{\theta}_n \colon \R^d\times\Omega \to \R $ 
	is $ (\Borel(\R^d) \otimes \sigma_{\Omega}(U^{\theta}_n)) $/$ \Borel(\R) $-measurable. Combining this with the fact that 
	for every 
		$ \theta \in \Theta $ 
	it holds that 
		$ W^{\theta} \colon [0,\infty) \times \Omega \to \R^d $ 
	is $ (\Borel([0,\infty)) \otimes \sigma_{\Omega}(W^{\theta})) $/$ \Borel(\R^d) $-measurable, 
	the fact that for every 
		$ \theta \in \Theta $ 
	it holds that 
		$ R^{\theta} \colon \Omega \to [0,\infty) $ 
	is $ \sigma_{\Omega}(R^{\theta}) $/$ \Borel([0,\infty)) $-measurable, 
	the assumption that $ f \in C(\R^d\times\R,\R) $, and  \eqref{mlp_scheme} demonstrates for all 
		$ n \in \B $, 
		$ \theta \in \Theta $	
	that 
		$ U^{\theta}_n \colon \R^d \times \Omega \to \R $ 
	is $ (\Borel(\R^d) \otimes \sigma_{\Omega}( (R^{(\theta,\vartheta)})_{\vartheta\in\Theta} ,(W^{(\theta,\vartheta)})_{\vartheta\in\Theta} ) ) $/$ \Borel(\R) $-measurable. This shows that for every 
		$ n \in \B $ it holds that 
		$ n \in \A $. 
	Hence, we obtain for all 
		$ n \in \N $ 
	that 
		$(n \in \B \Rightarrow n+1 \in \B)$. 
	This, the fact that $ 1 \in \B $, and induction demonstrate that 
		$ \N \subseteq \B $. 
	Therefore, we obtain that $ \N_0 \subseteq \A $. This establishes Item~\eqref{properties_of_mlp:item2}. 
	Next we prove Item~\eqref{properties_of_mlp:item3}. For this note that Hutzenthaler et al.~\cite[Corollary 2.5]{Overcoming} (with 
		$U_1 \is ([0,\infty) \times \Omega \ni (t,\omega) \mapsto \varphi(W^{\theta}_t(\omega)) \in \R )$, 
		$U_2 \is ([0,\infty) \times \Omega \ni (t,\omega) \mapsto
		\varphi(W^0_t(\omega)) \in \R )$, 
		$Y_1 \is (\R^d\times \Omega \ni (x,\omega) \mapsto R^{\theta}(\omega) \in [0,\infty) )	$, 
		$Y_2 \is (\R^d\times \Omega \ni (x,\omega) \mapsto 
		R^0(\omega) \in [0,\infty) ) $
	for $ \theta \in \Theta $, 
		$ \varphi = (\R^d \ni (x_1,x_2,\dots,x_d) \mapsto \sum_{i=1}^d a_ix_i \in \R) $ 
	for $ a_1,a_2,\ldots,a_d \in \R $
	in the notation of Hutzenthaler et al.~\cite[Corollary 2.5]{Overcoming}) 
	ensures for all 
		$ \theta \in \Theta $ 
	that 	
		$ W^{\theta}_{R^{\theta}} $ and $ W^{0}_{R^0} $ are identically distributed random variables. 
	Hence, we obtain for all 
		$ \theta \in \Theta $ 
	that 
		$ \R^d \times \Omega \ni (x,\omega) \mapsto x + BW^{\theta}_{R^{\theta}(\omega)}(\omega) \in \R^d $ 
	and 
		$ \R^d \times \Omega \ni (x,\omega) \mapsto x + BW^{0}_{R^{0}(\omega)}(\omega) \in \R^d $ 
	are identically distributed random fields. This establishes Item~\eqref{properties_of_mlp:item3}. 
	Next we prove Item~\eqref{properties_of_mlp:item4}. For this 
	let $ \A $ and $ \B $ be the sets given by 
		\begin{equation} 
		\A = \left\{ n \in \N_0 \colon 
		\left( \begin{array}{c} 
		U^{\theta}_n\colon \R^d\times\Omega\to\R, \theta\in\Theta,\\
		\text{are identically distributed random fields} 
		\end{array} 
		\right) 
		\right\}
		\end{equation} 
	and $ \B = \{ n \in \N \colon \{0,1,\ldots,n-1\} \subseteq \A \} $. Observe that the assumption that for all 
		$ \theta\in\Theta $, 
		$ x \in \R^d $ 
	it holds that 
		$ U^{\theta}_0(x) = 0 $ 
	shows that $ 0 \in \A $. This implies that $ 1 \in \B $. Next note that Item~\eqref{properties_of_mlp:item1}, 
	Item~\eqref{properties_of_mlp:item2}, 
	Item~\eqref{properties_of_mlp:item3}, 
	and Hutzenthaler et al.~\cite[Corollary 2.5]{Overcoming} ensure for all 
		$ n \in \B $, 
		$ \theta \in \Theta $ 
	that $ U^{\theta}_n $ and $ U^{0}_n $ are identically distributed random fields. This demonstrates that for all 
		$ n \in \B $ 
	it holds that 
		$ n \in \A $. 
	Hence, we obtain that for all 
		$ n \in \N $ 
	it holds that 
		$ (n\in\B \Rightarrow n+1\in\B)$.
	This, the fact that $1\in\B$, and induction establish Item~\eqref{properties_of_mlp:item4}. The proof of \cref{lem:properties_of_mlp} is thus completed. 
\end{proof}

\subsection{Integrability properties of MLP approximations} 
\label{subsec:mlp_integrability}

\begin{lemma} \label{lem:elementary_integration_rules}
	Assume \cref{setting}, 
	let $ p \in [1,\infty) $, 
		$ k \in \N_0 $, 
		$ \theta,\vartheta \in \Theta $, 
	assume that 
		$ \theta \notin \{(\vartheta,\eta)\colon \eta \in \Theta \} $, 
	and let 
		$ \nu_{t,x} \colon \Borel(\R^d) \to [0,\infty) $, $ t\in [0,\infty) $, $ x \in \R^d $, 
	satisfy for all 
		$ t \in [0,\infty) $, 
		$ x \in \R^d $, 
		$ A \in \Borel(\R^d) $  
	that 
		$ \nu_{t,x}(A) = \P( x + BW^0_t \in A ) $. 
	Then 
		\begin{enumerate} [(i)]
		\item \label{elementary_integration_rules:item1} 
		it holds for all $ x \in \R^d $ that 
			\begin{equation} 
			\begin{split}
			\Exp{| U^{\vartheta}_k(x + BW^{\theta}_{R^{\theta}}) |^p} 
			& = 
			\int_0^{\infty} \Exp{ |U^{\vartheta}_k(x+BW^{\theta}_s)|^p}\!\,\big( R^{\theta}(\P) \big)(ds)
			\\
			& =
			\int_0^{\infty} \lambda e^{-\lambda s} \, \Exp{ |U^{\vartheta}_k(x+BW^{\theta}_s)|^p}\!\,ds 
			\end{split}
			\end{equation}  
		and
		\item \label{elementary_integration_rules:item2} 
		it holds for all $ x \in \R^d $, $ t \in [0,\infty) $ that
			\begin{equation} 
			\Exp{|U^{\vartheta}_k(x+BW^{\theta}_t)|^p} 
			= 
			\int_{\R^d} 
			\Exp{|U^0_k(y)|^p} \!\,
			\nu_{t,x}(dy). 
			\end{equation} 
		\end{enumerate} 
\end{lemma} 

\begin{proof}[Proof of \cref{lem:elementary_integration_rules}]
	First, observe that Item~\eqref{properties_of_mlp:item2} in \cref{lem:properties_of_mlp} and the fact that $ W^{\theta} \colon [0,\infty) \times \Omega \to \R^d $ is $ (\Borel([0,\infty))\otimes \sigma_{\Omega}(W^{\theta}) )$/$ \Borel(\R^d) $-measurable ensure for all 
		$ x \in \R^d $ 
	that 
		\begin{equation} 
		[0,\infty) \times \Omega \ni (t,\omega) \mapsto 
		|U^{\vartheta}_k(x+BW^{\theta}_t(\omega),\omega)|^p \in [0,\infty) 
		\end{equation}  
	is 
		$ (\Borel([0,\infty)) \otimes \sigma_{\Omega}(W^{\theta}, (W^{(\vartheta,\eta)})_{\eta\in\Theta},(R^{(\vartheta,\eta)})_{\eta\in\Theta}) ) $/$ \Borel( [0,\infty) ) $-measurable. 
	In addition, note that the assumption that $ \theta \notin \{(\vartheta,\eta)\colon \eta\in\Theta\} $ proves that $ R^{\theta} $ is independent of $ \sigma_{\Omega}(W^{\theta}, (W^{(\vartheta,\eta)})_{\eta\in\Theta},(R^{(\vartheta,\eta)})_{\eta\in\Theta}) $. 
	Hutzenthaler et al.~\cite[Lemma 2.2]{Overcoming} (with 
		$ (\Omega,\mc{F},\P) \is (\Omega,\mc{F},\P) $, 
		$ \mc{G} \is \sigma_{\Omega}(W^{\theta}, (W^{(\vartheta,\eta)})_{\eta\in\Theta},(R^{(\vartheta,\eta)})_{\eta\in\Theta}) $, 
		$ (S,\mc{S}) \is ([0,\infty),\Borel([0,\infty))) $, 
		$ U \is ( 	[0,\infty) \times \Omega \ni (t,\omega) \mapsto 
		|U^{\vartheta}_k(x+BW^{\theta}_t(\omega),\omega)|^p \in [0,\infty)  ) $, $ Y \is R^{\theta} $  
	for $ x \in \R^d $
	in the notation of Hutzenthaler et al.~\cite[Lemma 2.2]{Overcoming}) therefore demonstrates for all 
		$ x \in \R^d $ 
	that 
		\begin{equation} 
		\Exp{\left| U^{\vartheta}_k(x+BW^{\theta}_{R^{\theta}}) \right|^p} 
		= 
		\int_0^{\infty} \Exp{ \left| U^{\vartheta}_k(x+BW^{\theta}_{t}) \right|^p} \!\, \big( R^{\theta}(\P) \big)(dt). 
		\end{equation} 
	This and the assumption that for all $ t \in (0,\infty) $ it holds that $ \P(R^{\theta} \geq t) = e^{-\lambda t} $ establish Item~\eqref{elementary_integration_rules:item1}. Next observe that the assumption that $ \theta \notin \{(\vartheta,\eta) \colon \eta \in \Theta \} $ ensures that 
		$ W^{\theta} $ 
	and 
		$ \sigma_{\Omega}( (W^{(\vartheta,\eta)} )_{\eta\in\Theta},  (R^{(\vartheta,\eta)})_{\eta\in\Theta} ) $ 
	are independent. 
	Combining this with Item~\eqref{properties_of_mlp:item2} in \cref{lem:properties_of_mlp} and Hutzenthaler et al.~\cite[Lemma 2.2]{Overcoming} (with 
		$ ( \Omega, \mc F, \P ) \is ( \Omega, \mc F, \P ) $, 
		$ \mc G \is \sigma_{\Omega}( (W^{(\vartheta,\eta)} )_{\eta\in\Theta},  (R^{(\vartheta,\eta)})_{\eta\in\Theta} ) $, 
		$ ( S, \mc S ) \is ( \R^d, \Borel(\R^d) ) $, 
		$ U \is ( \R^d \times \Omega \ni (y,\omega) \mapsto |U_k^{\vartheta}(y,\omega)|^p \in [0,\infty) ) $, 
		$ Y \is ( \Omega \ni \omega \mapsto x + BW^{\theta}_t(\omega) \in \R^d )$ 
	for 
		$ t \in [0,\infty) $, 
		$ x \in \R^d $ 
	in the notation of Hutzenthaler et al.~\cite[Lemma 2.2]{Overcoming}) demonstrates for all 
		$ x \in \R^d $, 
		$ t \in [0,\infty) $ 
	that 
		\begin{equation} 
		\Exp{\left|U^{\vartheta}_k(x+BW^{\theta}_t)\right|^p} 
		= 
		\int_{\R^d} \Exp{|U^{\vartheta}_k(y)|^p} \!\, \nu_{t,x}(dy). 
		\end{equation} 
	Item~\eqref{properties_of_mlp:item4} in \cref{lem:properties_of_mlp} hence establishes Item~\eqref{elementary_integration_rules:item2}. 
	The proof of \cref{lem:elementary_integration_rules} is thus completed. 	
\end{proof}

\begin{lemma} \label{lem:mlp_integrability}
	Assume \cref{setting}, 
	let $ \varepsilon \in (0,\infty) $, 
		$ p \in [1,\infty) $,
		$ L \in \R $, 
	and 
	assume for all $ x \in \R^d $, $ v,w \in \R $ that
		$| f ( x , v ) - f ( x , w ) | \leq L | v - w | $ 
	and 
		$ \int_0^{\infty} e^{(\varepsilon-\lambda) s}\,\EXP{ | f ( x+BW^{0}_s , 0 ) |^p} \,ds<\infty$. 
	Then it holds for all
		$n \in \N_0 $, $ x \in \R^d $ 
	that 	
		\begin{equation} \label{mlp_integrability:claim}
		\int_0^{\infty} e^{(\varepsilon-\lambda) s} \,\Exp{  | U^{0}_n(x+BW^{0}_{s}) |^p }\!\,ds < \infty. 
		\end{equation}
\end{lemma}

\begin{proof}[Proof of \cref{lem:mlp_integrability}]
	Throughout this proof let $ x \in \R^d $, 
	let 
		$ \A \subseteq \N $ 
	be the set given by 
		\begin{equation} 
		\A 
		= 
		\left\{ 
			n \in \N
			\colon 
			\left( \Forall k \in \{0,1,\ldots,n-1\}\colon 
			\int_0^{\infty} e^{(\varepsilon-\lambda) s}\,\Exp{ |U^0_k(x+BW^{0}_s)|^p}\!\,ds < \infty 
			\right)  
		\right\} \!, 
		\end{equation}
	and let 
		$ \nu_{t,y} \colon \Borel(\R^d) \to [0,\infty) $, $ t\in [0,\infty) $, $ y \in \R^d $, 
	satisfy for all 
		$ t \in [0,\infty) $, 
		$ y \in \R^d $, 
		$ A \in \Borel(\R^d) $ 
	that 
		$ \nu_{t,y}(A) = \P( y + BW^0_t \in A) $.
	Observe that the assumption that for all $ y \in \R^d $ it holds that $ U^0_0(y) = 0 $ ensures that $ 1 \in \A $. 
	Moreover, note that \eqref{mlp_scheme}, the fact that for all 
		$ a,b \in \R $ 
	it holds that 
		$ |a+b|^p \leq 2^{p-1} |a|^p + 2^{p-1} |b|^p $, 
	and the assumption that for all 
		$ y \in \R^d $, 
		$ v,w \in \R $ 
	it holds that 
		$ | f(y,v) - f(y,w) | \leq L | v-w | $ 
	ensure for all 
		$ n \in \N $, 
		$ y \in \R^d $ 
	that
		\begin{align}
		\nonumber
		& |U^0_n(y)|^p
		\leq 
		2^{p-1} \Bigg[ \sum_{k=1}^{n-1} \frac{1}{\lambda M^{(n-k)}} 
		\sum_{m=1}^{M^{(n-k)}} 
		\bigg|
		f\big(y+BW^{(0,k,m)}_{R^{(0,k,m)}},U^{(0,k,m)}_{k}(y+BW^{(0,k,m)}_{R^{(0,k,m)}})\big) 
		\\ \nonumber
		& -
		f\big(y+BW^{(0,k,m)}_{R^{(0,k,m)}},U^{(0,k,-m)}_{k-1}(y+BW^{(0,k,m)}_{R^{(0,k,m)}})\big)\bigg| \Bigg]^p 
		+ 
		2^{p-1} \left[ 
		\frac{1}{\lambda M^n} \sum_{m=1}^{M^n} \left| f(y+BW^{(0,0,m)}_{R^{(0,0,m)}},0) \right|
		\right]^p
		\\
		& \leq
		2^{p-1} \left[ 
		\sum_{k=1}^{n-1} \frac{1}{\lambda M^{(n-k)}} 
		\sum_{m=1}^{M^{(n-k)}} 
		L \left| U^{(0,k,m)}_k(y+BW^{(0,k,m)}_{R^{(0,k,m)}}) - U^{(0,k,-m)}_{k-1}(y+BW^{(0,k,m)}_{R^{(0,k,m)}}) \right|
		\right]^p
		\\ \nonumber
		& + 
		2^{p-1} \left[ 
		\frac{1}{\lambda M^n} \sum_{m=1}^{M^n} \left| f(y+BW^{(0,0,m)}_{R^{(0,0,m)}},0) \right|
		\right]^p
		.  
		\end{align} 
	This and the fact that $ [0,\infty) \ni t \mapsto t^p \in [0,\infty) $ is convex guarantee for all 
		$ n \in \N $, 
		$ y \in \R^d $ 
	that 
		\begin{equation} 
		\begin{split}
		|U^0_n(y)|^p 
		& \leq 
		\frac{2^{p-1}}{\lambda^p M^n} \sum_{m=1}^{M^n} \left| f(y+BW^{(0,0,m)}_{R^{(0,0,m)}},0) \right|^p
		\\
		& + 
		\sum_{k=1}^{n-1} \frac{2^{p-1} n^p L^p}{\lambda^p M^{(n-k)}} 
		\sum_{m=1}^{M^{(n-k)}} 
		\left|
		U^{(0,k,m)}_{k}(y+BW^{(0,k,m)}_{R^{(0,k,m)}}) 
		- 
		U^{(0,k,-m)}_{k-1}(y+BW^{(0,k,m)}_{R^{(0,k,m)}})
		\right|^p
		\\
		& \leq 
		\frac{2^{p-1}}{\lambda^p M^n} \sum_{m=1}^{M^n} \left| f(y+BW^{(0,0,m)}_{R^{(0,0,m)}},0) \right|^p
		\\
		& + 
		\sum_{k=1}^{n-1} \frac{4^{p-1}n^pL^p}{\lambda^p M^{(n-k)}} 
		\sum_{m=1}^{M^{(n-k)}} 
		\left[ 
		\left|
		U^{(0,k,m)}_{k}(y+BW^{(0,k,m)}_{R^{(0,k,m)}}) \right|^p
		+ 
		\left| U^{(0,k,-m)}_{k-1}(y+BW^{(0,k,m)}_{R^{(0,k,m)}})
		\right|^p
		\right] \!.  
		\end{split}
		\end{equation} 
	Item~\eqref{properties_of_mlp:item3} in \cref{lem:properties_of_mlp} hence ensures for all 
		$ n \in\N $, 
		$ y \in \R^d $
	that 
		\begin{equation} \label{integrability:eq01}
		\begin{split}		
		& 
		\Exp{|U^0_n(y)|^p} 
		\leq 
		\frac{2^{p-1}}{\lambda^p} 
		\Exp{\left| f(y+BW^{0}_{R^{0}},0) \right|^p}
		\\
		& + 
		\sum_{k=1}^{n-1} \frac{4^{p-1}n^pL^p}{\lambda^p M^{(n-k)}} 
		\sum_{m=1}^{M^{(n-k)}} 
		\Exp{\left|
		U^{(0,k,m)}_{k}(y+BW^{(0,k,m)}_{R^{(0,k,m)}})\right|^p 
		+ 
		\left| 
		U^{(0,k,-m)}_{k-1}(y+BW^{(0,k,m)}_{R^{(0,k,m)}})
		\right|^p}\!.  
		\end{split}
		\end{equation}
	Next note that Item~\eqref{elementary_integration_rules:item1} in \cref{lem:elementary_integration_rules} (with 	
		$ k \is k $, 
		$ p \is p $, 
		$ \theta \is (0,k,m) $, 
		$ \vartheta \is (0,k,m) $ 
	for $k,m\in\N$ in the notation of \cref{lem:elementary_integration_rules}) demonstrates for all 
		$ k,m \in \N $, 
		$ y \in \R^d $ 
	that	
		\begin{equation} \label{integrability:eq02}
		\Exp{|U^{(0,k,m)}_k(y+BW^{(0,k,m)}_{R^{(0,k,m)}})|^p} 
		= 
		\int_0^{\infty} \lambda e^{-\lambda s} \, 
		\Exp{|U^{(0,k,m)}_k(y+BW^{(0,k,m)}_{s})|^p}\!\,ds. 
		\end{equation} 
	Moreover, Item~\eqref{elementary_integration_rules:item1} in \cref{lem:elementary_integration_rules} (with 
		$ k \is k-1 $, 
		$ p \is p $, 
		$ \theta \is (0,k,m) $, 
		$ \vartheta \is (0,k,-m) $ 
	for $ k,m \in \N $ in the notation of \cref{lem:elementary_integration_rules}) proves for all 
		$ k,m \in \N $, 
		$ y \in \R^d $ 
	that	
		\begin{equation} \label{integrability:eq03}
		\Exp{|U^{(0,k,-m)}_{k-1}(y+BW^{(0,k,m)}_{R^{(0,k,m)}})|^p} 
		= 
		\int_0^{\infty} \lambda e^{-\lambda s} \, 
		\Exp{|U^{(0,k,-m)}_{k-1}(y+BW^{(0,k,m)}_{s})|^p}\!\,ds.  
		\end{equation}
	This, \eqref{integrability:eq01}, and \eqref{integrability:eq02} ensure for all 
		$ n \in \N $, 
		$ y \in \R^d $ 
	that 
		\begin{align} 
		&
		\Exp{|U^0_n(y)|^p} 
		\leq 
		\frac{2^{p-1}}{\lambda^{p-1}} 
		\int_0^{\infty} e^{-\lambda s} \,
		\Exp{\left| f(y+BW^{0}_{s},0) \right|^p}\!\,ds
		\\ \nonumber
		& + 
		\sum_{k=1}^{n-1} \frac{4^{p-1}n^pL^p}{\lambda^{p-1} M^{(n-k)}} 
		\sum_{m=1}^{M^{(n-k)}} 
		\int_0^{\infty} e^{-\lambda s}\,
		\Exp{\left|
		U^{(0,k,m)}_{k}(y+BW^{(0,k,m)}_{s})\right|^p 
		+ 
		\left| 
		U^{(0,k,-m)}_{k-1}(y+BW^{(0,k,m)}_s)
		\right|^p}\!\,ds.  		
		\end{align} 
	Item~\eqref{elementary_integration_rules:item2} in \cref{lem:elementary_integration_rules} (with 
		$ k \is k $, 
		$ p \is p $, 
		$ \theta \is (0,k,m) $, 
		$\vartheta \is (0,k,m) $ 
	for $ k,m \in \N $
	in the notation of \cref{lem:elementary_integration_rules}) and 
	Item~\eqref{elementary_integration_rules:item2} in \cref{lem:elementary_integration_rules} (with 
		$ k \is k-1 $, 
		$ p \is p $, 
		$ \theta \is (0,k,m) $, 
		$ \vartheta \is (0,k,-m) $ 
	for $ k,m \in \N $
	in the notation of \cref{lem:elementary_integration_rules}) hence guarantee for all 
		$ n \in \N $, 
		$ y \in \R^d $ 
	that 
		\begin{equation} 
		\begin{split} 
		&
		\Exp{|U^0_n(y)|^p} 
		\leq 
		\frac{2^{p-1}}{\lambda^{p-1}} 
		\int_0^{\infty} e^{-\lambda s} \,
		\Exp{\left| f(y+BW^{0}_{s},0) \right|^p}\!\,ds
		\\
		& + 
		\sum_{k=1}^{n-1} \frac{4^{p-1}n^pL^p }{\lambda^{p-1}}
		\int_0^{\infty} e^{-\lambda s}\,
		\Exp{\left|
			U^{0}_{k}(y+BW^{0}_{s})\right|^p 
			+ 
			\left| 
			U^{0}_{k-1}(y+BW^{0}_{s})
			\right|^p}\!\,ds.  		
		\\
		& \leq 
		\frac{2^{p-1}}{\lambda^{p-1}} 
		\int_0^{\infty} e^{-\lambda s} 
		\left[ 
		\int_{\R^d} 
		\left(
		\left| f(z,0) \right|^p
		+ 
		2^pn^pL^p
		\sum_{k=0}^{n-1} 
		\Exp{ 
			\left|
			U^{0}_{k}(z)\right|^p}\!\right) \!\, \nu_{s,y}(dz) \right] \!\,ds
		.  		
		\end{split}
		\end{equation} 
	Item~\eqref{elementary_integration_rules:item2} in \cref{lem:elementary_integration_rules} (with 
		$ k \is n $, 
		$ p \is p $, 
		$ \theta \is 0 $, 
		$ \vartheta \is 0 $ 
	for $ n \in \N $
	in the notation of \cref{lem:elementary_integration_rules}) therefore implies for all 
		$ n \in \N $,
		$ t \in [0,\infty) $
	that 	
		\begin{equation} 
		\begin{split}
		& \Exp{|U^0_n(x+BW^{0}_t)|^p} 
		= 
		\int_{\R^d} \Exp{|U^0_n(y)|^p} \!\,\nu_{t,x}(dy)
		\\
		& \leq 
		\frac{2^{p-1}}{\lambda^{p-1}}
		\int_{\R^d}
		\int_0^{\infty} e^{-\lambda s} 
		\left[ 
		\int_{\R^d} 
		\left(
		\left| f(z,0) \right|^p 
		+ 
		2^pn^pL^p 
		\sum_{k=0}^{n-1} 
		\Exp{ 
			\left|
			U^{0}_{k}(z)\right|^p}\!\right) \! \nu_{s,y}(dz) \right] \!\,ds\,
		\nu_{t,x}(dy). 
		\end{split}
		\end{equation}
	Fubini's theorem hence proves for all 
		$ n \in \N $, 
		$ t \in [0,\infty) $
	that 
		\begin{equation}
		\begin{split}		
		& \Exp{|U^0_n(x+BW^{0}_t)|^p} 
		\\
		& \leq 
		\frac{2^{p-1}}{\lambda^{p-1}}
		\int_0^{\infty} 
		e^{-\lambda s} 
		\int_{\R^d}
		\left[ 
		\int_{\R^d} 
		\left(
		\left| f(z,0) \right|^p 
		+ 
		2^pn^pL^p \sum_{k=0}^{n-1} 
		\Exp{ 
			\left|
			U^{0}_{k}(z)\right|^p}\!\right) \! \nu_{s,y}(dz) \right] \!
		\nu_{t,x}(dy) \,ds. 
	\end{split}
	\end{equation} 
	The fact that for all 
		$ t,s\in [0,\infty) $,
		$ A \in \Borel(\R^d) $
	it holds that 
		$ \nu_{t+s,x}(A) = \int_{\R^d} \nu_{s,y}(A) \, \nu_{t,x}(dy) $ 
	therefore ensures for all 
		$ n \in \N $, 
		$ t \in [0,\infty) $ 
	that 
	\begin{equation} 
	\begin{split}
		& 
		\Exp{|U^0_n(x+BW^{0}_t)|^p} 
		\\
		& \leq  
		\frac{2^{p-1}}{\lambda^{p-1}} 
		\int_0^{\infty} e^{-\lambda s} 
		\left[ 
		\int_{\R^d} 
		\left(
		\left| f(z,0) \right|^p
		+ 
		2^pn^pL^p
		\sum_{k=0}^{n-1} 
		\Exp{ 
			\left|
			U^{0}_{k}(z)\right|^p}\!\right) \!\, \nu_{t+s,x}(dz) \right] \!\,ds. 
	\end{split} 
	\end{equation} 
	Combining this with Item~\eqref{elementary_integration_rules:item2} in \cref{lem:elementary_integration_rules} (with 
		$ k \is k $, 
		$ p \is p $, 
		$ \theta \is 0 $, 
		$ \vartheta \is 0 $
	for $ k \in \N $
	in the notation of \cref{lem:elementary_integration_rules}) 
	demonstrates for all 
		$ n \in \N $, 
		$ t \in [0,\infty) $ 
	that 
		\begin{equation} 
		\begin{split}
		&
		\Exp{|U^0_n(x+BW^{0}_t)|^p}
		\\
		& \leq 
		\frac{2^{p-1}}{\lambda^{p-1}}
		\int_0^{\infty} 
		e^{-\lambda s}
		\left[
		\Exp{|f(x+BW^{0}_{t+s},0)|^p} 
		+ 
		2^pn^pL^p
		\sum_{k=0}^{n-1} \Exp{|U^0_k(x+BW^{0}_{t+s})|^p}
		\right] \!\,ds 
		\\
		& = 
		e^{\lambda t} 
		\frac{2^{p-1}}{\lambda^{p-1}}
		\int_t^{\infty} 
		e^{-\lambda s} \left(
			\Exp{|f(x+BW^{0}_s,0)|^p}
			+	 
			2^pn^pL^p 
			\sum_{k=0}^{n-1} 
			\Exp{|U^0_k(x+BW^{0}_{s})|^p}
		\right)\!\,ds. 
		\end{split}
		\end{equation} 
	Hence, we obtain for all 
		$ n \in \N $
	that 
		\begin{equation}
		\begin{split}
		& 
		\int_0^{\infty} e^{(\varepsilon-\lambda) t} \, \Exp{|U^{0}_n(x+BW^{0}_{t}) |^p}\!\,dt
		\\
		& \leq 
		\frac{2^{p-1}}{\lambda^{p-1}}
		\int_0^{\infty} e^{\varepsilon t}
		\int_t^{\infty} e^{-\lambda s} \left(
	 	\Exp{|f(x+BW^{0}_s,0)|^p}
		+	 
		2^pn^pL^p 
		\sum_{k=0}^{n-1} 
		\Exp{|U^0_k(x+BW^{0}_{s})|^p}
		\right)\!\,ds\,dt 
		\\
		& = 
		\frac{2^{p-1}}{\lambda^{p-1}}
		\int_0^{\infty} 
		e^{-\lambda s}
		\left(
		 \Exp{|f(x+BW^{0}_s,0)|^p}
		+	 
		2^pn^pL^p
		\sum_{k=0}^{n-1} 
		\Exp{|U^0_k(x+BW^{0}_{s})|^p}
		\right)\!
		\left[
		\int_0^s e^{\varepsilon t} \,dt
		\right]\!\,ds 
		\\
		& \leq  
		\frac{2^{p-1}}{\lambda^{p-1}\varepsilon}
		\int_0^{\infty}
			e^{(\varepsilon-\lambda) s} \left(
			\Exp{|f(x+BW^{0}_s,0)|^p}
			+	 
			2^pn^pL^p
			\sum_{k=0}^{n-1} 
			\Exp{|U^0_k(x+BW^{0}_{s})|^p}
			\right)\!\,ds.  
		\end{split}
		\end{equation}
	This implies for all 
		$ n \in \A $
	that 
		$ \int_0^{\infty} e^{(\varepsilon-\lambda)s} \,\EXP{|U^0_n(x+BW^{0}_s)|^p}\,ds < \infty $. 
	Therefore, we obtain for all $ n \in \N $ that $ (n \in \A \Rightarrow n+1 \in \A )$. Combining this with the fact that $ 1 \in \A $ and induction demonstrates that $ \N \subseteq \A $. This establishes \eqref{mlp_integrability:claim}. The proof of \cref{lem:mlp_integrability} is thus completed. 
\end{proof}

\subsection{Bias and variance estimates for MLP approximations} \label{subsec:bias_variance_estimates}

\begin{lemma}[Mean] \label{lem:mean}
	Assume 
		\cref{setting}, 
	let $L\in\R $, $ \varepsilon \in (0,\infty) $,  
	and assume for all $x\in\R^d$, $v,w\in\R$ that 
		$|f(x,v)-f(x,w)|\leq L|v-w|$ 
	and 
		$\int_0^{\infty} e^{(\varepsilon-\lambda) s}\,\EXP{|f(x+BW^{0}_s,0)|}\,ds<\infty$.  
	Then 
\begin{enumerate}[(i)]
	\item \label{mean:item1b}
	it holds for all 
		$ n \in \N $, 
		$ x \in \R^d $ 
	that 
		$
		\int_0^{\infty} e^{(\varepsilon-\lambda) s}\,
		\EXP{|f(x+BW^{0}_s,U^{0}_{n-1}(x+BW^{0}_s))|}\,ds<\infty
		$
	and
	\item \label{mean:item2} 
	it holds for all 
		$ n \in \N $, 
		$ \theta \in \Theta $, 
		$ x \in \R^d $ 
	that 
	\begin{equation} 
	\begin{split}
	& \Exp{U^{\theta}_{n}(x)} 
	= 
	\frac{1}{\lambda} \Exp{f\big(x+BW^{0}_{R^0},U^{0}_{n-1}(x+BW^{0}_{R^{0}})\big)}\!.
	\end{split}
	\end{equation}
	\end{enumerate}
\end{lemma} 

\begin{proof}[Proof of \cref{lem:mean}]
	First, observe that \cref{lem:mlp_integrability} (with 
		$ \varepsilon \is \varepsilon $, 
		$ p \is 1 $, 
		$ L \is L $ 
	in the notation of \cref{lem:mlp_integrability})
	assures for all 
		$ n \in \N_0 $, 
		$ x \in \R^d $
	that 
		\begin{equation} 
		\int_0^{\infty} e^{(\varepsilon-\lambda) t}\,\EXP{|U^0_n(x+BW^{0}_t)|}\,dt < \infty. 
		\end{equation} 
	Combining this with the assumption that for all 
		$ x \in \R^d $ 
	it holds that 
		$ 
		\int_0^{\infty} e^{(\varepsilon-\lambda) s}\,\EXP{|f(x+BW^{0}_s,0)|}\,ds < \infty
		$ 
	and the assumption that for all 
		$ x \in \R^d $, 
		$ v,w \in \R $ 
	it holds that 
		$ | f(x,v) - f(x,w) | \leq L |v-w| $ 
	ensures for all 
		$ n \in \N $, 
		$ x \in \R^d $ 
	that 
		\begin{equation} 
		\begin{split}
		& 
		\int_0^{\infty} e^{(\varepsilon-\lambda)t}\,\Exp{|f(x+BW^{0}_t,U^0_{n-1}(x+BW^{0}_t))|}\!\,dt
		\\
		& \leq 
		\int_0^{\infty} 
		e^{(\varepsilon-\lambda)t}\,\Exp{|f(x+BW^{0}_t,0)|}\!\,dt 
		+ 
		L \int_0^{\infty} 
		e^{(\varepsilon-\lambda)t}\,\Exp{|U^0_{n-1}(x+BW^{0}_t)|}\!\,dt < \infty. 
		\end{split}
		\end{equation} 
	This establishes Item~\eqref{mean:item1b}. 
	Next note that Items~\eqref{properties_of_mlp:item1}--\eqref{properties_of_mlp:item4} in \cref{lem:properties_of_mlp} 
	and Hutzenthaler et al.~\cite[Corollary 2.5]{Overcoming} 
	ensure 
	\begin{itemize}
		\item 
		for all 
			$ \theta \in \Theta $, 
			$ k,m \in \N $ 
		that 
			$ \R^d \times \Omega \ni (x,\omega) \mapsto f(x+BW^{(\theta,k,m)}_{R^{(\theta,k,m)}(\omega)}(\omega),U^{(\theta,k,m)}_k(x+BW^{(\theta,k,m)}_{R^{(\theta,k,m)}(\omega)}(\omega),\omega)) \in \R $ 
		and 
			$ \R^d \times \Omega \ni (x,\omega) \mapsto f(x+BW^{0}_{R^{0}(\omega)}(\omega),U^{0}_k(x+BW^{0}_{R^{0}(\omega)}(\omega),\omega)) \in \R $ 
		are identically distributed random fields 
		and  
		\item 
		for all 
			$ \theta \in \Theta $, 
			$ k,m \in \N $ 
		that 
			$ \R^d \times \Omega \ni (x,\omega) \mapsto f(x+BW^{(\theta,k,m)}_{R^{(\theta,k,m)}(\omega)}(\omega),U^{(\theta,k,-m)}_{k-1}(x+BW^{(\theta,k,m)}_{R^{(\theta,k,m)}(\omega)}(\omega),\omega)) \in \R $ 
		and 
			$ \R^d \times \Omega \ni (x,\omega) \mapsto f(x+BW^{0}_{R^{0}(\omega)}(\omega),U^{0}_{k-1}(x+BW^{0}_{R^{0}(\omega)}(\omega),\omega)) \in \R $ 
		are identically distributed random fields.  
		\end{itemize}
	Combining this with Item~\eqref{mean:item1b} shows for all 
		$ \theta \in \Theta $,
		$ n \in \N $, 
		$ x \in \R^d $
	that 
		\begin{align} \label{mean:computing_01}
		\nonumber
		& \Exp{U^{\theta}_{n}(x)} 
		= 
		\sum_{k=1}^{n-1} 
		\frac{1}{\lambda M^{(n-k)}} 
		\sum_{m=1}^{M^{(n-k)}} 
		\bigg(
		\Exp{f\Big(x+BW^{(\theta,k,m)}_{R^{(\theta,k,m)}},U^{(\theta,k,m)}_{k}\big(x+BW^{(\theta,k,m)}_{R^{(\theta,k,m)}}\big)\Big)}
		\\ \nonumber
		& -
		\Exp{ 
			f\Big( 
			x+BW^{(\theta,k,m)}_{R^{(\theta,k,m)}},U^{(\theta,k,-m)}_{k-1}\big(x+BW^{(\theta,k,m)}_{R^{(\theta,k,m)}}\big)
			\Big)}\bigg)
		+
		\frac{1}{\lambda M^n} 
		\sum_{m=1}^{M^{n}} \Exp{f\big(x+BW^{(\theta,0,m)}_{R^{(\theta,0,m)}},0\big)}\!
		\\ 
		& 
		= 
		\sum_{k=1}^{n-1} 
		\frac{1}{\lambda} 
		\bigg(
		\Exp{f\big(x+BW^{0}_{R^{0}},U^{0}_{k}(x+BW^{0}_{R^{0}})\big)}
		-
		\Exp{f\big(x+BW^{0}_{R^{0}},U^{0}_{k-1}(x+BW^{0}_{R^{0}})\big)}\bigg)
		\\ \nonumber
		& +
		\frac{1}{\lambda} \Exp{f\big(x+BW^{0}_{R^{0}},0\big)}.
	\end{align}
	The assumption that for all 
		$ x \in \R^d $ 
	it holds that 
		$ U^0_0(x) = 0 $ 
	hence proves for all 
		$ \theta \in \Theta $, 
		$ n \in \N $, 
		$ x \in \R^d $ 
	that 	
		\begin{equation} 
		\begin{split}
		\Exp{U^{\theta}_{n}(x)} 
		& = 
		\frac{1}{\lambda} 
		\Big(
		\Exp{f\big(x+BW^0_{R^{0}},0\big)}
		\\
		& + 
		\Exp{f\big(x+BW^0_{R^{0}},U^{0}_{n-1}(x+BW^0_{R^{0}})\big)}
		-
		\Exp{f\big(x+BW^0_{R^{0}},U^{0}_{0}(x+BW^0_{R^{0}})\big)}
			\Big)
		\\
		& = 
		\frac{1}{\lambda} \Exp{f\big(x+BW^{0}_{R^0},U^{0}_{n-1}(x+BW^{0}_{R^{0}})\big)}\!.
		\end{split}
		\end{equation}
	This establishes Item~\eqref{mean:item2}. 
	The proof of \cref{lem:mean} is thus completed. 
\end{proof}

\begin{lemma}[Bias] \label{lem:bias}
	Assume \cref{setting}, 
	let $ L \in \R $, 
		$ \varepsilon \in (0,\infty) $, 
		$ p \in [1,\infty) $, 
	assume for all 
		$ x \in \R^d $, 
		$ v , w \in \R $  
	that 
		$ | f ( x , v ) - f ( x , w ) | \leq L | v - w | $ 
	and 
		$\int_0^{\infty} e^{(\varepsilon-\lambda) s} \, \EXP{|f(x+BW^{0}_s,0)|}\,ds<\infty$, 
	let $u \colon \R^d \to \R$ be $ \Borel(\R^d) $/$ \Borel(\R) $-measurable, 
	and assume for all 
		$ x \in \R^d $ 
	that 
		$ \int_0^{\infty} e^{ -\lambda s } \, \EXP{ | u ( x+BW^{ 0 }_s ) | }\,ds < \infty $
	and
		\begin{equation} \label{bias:exact_solution}
		u(x) 
		= 
		\frac{1}{\lambda} 
		\Exp{ 
			f(x+BW^{0}_{R^0},u(x+BW^{0}_{R^0})) }
		=
		\Exp{
		\int_0^{\infty} 
		e^{-\lambda s} f(x+BW^{0}_s,u(x+BW^{0}_s))\,ds}\!.
		\end{equation}
Then it holds for all 
	$ n \in \N $, 
	$ \theta \in \Theta $, 
	$ x \in \R^d $ 
that	
	\begin{equation} \label{bias:claim}
	\left|\Exp{U^{\theta}_{n}(x)}-u(x)\right|^p 
	\leq
	\frac{L^p}{\lambda^p} 
	\Exp{|U^0_{n-1}(x+BW^{0}_{R^{0}})-u(x+BW^{0}_{R^0})|^p}\!. 
	\end{equation} 
\end{lemma} 

\begin{proof}[Proof of \cref{lem:bias}] 
First, observe that \cref{lem:mean} (with 
	$ \varepsilon \is \varepsilon $, 
	$ L \is L $ 
in the notation of \cref{lem:mean}) and \eqref{bias:exact_solution} guarantee for all 
	$ n \in \N $, 
	$ \theta \in \Theta $, 
	$ x \in \R^d $ 
that 
	\begin{equation} 
	\begin{split}
	\Exp{U^{\theta}_{n}(x)}-u(x) 
	& =  
	\frac{1}{\lambda} \Exp{f(x+BW^{0}_{R^0},U^0_{n-1}(x+BW^{0}_{R^0})) - f(x+BW^{0}_{R^0},u(x+BW^{0}_{R^0}))}\!.
	\end{split} 
	\end{equation} 
The assumption that for all 
	$ x \in \R^d $, 
	$ v,w \in \R $ 
it holds that 
	$|f(x,v)-f(x,w)|\leq L|v-w|$ therefore implies 
for all 
	$ n \in \N $, 
	$ \theta \in \Theta $, 
	$ x \in \R^d $ 
that 
	\begin{equation}
	\left|\Exp{U^{\theta}_{n}(x)}-u(x)\right|
	\leq  
	\frac{L}{\lambda} \Exp{|U^0_{n-1}(x+BW^{0}_{R^0})-u(x+BW^{0}_{R^0})|}\!.
	\end{equation}
Jensen's inequality hence establishes \eqref{bias:claim}. The proof of \cref{lem:bias} is thus completed. 
\end{proof} 

\begin{lemma} [Variance] \label{lem:variance}
Assume  
	\cref{setting}, 
let $ L \in \R $, 
	$ \varepsilon \in (0,\infty) $,
and assume for all 
	$ x \in \R^d $, 
	$ v,w \in \R $ 
that 
	$ |f(x,v)-f(x,w)|\leq L|v-w| $ 
and 
	$ \int_0^{\infty} e^{(\varepsilon-\lambda) s}\,\EXP{|f(x+BW^0_s,0)|}\,ds<\infty $. 
Then it holds for all 
	$ n \in \N $, 
	$ \theta \in \Theta $, 
	$ x \in \R^d $
that 
	\begin{equation} \label{variance:claim}
	\begin{split}
	& \Var{U^{\theta}_{n}(x)} 
	\\
	& \leq 
	\frac{1}{\lambda^2} \bigg[ 
	\frac{1}{M^n} 
	\Exp{|f(x+BW^{0}_{R^0},0)|^2} 
	+ 
	\sum_{k=1}^{n-1} \frac{L^2}{M^{(n-k)}} 
	\Exp{|U^0_{k}(x+BW^{0}_{R^0})-U^1_{k-1}(x+BW^{0}_{R^0})|^2}\bigg].
	\end{split}
	\end{equation} 
\end{lemma} 

\begin{proof}[Proof of \cref{lem:variance}] 
First, observe that Item~\eqref{properties_of_mlp:item2} in \cref{lem:properties_of_mlp} and \cref{setting} ensure for all 
	$ \theta \in \Theta $, 
	$ x \in \R^d $
that 
	\begin{equation}
	\begin{split} 
	& 
	\N_0 \times \N \ni (k,m) 
	\mapsto \\
	&  
	\begin{cases} 
		\left( 
		\begin{array}{l}
		\Omega \ni \omega  
		\mapsto 
		f(x+BW^{(\theta,k,m)}_{R^{(\theta,k,m)}(\omega)}(\omega), 
		U_k^{(\theta,k,m)}(x+BW^{(\theta,k,m)}_{R^{(\theta,k,m)}(\omega)}(\omega),\omega)) 
		\\
		  - 
		f(x+BW^{(\theta,k,m)}_{R^{(\theta,k,m)}(\omega)}(\omega), 
		U_{k-1}^{(\theta,k,-m)}(x+BW^{(\theta,k,m)}_{R^{(\theta,k,m)}(\omega)}(\omega),\omega)) \in \R
		\end{array} \right) & \colon k\geq 1 \\
		\Omega \ni \omega \mapsto f(x+BW^{(\theta,0,m)}_{R^{(\theta,0,m)}(\omega)}(\omega),0) \in \R & \colon k=0
	\end{cases} 
	\end{split}
	\end{equation} 
	is an independent family of random variables. This, Item~\eqref{mean:item1b} in \cref{lem:mean}, and \eqref{mlp_scheme} imply for all 
		$ n \in \N $, 
		$ \theta \in \Theta $, 
		$ x \in \R^d $ 
	that 
	\begin{equation} 
	\begin{split}
	\Var{U^{\theta}_{n}(x)} 
	& = 
	\frac{1}{\lambda^2} 
	\Bigg(
	\sum_{m=1}^{M^{n}} \var{\frac{f(x+BW^{(\theta,0,m)}_{R^{(\theta,0,m)}},0)}{M^n}}
	\\
	& 
	+ 
	\sum_{k=1}^{n-1} 
	\sum_{m=1}^{M^{(n-k)}} 
	\VARRRRR{ \frac{ f(x+BW^{(\theta,k,m)}_{R^{(\theta,k,m)}},U^{(\theta,k,m)}_{k}(x+BW^{(\theta,k,m)}_{R^{(\theta,k,m)}}) ) }{M^{(n-k)}} \ldots
	\\
	& \qquad \qquad \qquad \qquad \ldots - 
	\frac{f(x+BW^{(\theta,k,m)}_{R^{(\theta,k,m)}},U^{(\theta,k,-m)}_{k-1}(x+BW^{(\theta,k,m)}_{R^{(\theta,k,m)}}))	}{M^{(n-k)}} }
	\Bigg).
	\end{split}
	\end{equation}
Items~\eqref{properties_of_mlp:item1}--\eqref{properties_of_mlp:item4} in \cref{lem:properties_of_mlp} and Hutzenthaler et al.~\cite[Corollary 2.5]{Overcoming} hence imply for all 
	$ n \in \N $, 
	$ \theta \in \Theta $, 
	$ x \in \R^d $ 
that 
	\begin{equation}
	\begin{split}	
	&\var{U^{\theta}_n(x)}
	= 
	\frac{1}{\lambda^2} 
	\Bigg[
	\frac{1}{M^n} \Var{f(x+BW^{0}_{R^{0}},0)}	
	\\
	& + 
	\sum_{k=1}^{n-1} \frac{1}{M^{(n-k)}} 
	\var{ f(x+BW^{0}_{R^{0}},U^{0}_{k}(x+BW^{0}_{R^{0}} ) )
		- 
		f(x+BW^{0}_{R^{0}},U^{1}_{k-1}(x+BW^{0}_{R^{0}})) }
	\Bigg]. 
	\end{split}
	\end{equation} 
Combining this with the fact that for every random variable 
	$Y\colon\Omega\to\R$ 
with 
	$\Exp{|Y|}<\infty$ 
it holds that 
	$ \VAR{Y} \leq \EXP{Y^2} \in [0,\infty] $ 
guarantees that for all 
	$ n \in \N $, 
	$ \theta \in \Theta $, 
	$ x \in \R^d $ 
it holds that 
	\begin{equation} 
	\begin{split}
	& 
	\var{U^{\theta}_{n}(x)} 
	\leq 
	\frac{1}{\lambda^2} \bigg[
	\frac{1}{M^n} 
	\Exp{|f(x+BW^{0}_{R^0},0)|^2} 
	\\
	& 
	+ 
	\sum_{k=1}^{n-1} 
	\frac{1}{M^{(n-k)}} 
	\Exp{|f(x+BW^{0}_{R^{0}},U^{0}_{k}(x+BW^{0}_{R^{0}} )) 
	- 
	f(x+BW^{0}_{R^{0}},U^{1}_{k-1}(x+BW^{0}_{R^{0}} ))|^2}
	\bigg].  
	\end{split}
	\end{equation} 
The assumption that for all 
	$ x \in \R^d $, 
	$ v,w \in \R $ 
it holds that 
	$ |f(x,v)-f(x,w)| \leq L |v-w| $
therefore ensures for all 
	$ n \in \N $, 
	$ \theta \in \Theta $, 
	$ x \in \R^d $ 
that 
	\begin{equation} 
	\begin{split}
	&
	\var{U^{\theta}_n(x)}
	\\
	& \leq 
	\frac{1}{\lambda^2} \bigg[ 
	\frac{1}{M^n} 
	\Exp{|f(x+BW^{0}_{R^0},0)|^2} 
	+ 
	\sum_{k=1}^{n-1} \frac{L^2}{M^{(n-k)}} 
	\Exp{|U^0_{k}(x+BW^{0}_{R^0})-U^1_{k-1}(x+BW^0_{R^0})|^2}\bigg]. 
	\end{split}
	\end{equation}
	This establishes \eqref{variance:claim}. 
The proof of \cref{lem:variance} is thus completed. 
\end{proof} 

\subsection{Error analysis for MLP approximations} \label{subsec:convergence_analysis_mlp}

\begin{lemma}[Recursive overall error estimate] \label{lem:recursive_estimate}
	Assume  
		\cref{setting}, 
	let 
		$ L \in \R $, 
		$ \varepsilon \in (0,\infty) $, 
	assume for all 
		$ x \in \R^d $, 
		$ v,w \in \R $ 
	that 
		$ | f( x , v ) - f( x , w ) | \leq L | v - w | $ 
	and 
		$\int_0^{\infty} e^{(\varepsilon-\lambda) s} \, \EXP{|f(x+BW^{0}_s,0)|}\,ds<\infty$,
	let 
		$u \colon \R^d \to \R$ be $ \Borel(\R^{d}) $/$ \Borel(\R) $-measurable, 
	assume for all 
		$ x \in \R^d $ 
	that 
		$ \int_0^{\infty} e^{-\lambda s} \, \EXP{|u(x+BW^{0}_s)|}\,ds < \infty $
	and 
		\begin{equation} \label{recursive_estimate:exact_solution}
		u(x) 
		= 
		\frac{1}{\lambda} 
		\Exp{ 
			f(x+BW^{0}_{R^0},u(x+BW^{0}_{R^0})) }
		= 
		\Exp{
			\int_0^{\infty} 
			e^{-\lambda s} f(x+BW^{0}_s,u(x+BW^{0}_s))\,ds}\!, 
		\end{equation}
	and let 	
		$ \alpha, \beta \colon \Borel([0,\infty) ) \to [0,\infty) $ 
	be measures which satisfy for all 
		$ A \in \Borel([0,\infty)) $ 
	that 
		$ \beta(A) \geq \int_A e^{-\lambda s} (\int_{[0,s]} e^{\lambda t}\,\alpha(dt))\,ds $. 
	Then it holds for all 
		$ n \in \N $, 
		$ \theta \in \Theta $, 
		$ x \in \R^d $ 
	that 
		\begin{equation} \label{recursive_estimate:claim}
		\begin{split}
		& 
		\left[ \int_{[0,\infty)} \Exp{|U^{\theta}_n(x+BW^{\theta}_t) - u(x+BW^{\theta}_t)|^2}\!\,\alpha(dt) \right]^{\nicefrac12}
		\\&\qquad
		\leq 
		\frac{1}{\sqrt{\lambda M^n}} 
		\left[ 
		\int_{[0,\infty)} \Exp{|f(x+BW^{0}_{s},0)|^2}\!\,\beta(ds) \right]^{\nicefrac12} 
		\\
		&\qquad\qquad + 
		\sum_{k=0}^{n-1} 
		\frac{L(1+\sqrt{M})}{\sqrt{\lambda M^{(n-k)}}} 
		\left[ 
		\int_{[0,\infty)} \Exp{|U^0_{k}(x+BW^{0}_{s})-u(x+BW^{0}_{s})|^2} \beta(ds)
		\right]^{\nicefrac12}.  
		\end{split}
		\end{equation}	
\end{lemma} 

\begin{proof}[Proof of \cref{lem:recursive_estimate}]
Throughout this proof let 
	$ \nu_{t,x} \colon \Borel(\R^d) \to [0,\infty) $, $ t\in [0,\infty) $, $ x \in \R^d $, 
satisfy for all 
	$ t\in [0,\infty) $, 
	$ x \in \R^d $, 
	$ A \in \Borel(\R^d) $ 
that 
	\begin{equation} \label{recursive_estimate:nu_t_x}
	\nu_{t,x}(A) = \P(x+BW^0_t\in A).
	\end{equation} 
Observe that the triangle inequality ensures for all 
	$ n \in \N $,
	$ \theta \in \Theta $,  
	$ x \in \R^d $
that 
	\begin{equation} 
	\begin{split}
	& \left(\Exp{|U^{\theta}_{n}(x)-u(x)|^2}\right)^{\!\nicefrac12}
	\leq 
	\left(\Exp{|U^{\theta}_{n}(x)-\Exp{U^{\theta}_{n}(x)}|^2}\right)^{\!\nicefrac12}
	+ 
	\left(\Exp{|\Exp{U^{\theta}_{n}(x)}-u(x)|^2}\right)^{\!\nicefrac12}
	\\
	& = 
	\left(\Exp{|U^{\theta}_{n}(x)-\Exp{U^{\theta}_{n}(x)}|^2}\right)^{\!\nicefrac12}
	+ 
	\left| \Exp{U^{\theta}_{n}(x)} - u(x) \right| 
	=  \sqrt{\Var{U^{\theta}_{n}(x)}} + \left| \Exp{U^{\theta}_{n}(x)} - u(x) \right| \!.
	\end{split}
	\end{equation} 
\cref{lem:bias} (with 
	$ L \is L $, 
	$ \varepsilon \is \varepsilon $, 
	$ p \is 2 $, 
	$ u \is u $
in the notation of \cref{lem:bias}) and \cref{lem:variance} (with 
	$ L \is L $, 
	$ \varepsilon \is \varepsilon $
in the notation of \cref{lem:variance}) hence prove for all 
	$ n \in \N $, 
	$ \theta \in \Theta $, 
	$ x \in \R^d $ 
that 
	\begin{equation} 
	\begin{split}
	& \left(\Exp{|U^{\theta}_{n}(x)-u(x)|^2}\right)^{\!\nicefrac12}
	\leq 
	\frac{L}{\lambda} 
	\left(\Exp{|U^0_{n-1}(x+BW^{0}_{R^{0}})-u(x+BW^{0}_{R^0})|^2}\right)^{\!\nicefrac12}
	\\
	& + 
	\frac{1}{\lambda} \bigg[ 
	\frac{1}{M^n} 
	\Exp{|f(x+BW^{0}_{R^0},0)|^2} 
	+ 
	\sum_{k=1}^{n-1} \frac{L^2}{M^{(n-k)}} 
	\Exp{|U^0_{k}(x+BW^{0}_{R^0})-U^1_{k-1}(x+BW^{0}_{R^0})|^2}\bigg]^{\nicefrac12}.
	\end{split}
	\end{equation}
Combining this with the fact that for all 
	$n\in\N$, 
	$a_1,a_2,\ldots,a_n\in [0,\infty)$  
it holds that 
	$\sqrt{\sum_{k=1}^n a_k} \leq \sum_{k=1}^n \sqrt{a_k}$	
shows for all 
	$ n \in \N $, 
	$ \theta \in \Theta $,  
	$ x \in \R^d $ 
that 
	\begin{equation}
	\begin{split}
	& 
	\left(\Exp{|U^{\theta}_{n}(x)-u(x)|^2}\right)^{\!\nicefrac12}
	\\[1ex]
	& 
	\leq 
		\frac{L}{\lambda} 
	\left(\Exp{|U^0_{n-1}(x+BW^{0}_{R^{0}})-u(x+BW^0_{R^0})|^2}\right)^{\!\nicefrac12}
	+ 
	\frac{1}{\lambda\sqrt{M^n}} \left(  
	\Exp{|f(x+BW^{0}_{R^0},0)|^2} \right)^{\!\nicefrac12}
	\\
	& + 
	\sum_{k=1}^{n-1} \frac{L}{\lambda\sqrt{M^{(n-k)}}} 
	\left( 
	\Exp{|U^0_{k}(x+BW^{0}_{R^0})-U^1_{k-1}(x+BW^{0}_{R^0})|^2}\right)^{\!\nicefrac12}.
	\end{split}
	\end{equation}
The triangle inequality hence ensures for all 
	$ n \in \N $, 
	$ \theta \in \Theta $, 
	$ x \in \R^d $ 
that 
	\begin{equation}
	\begin{split}
	& 
	\left(\Exp{|U^{\theta}_{n}(x)-u(x)|^2}\right)^{\!\nicefrac12}
	\\
	& \leq 
	\frac{L}{\lambda} 
	\left(\Exp{|U^0_{n-1}(x+BW^{0}_{R^{0}})-u(x+BW^{0}_{R^0})|^2}\right)^{\!\nicefrac12}
	 + 
	\frac{1}{\lambda\sqrt{M^n}} \left(  
	\Exp{|f(x+BW^{0}_{R^0},0)|^2} \right)^{\!\nicefrac12}
	\\
	&  + 
	\sum_{k=1}^{n-1} \frac{L}{\lambda\sqrt{M^{(n-k)}}} 
	\Bigg[ 
	\left( 
	\Exp{|U^0_{k}(x+BW^{0}_{R^0})-u(x+BW^{0}_{R^0})|^2}\right)^{\!\nicefrac12} 
	\\
	& \qquad\qquad\qquad\qquad	+
	\left(\Exp{|U^1_{k-1}(x+BW^{0}_{R^0})-u(x+BW^{0}_{R^0})|^2}\right)^{\!\nicefrac12} 
	\Bigg] .
	\end{split}
	\end{equation}
This, Items~\eqref{properties_of_mlp:item2} and \eqref{properties_of_mlp:item4} in \cref{lem:properties_of_mlp}, and Hutzenthaler et al.~\cite[Lemma 2.2]{Overcoming} (with 
	$ (\Omega,\mc F,\P) \is (\Omega,\mc F,\P) $, 
	$ \mc G \is \sigma_{\Omega}( (W^{(1,\vartheta)})_{\vartheta\in\Theta}, (R^{(1,\vartheta)})_{\vartheta\in\Theta} ) $, 
	$ (S,\mc S) \is (\R^d,\Borel(\R^d)) $, 
	$ U \is ( \R^d \times \Omega \ni (y,\omega) \mapsto |U^1_{k-1}(y,\omega)-u(y)|^2 \in [0,\infty) ) $, 
	$ Y \is ( \Omega \ni \omega \mapsto x + BW^0_{R^0(\omega)}(\omega) \in \R^d ) $ 
for $ x \in \R^d $, 
	$ k \in \N $
in the notation of Hutzenthaler et al.~\cite[Lemma 2.2]{Overcoming}) demonstrate for all 
	$ n \in \N $, 
	$ \theta \in \Theta $, 
	$ x \in \R^d $
that
	\begin{equation} 
	\begin{split}
	\left(\Exp{|U^{\theta}_{n}(x)-u(x)|^2}\right)^{\!\nicefrac12}
	& \leq 
	\frac{1}{\lambda\sqrt{M^n}} \left(  
	\Exp{|f(x+BW^0_{R^0},0)|^2} \right)^{\!\nicefrac12}
	\\
	& + 
	\sum_{k=1}^{n-1} \frac{L}{\lambda\sqrt{M^{(n-k)}}} 
	\left( 
	\Exp{|U^0_{k}(x+BW^0_{R^0})-u(x+BW^0_{R^0})|^2}\right)^{\!\nicefrac12} 
	\\
	& + 
	\sum_{k=1}^{n} \frac{L}{\lambda\sqrt{M^{(n-k)}}} 
	\left(\Exp{|U^0_{k-1}(x+BW^0_{R^0})-u(x+BW^0_{R^0})|^2}\right)^{\!\nicefrac12}. 
	\end{split} 
	\end{equation}
Hence, we obtain for all 
	$ n \in \N $, 
	$ \theta \in \Theta $, 
	$ x \in \R^d $ 
that 
	\begin{equation} 
	\begin{split}
	\left(\Exp{|U^{\theta}_{n}(x)-u(x)|^2}\right)^{\!\nicefrac12}
	& \leq 
	\frac{1}{\lambda\sqrt{M^n}} \left(  
	\Exp{|f(x+BW^0_{R^0},0)|^2} \right)^{\!\nicefrac12}
	\\
	&
	+ 
	\sum_{k=0}^{n-1} \frac{L\mathbbm{1}_{\N}(k)+L\sqrt{M}}{\lambda\sqrt{M^{(n-k)}}} 
	\left( 
	\Exp{|U^0_{k}(x+BW^0_{R^0})-u(x+BW^0_{R^0})|^2}\right)^{\!\nicefrac12} 
	\\
	& \leq 
	\frac{1}{\lambda\sqrt{M^n}} \left(  
	\Exp{|f(x+BW^{0}_{R^0},0)|^2} \right)^{\!\nicefrac12}
	\\
	& + 
	\sum_{k=0}^{n-1} \frac{L(1+\sqrt{M})}{\lambda\sqrt{M^{(n-k)}}} 
	\left(\Exp{|U^0_{k}(x+BW^0_{R^0})-u(x+BW^{0}_{R^0})|^2}\right)^{\!\nicefrac12}
	. 
	\end{split} 
	\end{equation}
Item~\eqref{properties_of_mlp:item2} in \cref{lem:properties_of_mlp} and Hutzenthaler et al.~\cite[Lemma 2.2]{Overcoming} hence ensure for all 
	$ n \in \N $, 
	$ \theta \in \Theta $, 
	$ x \in \R^d $
that 
\begin{equation} \label{overall_analysis:eq03}
\begin{split}
	& 
	\left(\Exp{|U^{\theta}_{n}(x)-u(x)|^2}\right)^{\!\nicefrac12}
	\leq 
	\frac{1}{\sqrt{\lambda M^n}} 
	\left( 
	\int_0^{\infty} e^{-\lambda s}\,\Exp{|f(x+BW^{0}_s,0)|^2}\!\,ds 
	\right)^{\!\nicefrac12}
	\\
	& \qquad + 
	\sum_{k=0}^{n-1} \frac{L(1+\sqrt{M})}{\sqrt{\lambda M^{(n-k)}}} 
	\left( \int_0^{\infty} e^{-\lambda s} \, \Exp{|U^0_{k}(x+BW^0_s)-u(x+BW^0_s)|^2}\!\,ds \right)^{\!\nicefrac12}.
	\end{split}
	\end{equation} 
Moreover, note that Items~\eqref{properties_of_mlp:item2} and \eqref{properties_of_mlp:item4} in \cref{lem:properties_of_mlp}, Hutzenthaler et al.~\cite[Lemma 2.2]{Overcoming} (with 
	$ (\Omega,\mc F,\P) \is (\Omega,\mc F,\P) $, 
	$ \mc G \is \sigma_{\Omega}( (W^{(\theta,\vartheta)})_{\vartheta\in\Theta}, (R^{(\theta,\vartheta)})_{\vartheta\in\Theta} ) $, 
	$ (S,\mc S) \is (\R^d,\Borel(\R^d)) $, 
	$ U \is ( \R^d\times\Omega \ni (y,\omega) \mapsto |U^{\theta}_n(y,\omega) - u(y) |^2 \in [0,\infty) ) $, 
	$ Y \is (\Omega \ni \omega \mapsto x + BW^{\theta}_{t}(\omega) \in \R^d) $ 
for 
	$ \theta \in \Theta $, 
	$ t \in [0,\infty) $,
	$ x \in \R^d $, 
	$ n \in \N $
in the notation of Hutzenthaler et al.~\cite[Lemma 2.2]{Overcoming}), and \eqref{recursive_estimate:nu_t_x} prove for all 
	$ n \in \N $, 
	$ \theta \in \Theta $,  
	$ t \in [0,\infty) $,  
	$ x \in \R^d $
that 
	\begin{equation} 
	\begin{split}
	\left(\Exp{|U^{\theta}_n(x+BW^{\theta}_t)-u(x+BW^{\theta}_t)|^2}\right)^{\!\nicefrac12}	
	& = 
	\left[\int_{\R^d} 
	\Exp{|U^{\theta}_n(y)-u(y)|^2} \!\,\nu_{t,x}(dy)\right]^{\nicefrac12}. 
	\\
	& = 
	\left[\int_{\R^d} 
	 \Exp{|U^0_n(y)-u(y)|^2} \!\,\nu_{t,x}(dy)\right]^{\nicefrac12}.  
	\end{split} 
	\end{equation} 
The triangle inequality and \eqref{overall_analysis:eq03} hence show for all 
	$ n \in \N $, 
	$ \theta \in \Theta $, 
	$ t \in [0,\infty) $,
	$ x \in \R^d $
that 
	\begin{equation} 
	\begin{split}
	& 
	\left( \Exp{|U^{\theta}_n(x+BW^{\theta}_t) - u(x+BW^{\theta}_t)|^2} \right)^{\!\nicefrac12} 
	\\
	& \leq 
	\frac{1}{\sqrt{\lambda M^n}} 
	\left[ \int_{\R^d}
	\int_0^{\infty} e^{-\lambda s}\,\Exp{|f(y+BW^{0}_{s},0)|^2}\!\,ds\,\nu_{t,x}(dy)  
	\right]^{\nicefrac12}
	\\
	& 
	+ 
	\sum_{k=0}^{n-1} \frac{L(1+\sqrt{M})}{\sqrt{\lambda M^{(n-k)}}} 
	\left[
	\int_{\R^d} 
	\int_0^{\infty} e^{-\lambda s} \, \Exp{|U^0_{k}(y+BW^{0}_{s})-u(y+BW^{0}_{s})|^2}\!\,ds\,\nu_{t,x}(dy)
	\right]^{\nicefrac12}.
	\end{split}
	\end{equation} 
Combining this with the fact that for all 
	$ t,s \in [0,\infty) $, 
	$ x \in \R^d $, 
	$ A \in \Borel(\R^d) $ 
it holds that 
	$ \nu_{t+s,x}(A) = \int_{\R^d} \nu_{s,y}(A) \,\nu_{t,x}(dy) $
and Fubini's theorem proves for all 
	$ n \in \N $, 
	$ \theta \in \Theta $, 
	$ t \in [0,\infty) $,	
	$ x \in \R^{d} $
that 
	\begin{equation}
	\begin{split}
	& 
	\left( \Exp{|U^{\theta}_n(x+BW^{\theta}_t) - u(x+BW^{\theta}_t)|^2} \right)^{\!\nicefrac12} 
	\\
	& \leq 
	\frac{1}{\sqrt{\lambda M^n}} 
	\left[ \int_0^{\infty} e^{-\lambda s}\,\Exp{|f(x+BW^{0}_{t+s},0)|^2}\!\,ds
	\right]^{\nicefrac12}
	\\
	& + 
	\sum_{k=0}^{n-1} \frac{L(1+\sqrt{M})}{\sqrt{\lambda M^{(n-k)}}} 
	\left[
	\int_0^{\infty} e^{-\lambda s} \, \Exp{|U^0_{k}(x+BW^{0}_{t+s})-u(x+BW^{0}_{t+s})|^2}\!\,ds 
	\right]^{\nicefrac12}.
	\end{split}
	\end{equation} 
	The triangle inequality hence ensures for all 
		$ n \in \N $, 
		$ \theta \in \Theta $, 
		$ x \in \R^d $ 
	that 
		\begin{equation} 
		\begin{split}
		&\left[ \int_{[0,\infty)} \Exp{|U^{\theta}_n(x+BW^{\theta}_t) - u(x+BW^{\theta}_t)|^2}\!\,\alpha(dt) \right]^{\nicefrac12}
		\\
		&\leq 
		\frac{1}{\sqrt{\lambda M^n}} 
		\left[ 
		\int_{[0,\infty)} \int_0^{\infty} e^{-\lambda s}\,\Exp{|f(x+BW^{0}_{t+s},0)|^2}\!\,ds\,\alpha(dt) \right]^{\nicefrac12} 
		\\
		& + 
		\sum_{k=0}^{n-1} 
		\frac{L(1+\sqrt{M})}{\sqrt{\lambda M^{(n-k)}}} 
		\left[ 
		\int_{[0,\infty)}
		\int_0^{\infty} e^{-\lambda s} \, \Exp{|U^0_{k}(x+BW^{0}_{t+s})-u(x+BW^{0}_{t+s})|^2}\!\,ds 
		\,\alpha(dt)
		\right]^{\nicefrac12}
		\\
		& = 
		\frac{1}{\sqrt{\lambda M^n}} 
		\left[ 
		\int_{[0,\infty)} e^{\lambda t} \int_t^{\infty} e^{-\lambda s}\,\Exp{|f(x+BW^{0}_{s},0)|^2}\!\,ds\,\alpha(dt) \right]^{\nicefrac12} 
		\\
		& + 
		\sum_{k=0}^{n-1} 
		\frac{L(1+\sqrt{M})}{\sqrt{\lambda M^{(n-k)}}} 
		\left[ 
		\int_{[0,\infty)} e^{\lambda t}
		\int_t^{\infty} e^{-\lambda s} \, \Exp{|U^0_{k}(x+BW^{0}_{s})-u(x+BW^{0}_{s})|^2}\!\,ds 
		\,\alpha(dt)
		\right]^{\nicefrac12}
		. 
		\end{split}
		\end{equation} 	
	Fubini's theorem therefore implies for all 
		$ n \in \N $, 
		$ \theta \in \Theta $, 
		$ x \in \R^d $ 
	that 
		\begin{equation} 
		\begin{split}
		&\left[ \int_{[0,\infty)} \Exp{|U^{\theta}_n(x+BW^{\theta}_t) - u(x+BW^{\theta}_t)|^2}\!\,\alpha(dt) \right]^{\nicefrac12}
		\\
		&\leq 
		\frac{1}{\sqrt{\lambda M^n}} 
		\left[ 
		\int_0^{\infty} \left(\int_{[0,s]} e^{\lambda t}\,\alpha(dt)\right) e^{-\lambda s}\,\Exp{|f(x+BW^{0}_{s},0)|^2}\!\,ds \right]^{\nicefrac12} 
		\\
		& + 
		\sum_{k=0}^{n-1} 
		\frac{L(1+\sqrt{M})}{\sqrt{\lambda M^{(n-k)}}} 
		\left[ 
		\int_0^{\infty} \left(\int_{[0,s]} e^{\lambda t}\,\alpha(dt)\right) e^{-\lambda s} \, \Exp{|U^0_{k}(x+BW^{0}_{s})-u(x+BW^{0}_{s})|^2}\!\,ds 
		\right]^{\nicefrac12}. 
		\end{split}
		\end{equation} 	
	The assumption that for all $ A \in \Borel([0,\infty)) $ it holds that $\beta(A) \geq \int_A e^{-\lambda s} (\int_{[0,s]} e^{\lambda t}\,\alpha(dt))\, ds $ hence ensures for all 
		$ n \in \N $, 
		$ \theta \in \Theta $,
		$ x \in \R^d $ 
	that 
		\begin{equation} 
		\begin{split}
		& 
		\left[ \int_{[0,\infty)} \Exp{|U^{\theta}_n(x+BW^{\theta}_t) - u(x+BW^{\theta}_t)|^2}\!\,\alpha(dt) \right]^{\nicefrac12}
		\\&\qquad
		\leq 
		\frac{1}{\sqrt{\lambda M^n}} 
		\left[ 
		\int_{[0,\infty)} \Exp{|f(x+BW^{0}_{s},0)|^2}\!\,\beta(ds) \right]^{\nicefrac12} 
		\\
		&\qquad\qquad + 
		\sum_{k=0}^{n-1} 
		\frac{L(1+\sqrt{M})}{\sqrt{\lambda M^{(n-k)}}} 
		\left[ 
		\int_{[0,\infty)} \Exp{|U^0_{k}(x+BW^{0}_{s})-u(x+BW^{0}_{s})|^2} \beta(ds)
		\right]^{\nicefrac12}. 
		\end{split}
		\end{equation} 
	This establishes \eqref{recursive_estimate:claim}. The proof of \cref{lem:recursive_estimate} is thus completed. 
\end{proof}

\begin{prop}[Overall error estimate] \label{prop:combined_estimate}
	Assume  
		\cref{setting}, 
	let 
		$ L \in (0,\lambda) $, 
	assume for all 
		$ x \in \R^d $, 
		$ v,w \in \R $ 
	that 
		$ | f( x , v ) - f( x , w ) | \leq L | v - w | $ 
	and 
		$\int_0^{\infty} e^{(L-\lambda) s} \, \EXP{|f(x+BW^{0}_s,0)|}\,ds<\infty$, 
	let 
		$u \colon \R^d \to \R$ be $ \Borel(\R^d) $/$ \Borel(\R) $-measurable, 
	and assume for all 
		$ x \in \R^d $ 
	that 
		$ \int_0^{\infty} e^{(L-\lambda) s} \, \EXP{|u(x+BW^{0}_s)|^2}\,ds < \infty $
	and 
		\begin{equation} \label{combined_estimate:exact_solution}
			u(x)
			= 
			\frac{1}{\lambda} 
			\Exp{ 
				f(x+BW^{0}_{R^0},u(x+BW^{0}_{R^0})) } 
			= 
			\Exp{
				\int_0^{\infty} 
				e^{-\lambda s} f(x+BW^{0}_s,u(x+BW^{0}_s))\,ds}\!.
		\end{equation}
	Then it holds for all 
		$ n \in \N_0 $, 
		$ \theta \in \Theta $, 
		$ x \in \R^d $ 
	that 
		\begin{equation} \label{overall_error_analysis:claim}
		\begin{split}
		& 
		\left(\Exp{|U^{\theta}_{n}(x)-u(x)|^2}\right)^{\!\nicefrac12}
		\\
		& \leq 
		\frac{1}{\sqrt{\lambda}-\sqrt{L}}
		\left( 
		\int_0^{\infty} e^{(L-\lambda) s}\,\Exp{|f(x+BW^{0}_s,0)|^2}\!\,ds 
		\right)^{\!\nicefrac12}
		\frac{1}{\sqrt{M^n}} 
			\left[ 1 + (1+\sqrt{M})\sqrt{\frac{L}{\lambda}} \right]^{n}.   
		\end{split}
		\end{equation}	
\end{prop} 
	
\begin{proof}[Proof of \cref{prop:combined_estimate}]
	Throughout this proof let 
		$ \delta_0 \colon \Borel([0,\infty)) \to [0,\infty) $ 
	satisfy for all 
		$ A \in \Borel([0,\infty)) $ 
	that 
		\begin{equation}  
		\delta_0(A) 
		= 
		\begin{cases}
		1 & \colon 0 \in A \\
		0 & \colon 0 \notin A. 
		\end{cases}  
		\end{equation} 
	Observe that 
		\begin{enumerate}[(i)]
			\item\label{combined_estimate:proof_item1} for all 
				$ A \in \Borel([0,\infty)) $ 
			it holds that 
				\begin{equation} 
				\int_A 
				e^{-\lambda s} 
				\int_{[0,s]} e^{\lambda t} \,\delta_0(dt) \,ds
				= \int_A e^{-\lambda s} \,ds 
				\leq \int_A e^{(L-\lambda) s}\,ds
				\end{equation}
			and 
			\item\label{combined_estimate:proof_item2} for all 
				$ A \in \Borel([0,\infty)) $ 
			it holds that 
				\begin{equation} 
				\int_A 
				e^{-\lambda s} 
				\int_{[0,s]} e^{\lambda t}\, e^{(L-\lambda)t} \,dt \,ds 
				= 
				\int_A e^{-\lambda s} \frac{e^{Ls}-1}{L}\,ds 
				\leq 
				\frac{1}{L}	\int_A {e^{(L-\lambda) s}}\,ds. 
				\end{equation} 
		\end{enumerate}			
	Note that Item~\eqref{combined_estimate:proof_item2} and \cref{lem:recursive_estimate} (with 
		$ \varepsilon \is L $, 
		$ L \is L $, 
		$ u \is u $, 
		$ \alpha \is ( \Borel([0,\infty)) \ni A \mapsto \int_A e^{(L-\lambda)t}\,dt \in [0,\infty) ) $,  
		$ \beta \is ( \Borel([0,\infty)) \ni A \mapsto \frac{1}{L} \int_A e^{(L-\lambda)t}\,dt \in [0,\infty)) $ 
	in the notation of \cref{lem:recursive_estimate}) ensure for all 
		$ n \in \N $, 
		$ x \in \R^d $ 
	that 
		\begin{equation} \label{combined_estimate:recursive_inequalities}
		\begin{split}
		& 
		\left(
		\int_0^{\infty} e^{(L-\lambda)t} \,
		\Exp{|U^0_{n}(x+BW^{0}_t)-u(x+BW^{0}_t)|^2}\!\,dt
		\right)^{\!\nicefrac12}
		\\
		& \qquad \leq 
		\frac{1}{\sqrt{L\lambda M^n}} 
		\left( 
		\int_0^{\infty} e^{(L-\lambda) t} \, \Exp{|f(x+BW^{0}_{t},0)|^2}\!\,dt
		\right)^{\!\nicefrac12}
		\\
		& \qquad\qquad + \sum_{k=0}^{n-1} 
		\sqrt{\frac{L}{\lambda}} 
		\frac{1+\sqrt{M}}{\sqrt{M^{(n-k)}}} 
		\left(  
		\int_0^{\infty} e^{(L-\lambda)t}\,\Exp{|U^0_k(x+BW^{0}_t)-u(x+BW^{0}_t)|^2}\!\,dt
		\right)^{\!\nicefrac12}. 
		\end{split}
		\end{equation} 
	For the next step let 
		$ \eta^{(x)}_n \in [0,\infty] $, $ n \in \N_0 $, $ x \in \R^d $, 
	satisfy for all 
		$ n \in \N_0 $, 
		$ x \in \R^d $
	that 
		\begin{equation} 
		\eta_n^{(x)} = \sqrt{M^n} \left( \int_0^{\infty} e^{(L-\lambda)t} \, \Exp{|U^0_n(x+BW^0_t)-u(x+BW^0_t)|^2}\!\,dt \right)^{\!\nicefrac12}
		\end{equation} 
	and let $ a_1^{(x)},a_2^{(x)} \in [0,\infty] $, $ x \in \R^d $, 
	satisfy for all 
		$ x \in \R^d $ 
	that 
		\begin{equation} 
		a_1^{(x)} = \frac{1}{\sqrt{L\lambda}} \left( \int_0^{\infty} e^{(L-\lambda)t} \, \Exp{|f(x+BW^0_t,0)|^2}\!\,dt\right)^{\!\nicefrac12}
		\qandq 
		a_2^{(x)} =  (1 + \sqrt{M})\sqrt{\frac{L}{\lambda}}. 
		\end{equation} 
	Note that \eqref{combined_estimate:recursive_inequalities} ensures for all 
		$ n \in \N $, 
		$ x \in \R^d $ 
	that 
		\begin{equation} 
		\eta^{(x)}_n \leq a_1^{(x)} + a_2^{(x)} \sum_{k=0}^{n-1} \eta^{(x)}_k. 
		\end{equation} 
	The discrete Gronwall inequality therefore proves for all 
		$ n \in \N $, 
		$ x \in \R^d $ 
	that 
		\begin{equation} \label{combined_estimate:gronwall}
		\eta^{(x)}_n \leq ( a^{(x)}_1 + a^{(x)}_2\eta^{(x)}_0 )	( 1 + a_2^{(x)})^{n-1}. 
		\end{equation} 
	Next observe that Item~\eqref{a_priori_estimate:item2} in \cref{lem:a_priori_estimate_exact_solution} (with 
		$ d \is d $, 
		$ m \is d $, 
		$ B \is B $,
		$ L \is L $, 
		$ \lambda \is \lambda $, 
		$ \varepsilon \is L $, 
		$ (\Omega,\mc F,\P) \is (\Omega,\mc F,\P) $,
		$ W \is W^0 $,
		$ f \is f $, 
		$ u \is u $
	in the notation of \cref{lem:a_priori_estimate_exact_solution}) demonstrates for all 
		$ x \in \R^d $ 
	that 
		\begin{equation} 
		\begin{split}
		\eta^{(x)}_0 & = \left( \int_0^{\infty} e^{(L-\lambda)t} \, \Exp{|u(x+BW^0_t)|^2}\!\,dt \right)^{\!\nicefrac12} 
		\\
		& \leq 
		\frac{1}{\sqrt{L\lambda}-L} 
		\left( \int_0^{\infty} e^{(L-\lambda)t} \, \Exp{|f(x+BW^0_t,0)|^2} \!\,dt \right)^{\!\nicefrac12} 
		= 
		\frac{a_1^{(x)}}{ 1 - \sqrt{\frac{L}{\lambda}}}.  
		\end{split}
		\end{equation} 
	This and \eqref{combined_estimate:gronwall} ensure for all 
		$ n \in \N_0 $, 
		$ x \in \R^d $ 
	that $ \eta^{(x)}_n \leq [1-\sqrt{\nicefrac{L}{\lambda}}]^{-1} a_1^{(x)} (1+a_2^{(x)})^n $. 
	Hence, we obtain for all 
		$ n \in \N_0 $, 
		$ x \in \R^d $ 
	that 
		\begin{equation} \label{combined_estimate:eq01}
		\begin{split}
		& 
		\left(
		\int_0^{\infty} e^{(L-\lambda)t} \,
		\Exp{|U^0_{n}(x+BW^{0}_t)-u(x+BW^{0}_t)|^2}\!\,dt
		\right)^{\!\nicefrac12}
		\\
		& \leq 
		\frac{1}{\sqrt{M^n}}
		\left[ 1 + (1+\sqrt{M})\sqrt{\frac{L}{\lambda}} \right]^n
		\frac{1}{\sqrt{L\lambda}-L} 
		\left(  
		\int_0^{\infty} e^{(L-\lambda)t} \, \Exp{|f(x+BW^{0}_t,0)|^2}\!\,dt 
		\right)^{\!\nicefrac12}.
		\end{split}
		\end{equation}	
	Next note that Item~\eqref{combined_estimate:proof_item1} and \cref{lem:recursive_estimate} (with 
		$ \varepsilon \is L $, 
		$ L \is L $, 
		$ u \is u $,
		$ \alpha \is \delta_0 $, 
		$ \beta \is (\Borel([0,\infty)) \ni A \mapsto \int_A e^{(L-\lambda)s}\,ds \in [0,\infty))$ 
	in the notation of \cref{lem:recursive_estimate}) imply for all 
		$ n \in \N $, 
		$ \theta \in \Theta $, 
		$ x \in \R^d $ 
	that 
		\begin{equation}
		\begin{split}
		& 
		\left(\Exp{|U^{\theta}_{n}(x)-u(x)|^2}\right)^{\!\nicefrac12}
		\leq 
		\frac{1}{\sqrt{\lambda M^n}} 
		\left( 
		\int_0^{\infty} e^{(L-\lambda) s}\,\Exp{|f(x+BW^{0}_s,0)|^2}\!\,ds 
		\right)^{\!\nicefrac12}
		\\
		& + 
		\sum_{k=0}^{n-1} \frac{L(1+\sqrt{M})}{\sqrt{\lambda M^{(n-k)}}}
		\left( 
		\int_0^{\infty} 
		e^{(L-\lambda) s} \, \Exp{|U^0_{k}(x+BW^{0}_s)-u(x+BW^{0}_s)|^2}\!\,ds 
		\right)^{\!\nicefrac12}.
		\end{split}
		\end{equation}
	This and \eqref{combined_estimate:eq01} demonstrate for all 
		$ n \in \N $, 
		$ \theta \in \Theta $, 
		$ x \in \R^d $ 
	that 
		\begin{align}
		& 
		\left(\Exp{|U^{\theta}_{n}(x)-u(x)|^2}\right)^{\!\nicefrac12} 
		\leq 
		\frac{1}{\sqrt{\lambda M^n}} 
		\left( 
		\int_0^{\infty} e^{(L-\lambda) s}\,\Exp{|f(x+BW^{0}_s,0)|^2}\!\,ds 
		\right)^{\!\nicefrac12}
		\\ \nonumber
		& + 
		\sum_{k=0}^{n-1} \frac{L(1+\sqrt{M})}{\sqrt{\lambda M^{n}}}
		\frac{1}{\sqrt{L\lambda}-L}
		\left(  
		\int_0^{\infty} e^{(L-\lambda)s}\, \Exp{|f(x+BW^{0}_s,0)|^2}\!\,ds 
		\right)^{\!\nicefrac12}
		\left[ 1 + (1+\sqrt{M})\sqrt{\frac{L}{\lambda}} \right]^k. 
		\end{align} 
	Hence, we obtain for all 
		$ n \in \N $, 
		$ \theta \in \Theta $, 
		$ x \in \R^d $ 
	that 
	\begin{equation}
	\begin{split} 
	\left(\Exp{|U^{\theta}_{n}(x)-u(x)|^2}\right)^{\!\nicefrac12}
	& \leq 
	\frac{1}{\sqrt{\lambda M^n}}
	\left( 
	\int_0^{\infty} e^{(L-\lambda) s}\,\Exp{|f(x+BW^{0}_s,0)|^2}\!\,ds 
	\right)^{\!\nicefrac12}
	\\
	& \qquad \cdot
	\left[ 
	1 + 
	\sum_{k=0}^{n-1} \frac{L(1+\sqrt{M})}{\sqrt{L\lambda}-L}
	\left( 1 + (1+\sqrt{M})\sqrt{\frac{L}{\lambda}} \right)^{\!k}  
	\right]\!. 
	\end{split}
	\end{equation}
	The fact that for all 
		$ q \in \R $, 
		$ n \in \N $ 
	it holds that 
		$ (q-1) \sum_{k=0}^{n-1} q^k = q^n-1 $ 
	therefore implies for all 
		$ n \in \N $, 
		$ \theta \in \Theta $, 
		$ x \in \R^d $ 
	that 
	\begin{equation}
	\begin{split}
	\left(\Exp{|U^{\theta}_{n}(x)-u(x)|^2}\right)^{\!\nicefrac12}
	& \leq 
	\frac{1}{\sqrt{\lambda M^n}}
	\left( 
	\int_0^{\infty} e^{(L-\lambda) s}\,\Exp{|f(x+BW^{0}_s,0)|^2}\!\,ds 
	\right)^{\!\nicefrac12}
	\\
	& \cdot \left[ 1 + \frac{\sqrt{L\lambda}}{\sqrt{L\lambda}-L} \left( 1 + (1+\sqrt{M})\sqrt{\frac{L}{\lambda}} \right)^{\!n} - \frac{\sqrt{L\lambda}}{\sqrt{L\lambda}-L} \right]\!.
	\end{split}
	\end{equation}
	The fact that $ \frac{\sqrt{L\lambda}}{\sqrt{L\lambda}-L}>1$ hence implies for all 
		$ n \in \N $, 
		$ \theta \in \Theta $, 
		$ x \in \R^d $
	that 
		\begin{equation}
		\begin{split}
		& 
		\left(\Exp{|U^{\theta}_{n}(x)-u(x)|^2}\right)^{\!\nicefrac12}
		\\
		& \leq
		\frac{1}{\sqrt{\lambda M^n}}
		\left( 
		\int_0^{\infty} e^{(L-\lambda) s}\,\Exp{|f(x+BW^{0}_s,0)|^2}\!\,ds 
		\right)^{\!\nicefrac12}
		\frac{\sqrt{L\lambda}}{\sqrt{L\lambda}-L} 
		\left[ 1 + (1+\sqrt{M})\sqrt{\frac{L}{\lambda}} \right]^{\!n} 
		\\ 
		& =
		\frac{1}{\sqrt{\lambda}-\sqrt{L}}
		\left( 
		\int_0^{\infty} e^{(L-\lambda) s}\,\Exp{|f(x+BW^{0}_s,0)|^2}\!\,ds 
		\right)^{\!\nicefrac12}
		\frac{1}{\sqrt{M^n}} 
		\left[ 1 + (1+\sqrt{M})\sqrt{\frac{L}{\lambda}} \right]^{n}.   
		\end{split}
		\end{equation}
	Combining this with the fact that for all 
		$ \theta \in \Theta $, 
		$ x \in \R^d $ 
	it holds that $ U^{\theta}_0(x) = 0 $ and Item~\eqref{a_priori_estimate:item3} in \cref{lem:a_priori_estimate_exact_solution} establishes \eqref{overall_error_analysis:claim}. 
	The proof of \cref{prop:combined_estimate} is thus completed. 
\end{proof}

\begin{cor} \label{cor:L_zero_case}
	Assume \cref{setting},
	let $ \varepsilon \in (0,\infty) $, 
	assume for all 
		$ x \in \R^d $, 
		$ v \in \R $ 
	that 
		$ f(x,v) = f(x,0) $
	and 
		$ \int_0^{\infty} e^{(\varepsilon-\lambda)s}\,\EXP{|f(x+BW^{0}_s,0)|^2}\,ds < \infty $, 
	let 
		$ u \colon \R^d \to \R $ 
	be $ \Borel(\R^d) $/$ \Borel(\R) $-measurable, 
	and assume for all 
		$ x \in \R^d $ 
	that 
		\begin{equation} \label{L_zero_case:fixed_point_equation}
		u(x) = 
		\Exp{ \int_0^{\infty} e^{-\lambda s} f(x+BW^{0}_s,u(x+BW^{0}_s))\,ds} 
		= 
		\Exp{\int_0^{\infty} e^{-\lambda s} 
		f(x+BW^{0}_s,0)\,ds}\!.
		\end{equation} 
	Then it holds for all 
		$ n \in \N_0 $, 
		$ \theta \in \Theta $, 
		$ x \in \R^d $ 
	that 
		\begin{equation} \label{L_zero_case:claim}
		\left( \Exp{|U^{\theta}_n(x) - u(x) |^2} \right)^{\!\nicefrac12} 
		\leq 
		\frac{1}{\sqrt{\lambda M^n}} 
		\left( 
		\int_0^{\infty} e^{-\lambda s}\,\Exp{|f(x+BW^{0}_s,0)|^2}\!\,ds \right)^{\!\nicefrac12}.  
		\end{equation} 
\end{cor} 

\begin{proof}[Proof of \cref{cor:L_zero_case}]
	Throughout this proof let $ \nu_{t,x} \colon \Borel(\R^d) \to [0,\infty) $, $ t\in [0,\infty) $, $ x \in \R^d $, satisfy for all 
		$ t \in [0,\infty) $, 
		$ x \in \R^d $, 
		$ A \in \Borel(\R^d) $ 
	that 
		$ \nu_{t,x}(A) = \P(x+BW^0_t\in A) $. 
	Observe that the Cauchy-Schwarz inequality and \eqref{L_zero_case:fixed_point_equation} ensure for all 
		$ y \in \R^d $ 
	that 
		\begin{equation} 
		|u(y)|^2 \leq \frac{1}{\lambda} \int_0^{\infty} e^{-\lambda s} \, \Exp{|f(y+BW^0_s,0)|^2}\!\,ds. 
		\end{equation} 
	Fubini's theorem therefore shows for all 
		$ x \in \R^d $ 
	that 
		\begin{equation} 
		\begin{split} 
		& \int_0^{\infty} e^{(\varepsilon-\lambda)t} \, \Exp{|u(x+BW^0_t)|^2}\!\,dt 
		= 
		\int_0^{\infty} e^{(\varepsilon-\lambda)t} \int_{\R^d} 
		|u(y)|^2\,\nu_{t,x}(dy)\,dt
		\\
		& \leq 
		\int_0^{\infty} e^{(\varepsilon-\lambda)t}\,
		\int_{\R^d} \frac{1}{\lambda} \int_0^{\infty} e^{-\lambda s}\,\Exp{|f(y+BW^0_s,0)|^2}\!\,ds\,\nu_{t,x}(dy)\,dt
		\\
		& = 
		\frac{1}{\lambda} 
		\int_0^{\infty} e^{(\varepsilon-\lambda)t} 
		\int_0^{\infty} e^{-\lambda s} \, \int_{\R^d} \Exp{|f(y+BW^0_s,0)|^2}\!\,\nu_{t,x}(dy)\,ds\,dt . 
		\end{split}
		\end{equation}
	The fact that for all 
		$ t,s \in [0,\infty) $, 
		$ x \in \R^d $, 
		$ A \in \Borel(\R^d) $ 
	it holds that 
		$ \nu_{t+s,x}(A) = \int_{\R^d} \nu_{s,y}(A)\,\nu_{t,x}(dy) $ 
	and Fubini's theorem hence ensure that for all 
		$ x \in \R^d $ 
	it holds that		
		\begin{equation} 
		\begin{split}
		\int_0^{\infty} e^{(\varepsilon-\lambda)t} \, \Exp{|u(x+BW^0_t)|^2}\!\,dt 
		& 
		\leq  
		\frac{1}{\lambda} 
		\int_0^{\infty} e^{(\varepsilon-\lambda)t} 
		\int_0^{\infty} e^{-\lambda s} \, \Exp{|f(x+BW^0_{t+s},0)|^2}\!\,ds\,dt 
		\\
		& = 
		\frac{1}{\lambda} 
		\int_0^{\infty} e^{\varepsilon t} 
		\int_t^{\infty} e^{-\lambda s} \, \Exp{|f(x+BW^0_{s},0)|^2}\!\,ds\,dt 
		\\
		& = 
		\frac{1}{\lambda} 
		\int_0^{\infty} \left[ \int_0^s e^{\varepsilon t} \,dt\right]\!\,
		e^{-\lambda s} \, \Exp{|f(x+BW^0_s,0)|^2}\!\,ds 
		\\
		& \leq 
		\frac{1}{\varepsilon\lambda} 
		\int_0^{\infty} e^{(\varepsilon-\lambda)s} \, \Exp{|f(x+BW^0_s,0)|^2}\!\,ds < \infty. 
		\end{split}
		\end{equation} 
	\cref{prop:combined_estimate} (with 
		$ L \is L $, 
		$ u \is u $
	for $ L \in (0,\min\{\varepsilon,\lambda\}) $ 
	in the notation of \cref{prop:combined_estimate}) therefore proves for all 
		$ L \in (0,\min\{\varepsilon,\lambda\}) $, 
		$ n \in \N_0 $, 
		$ \theta \in \Theta $, 
		$ x \in \R^d $ 
	that 
		\begin{equation} 
		\begin{split}
		&
		\left( \Exp{| U^{\theta}_n(x) - u(x) |^2} \right)^{\!\nicefrac12} 
		\\
		&\leq 
		\frac{1}{\sqrt{\lambda}-\sqrt{L}} 
		\left( 
		 \int_0^{\infty} e^{(L-\lambda)s} \,\Exp{|f(x+BW^{0}_s,0)|^2}\!\,ds 
		\right)^{\!\nicefrac12} 
		\frac{1}{\sqrt{M^n}} \left[ 
		1 + (1+\sqrt{M})\sqrt{\frac{L}{\lambda}}
		\right]^n.   
		\end{split}
		\end{equation} 
	Lebesgue's dominated convergence theorem hence ensures for all 
		$ n \in \N_0 $, 
		$ \theta \in \Theta $,
		$ x \in \R^d $ 
	that 
		\begin{equation} 
		\begin{split}
		& 		\left( \Exp{| U^{\theta}_n(x) - u(x) |^2} \right)^{\!\nicefrac12} 
		\\ 
		& 
		\leq 
		\lim_{ L \searrow 0 }
		\left(
		\frac{1}{\sqrt{\lambda}-\sqrt{L}} 
		\left( 
		\int_0^{\infty} e^{(L-\lambda)s} \,\Exp{|f(x+BW^{0}_s,0)|^2}\!\,ds 
		\right)^{\!\nicefrac12} 
		\frac{1}{\sqrt{M^n}} \left[ 
		1 + (1+\sqrt{M})\sqrt{\frac{L}{\lambda}}
		\right]^n \right)
		\\ 
		& =    
		\frac{1}{\sqrt{\lambda M^n}} 
			\left( 
		\int_0^{\infty} e^{-\lambda s} \,\Exp{|f(x+BW^{0}_s,0)|^2}\!\,ds 
		\right)^{\!\nicefrac12}. 
		\end{split} 
		\end{equation} 
	This establishes \eqref{L_zero_case:claim}. The proof of \cref{cor:L_zero_case} is thus completed. 
\end{proof} 

\begin{cor} \label{cor:all_cases}
	Assume \cref{setting}, 
	let $ L \in [0,\lambda) $, 
		$ \eta \in (0,\infty) $, 	
	assume for all 
		$ x \in \R^d $, 
		$ v,w \in \R $ 
	that 
		$ | f( x, v ) - f( x, w ) | \leq L | v - w | $ 
	and 
		$ \int_0^{\infty} e^{(\max\{L,\eta\}-\lambda) s} \, \EXP{| f( x + BW^0_s, 0) |^2 } \, ds < \infty $, 
	let 
		$ u \colon \R^d \to \R $ 
	be 
		$ \Borel(\R^d) $/$ \Borel(\R) $-measurable, 
	and assume for all 
		$ x \in \R^d $ 
	that 
		$ \int_0^{\infty} e^{(L-\lambda)s} \, \EXP{|u(x+BW^0_s)|^2}\, ds < \infty $ 
	and 
		\begin{equation} 
		u(x) = \frac{1}{\lambda} \Exp{f(x+BW^0_{R^0},u(x+BW^0_{R^0}))} 
		= 
		\int_0^{\infty} e^{-\lambda s}\,\Exp{f(x+BW^0_s,u(x+BW^0_s))}\!\,ds.   
		\end{equation} 
	Then it holds for all 
		$ \theta \in \Theta $, 
		$ n \in \N_0 $, 
		$ x \in \R^d $ 
	that 
		\begin{equation} \label{all_cases:claim}
		\begin{split}
		& 
		\left(\Exp{|U_{n}^{\theta}(x)-u(x)|^2}\right)^{\!\nicefrac12} 
		\\
		& 
		\leq 
		\frac{1}{\sqrt{\lambda}-\sqrt{L}}\left( \int_0^{\infty} e^{(L-\lambda)s}\,\Exp{|f(x+BW^0_s,0)|^2}\!\,ds \right)^{\!\nicefrac12} \frac{1}{\sqrt{M^n}}\left[ 1 + (1
		+\sqrt{M})\sqrt{\frac{L}{\lambda}} \right]^n .
		\end{split}
		\end{equation}  
	\end{cor} 

\begin{proof}[Proof of \cref{cor:all_cases}] 
	First, observe that the assumption that for all 
		$ x \in \R^d $ 
	it holds that 
		$ \int_0^{\infty} e^{(\max\{L,\eta\}-\lambda)s} \,\EXP{|f(x+BW^0_s,0)|^2}\,ds < \infty $  
	guarantees for all 
		$ x \in \R^d $ 
	that 
		$ \int_0^{\infty} e^{(L-\lambda)s} \, \EXP{|f(x+BW^0_s,0)|^2}\,ds < \infty $ 
	and 
		$ \int_0^{\infty} e^{(\eta -\lambda)s} \, \EXP{|f(x+BW^0_s,0)|^2}\,ds < \infty $. 
	Therefore, we obtain that \cref{prop:combined_estimate} (with 
		$ L \is L $, 
		$ u \is u $ 
	in the notation of \cref{prop:combined_estimate}) establishes \eqref{all_cases:claim} in the case $ L \in (0,\lambda) $ and \cref{cor:L_zero_case} (with 
		$ \varepsilon \is \eta $, 
		$ u \is u $
	in the notation of \cref{cor:L_zero_case}) establishes \eqref{all_cases:claim} in the case $ L = 0 $. The proof of \cref{cor:all_cases} is thus completed.  
\end{proof} 

\begin{cor} \label{cor:error_estimate_viscosity_solutions}
	Assume \cref{setting}, 
	let	$ L,\rho \in [0,\infty) $,
	assume that
		$ \lambda \in (L+2\rho,\infty) $, 
	let $ \norm{\cdot} \colon \R^d \to [0,\infty) $ be the standard norm on $ \R^d $, 
	let $ u \in C(\R^d,\R) $, $ V \in C^2(\R^d,(0,\infty)) $ 
	satisfy for all 
		$ x \in \R^d $, 
		$ v,w \in \R $ 
	that 
		$ | f(x,v) - f(x,w) | \leq L | v-w | $ 
	and 
	\begin{equation} \label{error_estimate_viscosity_solutions:lyapunov}
	\operatorname{Trace}(BB^{*}(\operatorname{Hess} V)(x))+\tfrac{\norm{B^{*}(\nabla V)(x)}^2}{V(x)} \leq 2\rho V(x),
	\end{equation} 
	assume that 
	$ \sup_{r\in (0,\infty)} [\inf_{\norm{x}>r} V(x)] = \infty $ 
	and 
		$ \inf_{r\in (0,\infty)} [\sup_{\norm{x}>r} (\frac{|f(x,0)|+|u(x)|}{V(x)})] = 0 $, 
	and assume that $ u $ is a viscosity solution of 
	\begin{equation} 
		\lambda u(x) - \tfrac12\operatorname{Trace}\!\big(BB^{*}(\operatorname{Hess} u)(x)\big) =  f(x,u(x)) 
	\end{equation} 
	for $ x \in \R^d $ (cf.~\cref{cor:existence_and_uniqueness}). 
	Then it holds for all 
		$ \theta \in \Theta $, 
		$ n \in \N_0 $, 
		$ x \in \R^d $ 
	that 
		\begin{equation} \label{error_estimate_viscosity_solutions:claim}
		\begin{split}
		&\left( \Exp{\left| \frac{U^{\theta}_{n}(x) - u(x) }{V(x)}\right|^2 }\right)^{\!\nicefrac12} 
		\\
		& \leq  
		\frac{1}{(\sqrt{\lambda}-\sqrt{L})\sqrt{\lambda-(L+2\rho)}} \left[\sup_{y\in\R^d}\left( \frac{|f(y,0)|}{V(y)} \right) \right]
				\left[ 
		\frac{1}{\sqrt{M}}
		\left( 
		1 + (1 + \sqrt{M})\sqrt{\frac{L}{\lambda}}
		\right) 		
		\right]^n. 
		\end{split}
		\end{equation}  
\end{cor}

\begin{proof}[Proof of \cref{cor:error_estimate_viscosity_solutions}] 
	Throughout this proof let $ c\in \R$ satisfy for all 
		$ x \in \R^d $ 
	that 
		$ |u(x)| \leq c V(x) $. 
	Note that \eqref{error_estimate_viscosity_solutions:lyapunov} ensures that 
	\begin{enumerate} [(I)]
		\item \label{error_estimate_viscosity_solutions:proof_item1}
		for all 
			$ x \in \R^d $ 
		it holds that 
			$ \frac12\operatorname{Trace}(BB^{*}(\operatorname{Hess} V)(x)) \leq \rho V(x) $
		and 
		\item \label{error_estimate_viscosity_solutions:proof_item2} 
		for all 
			$ x \in \R^d $ 
		it holds that 
			\begin{equation} 
			\begin{split}
			& 
			\tfrac12\operatorname{Trace}\!\big(BB^{*}(\operatorname{Hess} V^2)(x) \big) 
			\\
			& 
			=
			\tfrac12\operatorname{Trace}\!\big(BB^{*}(
			2 V(x) (\operatorname{Hess} V)(x) + 2 (\nabla V)(x)\otimes (\nabla V)(x) ) \big) 
			\\& 
			= 
			V(x) \left[ 
			\operatorname{Trace}\!\big(BB^{*}(\operatorname{Hess} V)(x) \big) 
			+ 
			\tfrac{\norm{B^{*}(\nabla V)(x)}^2}{V(x)}
			\right] \leq 2 \rho [V(x)]^2. 
			\end{split} 
			\end{equation} 
	\end{enumerate}
	Item~\eqref{error_estimate_viscosity_solutions:proof_item1} and \cref{cor:existence_and_uniqueness} (with 
		$ d \is d $, 
		$ m \is d $, 
		$ B \is B $, 
		$ L \is L $, 
		$ \rho \is \rho $, 
		$ \lambda \is \lambda $, 
		$ \norm{\cdot} \is \norm{\cdot} $, 
		$ f \is f $, 
		$ V \is V $, 
		$ (\Omega,\mc F,\P) \is (\Omega,\mc F,\P) $, 
		$ W \is W^0 $
	in the notation of \cref{cor:existence_and_uniqueness}) demonstrate for all 
		$ x \in \R^d $ 
	that 
		\begin{equation} \label{comp_cost:exact_solution}
			u(x) = \Exp{\int_0^{\infty} e^{-\lambda s} f(x+BW^{0}_s,u(x+BW^{0}_s))\,ds}\!.
		\end{equation} 	
	Next note that \cite[Lemmas 3.1 and 3.2]{StochasticFixedPointEquations} and  Item~\eqref{error_estimate_viscosity_solutions:proof_item2} guarantee for all 
		$ t \in [0,\infty) $, 
		$ x \in \R^d $
	that 
		\begin{equation} \label{comp_cost_viscosity_solutions:Lyapunov_squared}
		\EXPP{ |V(x+BW^0_t)|^2 } \leq e^{2\rho t} [V(x)]^2 .
		\end{equation} 
	The fact that $ \sup_{y\in\R^d} (\frac{|f(y,0)|}{V(y)}) < \infty $ hence implies for all 
		$ x \in \R^d $, 
		$ \eta \in [0,\infty) $
	with 
		$ L+2\rho+\eta < \lambda $ 
	that 
		\begin{equation} \label{comp_cost_viscosity_solutions:estimate_on_f}
		\begin{split}
		& \int_0^{\infty} e^{(L+\eta-\lambda)s} \, 
		\Exp{|f(x+BW^0_s,0)|^2}\!\,ds 
		\\& 
		\leq 
		\int_0^{\infty} e^{(L+\eta-\lambda)s} 
		\left[\sup_{y\in\R^d}\left( \frac{|f(y,0)|}{V(y)} \right) \right]^2 \, \Exp{|V(x+BW^0_s)|^2}\!\,ds 
		\\& 
		\leq 
		\left[\sup_{y\in\R^d}\left( \frac{|f(y,0)|}{V(y)} \right) \right]^2  
		\left[ \int_0^{\infty} e^{(L+\eta+2\rho-\lambda)s} \,ds \right] |V(x)|^2 
		\\& 
		= 
		\frac{|V(x)|^2}{\lambda-(L+\eta+2\rho)} \left[\sup_{y\in\R^d}\left( \frac{|f(y,0)|}{V(y)} \right) \right]^2  
		< \infty. 
		\end{split} 
		\end{equation} 
	In addition, note that the fact that for all 
		$ x \in \R^d $ 
	it holds that $ |u(x)| \leq c V(x) $ and \eqref{comp_cost_viscosity_solutions:Lyapunov_squared} prove for all 
		$ x \in \R^d $ 
	that 
	\begin{equation} \label{comp_cost_special_case:square_integrability_for_u}
	\begin{split}
	\int_{0}^{\infty} e^{(L-\lambda)s}\,\Exp{|u(x+BW^0_s)|^2}\!\,ds 
	& \leq 
	\int_0^{\infty} 
	e^{(L-\lambda)s}\,c^2 \, \Exp{|V(x+BW^0_s)|^2}\!\,ds
	\\
	& \leq 
	c^2 |V(x)|^2\int_0^{\infty} 
	e^{(L+2\rho -\lambda)s} \,ds 
	= 
	\frac{[cV(x)]^2}{\lambda-(L+2\rho)} < \infty.  
	\end{split} 
	\end{equation} 
	This, \eqref{comp_cost:exact_solution}, \eqref{comp_cost_viscosity_solutions:estimate_on_f}, and \cref{cor:all_cases} assure for all 
		$ \theta \in \Theta $, 
		$ n \in \N_0 $, 
		$ x \in \R^d $ 
	that 
	\begin{equation} 
	\begin{split}
	&\left(\Exp{|U^{\theta}_{n}(x)-u(x)|^2}\right)^{\!\nicefrac12}
	\\
	& \leq 
	\frac{1}{\sqrt{\lambda}-\sqrt{L}}
	\left( 
	\int_0^{\infty} e^{(L-\lambda) s}\,\Exp{|f(x+BW^0_s,0)|^2}\!\,ds 
	\right)^{\!\nicefrac12} 
	\frac{1}{\sqrt{M^n}}
	\left[ 
	1 + (1 + \sqrt{M})\sqrt{\frac{L}{\lambda}}
	\right]^n
	\\
	& \leq 
	\frac{1}{\sqrt{\lambda}-\sqrt{L}}
	\frac{V(x)}{\sqrt{\lambda-(L+2\rho)}} \left[\sup_{y\in\R^d}\left( \frac{|f(y,0)|}{V(y)} \right) \right]
	\frac{1}{\sqrt{M^n}}
	\left[ 
	1 + (1 + \sqrt{M})\sqrt{\frac{L}{\lambda}}
	\right]^n
	.
	\end{split}
	\end{equation}
	This establishes \eqref{error_estimate_viscosity_solutions:claim}. The proof of \cref{cor:error_estimate_viscosity_solutions} is thus completed. 
\end{proof}

\begin{cor} \label{cor:error_estimate_semilinear_elliptic}
	Let $ d,M \in \N $, 
		$ B \in \R^{d\times d} $, 
		$ L \in \R $, 
		$ \lambda \in (L,\infty) $,
		$ \Theta = \cup_{n\in\N} \Z^n $, 
		$ f \in C(\R^d\times\R,\R) $ 
	satisfy for all 
		$ x \in \R^d $, 
		$ v,w \in \R $ 
	that 
		$ | f(x,v) - f(x,w) - \lambda ( v - w ) | \leq L | v-w | $, 
	assume that $ f $ is at most polynomially growing, 
	let $ \norm{\cdot} \colon \R^d \to [0,\infty) $ be the standard norm on $ \R^d $, 
	let $ ( \Omega, \mc F, \P ) $ be a probability space, 
	let $ W^{\theta} \colon [0,\infty) \times \Omega \to \R^d $, $ \theta\in\Theta$, be i.i.d.~standard Brownian motions, 
	let $ R^{\theta} \colon \Omega \to [0,\infty) $, $ \theta \in \Theta $, be i.i.d.~random variables, 
	assume for all 
		$ t\in [0,\infty) $ 
	that 	
		$ \P(R^0 \geq t) = e^{-\lambda t} $, 
	assume that 
		$ (R^{\theta})_{\theta\in\Theta} $ and $ (W^{\theta})_{\theta\in\Theta}$ are independent, 
	let $ U^{\theta}_{n} = (U^{\theta}_{n}(x))_{x\in\R^d}\colon\R^d\times\Omega\to\R$, $\theta\in\Theta$, $n\in\N_0$, 
	satisfy for all 
		$ \theta\in\Theta $, 
		$ n \in \N $, 
		$ x \in \R^d $
	that 
		$ U^{\theta}_{0}(x) = 0 $ 
	and 
		\begin{equation} \label{error_estimate_semilinear_elliptic:mlp_scheme}
		\begin{split}
		U^{\theta}_{n}(x) 
		& =  
		\frac{-1}{\lambda M^n} \left[ 
		\sum_{m=1}^{M^n} f(x+BW^{(\theta,0,m)}_{R^{(\theta,0,m)}},0) 
		\right]
		\\
		& 
		+ \sum_{k=1}^{n-1} \frac{1}{\lambda M^{(n-k)}} 
		\Bigg[ 
		\sum_{m=1}^ {M^{(n-k)}} 
		\bigg( 
		\lambda \Big[  U^{(\theta,k,m)}_{k}(x+BW^{(\theta,k,m)}_{R^{(\theta,k,m)}})
		- 
		U^{(\theta,k,-m)}_{k-1}(x+BW^{(\theta,k,m)}_{R^{(\theta,k,m)}}) \Big]
		\\& 
		\qquad \qquad \qquad \qquad \qquad - 
		\Big[   
		f\!\left( x+BW^{(\theta,k,m)}_{R^{(\theta,k,m)}}, 
		U^{(\theta,k,m)}_{k}(x+BW^{(\theta,k,m)}_{R^{(\theta,k,m)}}) \right)
		\\
		& \qquad \qquad \qquad \qquad \qquad \qquad  
		- f\!\left(  
		x+BW^{(\theta,k,m)}_{R^{(\theta,k,m)}}, 
		U^{(\theta,k,-m)}_{k-1}(x+BW^{(\theta,k,m)}_{R^{(\theta,k,m)}})
		\right) 
		\Big]
		\bigg)
		\Bigg], 
		\end{split} 
		\end{equation}
	and 
	let $ u \in \{ v \in C(\R^d,\R) \colon (\Forall\varepsilon\in (0,\infty)\colon
		[\sup_{x=(x_1,\ldots,x_d)\in\R^d} (|v(x)|\exp(-\varepsilon \sum_{i=1}^d |x_i|) )] < \infty) \} $ 
	be a viscosity solution of 
		\begin{equation} \label{error_estimate_semilinear_elliptic:viscosity_solution}
		\tfrac12\operatorname{Trace}\!\big(BB^{*}(\operatorname{Hess} u)(x)\big) = f(x,u(x)) 
		\end{equation}  
	for $ x \in \R^d $ (cf.~\cref{cor:existence_and_uniqueness_polynomial_growing_nonlinearity}). 
	Then it holds for all 
		$ n \in \N_0 $, 
		$ \theta \in \Theta $, 
		$ x \in \R^d $, 
		$ \varepsilon \in (0,\infty) $
	with 
		$ 2 (\varepsilon^2+\varepsilon d) [\sup_{y\in\R^d\setminus\{0\}} (\norm{By}\norm{y}^{-1})]^2 < \lambda - L $
	that 
		\begin{equation} \label{error_estimate_semilinear_elliptic:claim}
		\begin{split}
		\left( \Exp{\left| {U^{\theta}_{n}(x) - u(x) }\right|^2 }\right)^{\!\nicefrac12} 
		& \leq  
		\left[ 
		\frac{1}{\sqrt{M}}
		\left( 
		1 + (1 + \sqrt{M})\sqrt{\frac{L}{\lambda}}
		\right) 		
		\right]^n \left[\sup_{y\in\R^d}\left( \frac{|f(y,0)|}{\exp(\varepsilon ( 1 + \norm{y}^2 )^{\nicefrac12})} \right) \right]
		\\
		& \cdot \frac{\exp\!\left(\varepsilon (1 + \norm{x}^2)^{\nicefrac12} \right)}{(\sqrt{\lambda}-\sqrt{L})\sqrt{\lambda-(L+2(\varepsilon^2+\varepsilon d)[\sup_{y\in\R^d\setminus\{0\}}(\norm{By}\norm{y}^{-1})]^2)}}.
		\end{split}
		\end{equation}  
\end{cor}

\begin{proof}[Proof of \cref{cor:error_estimate_semilinear_elliptic}]
	Throughout this proof let 
		$ g \colon \R^d \times \R \to \R $ 
	satisfy for all 
		$ x \in \R^d $, 
		$ v \in \R $ 
	that 
		$ g(x,v) = \lambda v - f(x,v) $, 
	let $ \beta \in [0,\infty) $ satisfy  
		$ \beta = [\sup_{y\in\R^d\setminus\{0\}} (\frac{\norm{By}}{\norm{y}})]^2$, 
	and let 
		$ V_{\varepsilon} \colon \R^d \to (0,\infty) $, $ \varepsilon \in (0,\infty) $, 
	satisfy for all 
		$ \varepsilon \in (0,\infty) $, 
		$ x \in \R^d $ 
	that 
		\begin{equation} 
		V_{\varepsilon}(x) = \exp\!\left( \varepsilon (1+\norm{x}^2)^{\!\nicefrac12} \right)\!.
		\end{equation} 
	Observe that \cref{lem:special_lyapunov_functions} implies for all 
		$ \varepsilon \in (0,\infty) $, 
		$ x \in \R^d $ 
	that 
		\begin{equation} \label{error_estimate_semilinear_elliptic:lyapunov}
		\operatorname{Trace}\!\big( BB^{*}(\operatorname{Hess} V_{\varepsilon})(x) \big) 
		+ 
		\frac{\norm{B^{*}(\nabla V_{\varepsilon})(x)}^2}{V_{\varepsilon}(x)} 
		\leq 
		2(\varepsilon^2 + \varepsilon d) \beta 
		V_{\varepsilon}(x).  
		\end{equation}
	In addition, note that the fact that for all 
		$ x \in \R^d $, 
		$ v \in \R $ 
	it holds that 
		$ g(x,v) = \lambda v - f(x,v) $
	and \eqref{error_estimate_semilinear_elliptic:mlp_scheme} ensure for all 
		$ \theta \in \Theta $, 
		$ n \in \N $, 
		$ x \in \R^d $ 
	that 
		\begin{equation} \label{error_estimate_semilinear_elliptic:transformed_scheme}
		\begin{split}
		& 
		U^{\theta}_{n}(x) 
		=  
		\sum_{k=1}^{n-1} \frac{1}{\lambda M^{(n-k)}} 
		\Bigg[ \sum_{m=1}^{M^{(n-k)}} \bigg( 
		g\!\left( x+BW^{(\theta,k,m)}_{R^{(\theta,k,m)}}, 
		U^{(\theta,k,m)}_{k}(x+BW^{(\theta,k,m)}_{R^{(\theta,k,m)}}) \right)
		\\
		& - g\!\left(  
		x+BW^{(\theta,k,m)}_{R^{(\theta,k,m)}}, 
		U^{(\theta,k,-m)}_{k-1}(x+BW^{(\theta,k,m)}_{R^{(\theta,k,m)}})
		\right) 
		\bigg) 
		\Bigg]
		+ \frac{1}{\lambda M^n} \left[ 
		\sum_{m=1}^{M^n} g(x+BW^{(\theta,0,m)}_{R^{(\theta,0,m)}},0) 
		\right]
		\!.  
		\end{split} 
		\end{equation} 
	Moreover, observe that the fact that for all 
		$ x \in \R^d $, 
		$ v \in \R $ 
	it holds that 
		$ g(x,v) = \lambda v - f(x,v) $ 
	and \eqref{error_estimate_semilinear_elliptic:viscosity_solution} ensure that 
		$ u $ 
	is a viscosity solution of 
		\begin{equation} \label{error_estimate_semilinear_elliptic:equation_solved}
		\lambda u(x) - \tfrac12 \operatorname{Trace}\!\big( BB^{*}(\operatorname{Hess} u)(x) \big) = g(x,u(x)) 
		\end{equation} 
	for $ x \in \R^d $.
	In addition, note that the fact that for all 
		$ \varepsilon \in (0,\infty) $ 
	it holds that 
		$ \sup_{ x \in \R^d } ( \frac{ |g(x,0)|+|u(x)| }{ V_{\varepsilon}(x) } ) < \infty $ 
	guarantees for all 
		$ \varepsilon \in (0,\infty) $ 
	that 
		$ \inf_{ r \in (0,\infty) } [ \sup_{ \norm{x}>r } ( \frac{ |g(x,0)|+|u(x)| }{ V_{\varepsilon}(x) } ) ] = 0 $. 
	\cref{cor:error_estimate_viscosity_solutions} (with 
		$ d \is d $, 
		$ M \is M $, 
		$ \lambda \is \lambda $, 
		$ \Theta \is \Theta $, 
		$ B \is B $, 
		$ f \is g $, 
		$ (\Omega,\mc F,\P) \is (\Omega,\mc F,\P) $, 
		$ (W^{\theta})_{\theta\in\Theta} \is (W^{\theta})_{\theta\in\Theta} $, 
		$ (R^{\theta})_{\theta\in\Theta} \is (R^{\theta})_{\theta \in \Theta} $, 
		$ L \is L $, 
		$ \rho \is (\varepsilon^2+\varepsilon d)\beta $, 
		$ u \is u $, 
		$ V \is V_{\varepsilon} $
	for $ \varepsilon \in (0,\infty) $ 
	in the notation of \cref{cor:error_estimate_viscosity_solutions}), \eqref{error_estimate_semilinear_elliptic:lyapunov},  \eqref{error_estimate_semilinear_elliptic:transformed_scheme}, and \eqref{error_estimate_semilinear_elliptic:equation_solved}  therefore demonstrate for all
		$ \theta \in \Theta $, 
		$ n \in \N_0 $, 
		$ x \in \R^d $,  
		$ \varepsilon \in (0,\infty) $ 
	with 
		$ 2(\varepsilon^2 + \varepsilon d ) \beta < \lambda - L $ 
	that 
		\begin{equation} 
		\begin{split}
		& 
		\frac{\left(\Exp{|U^{\theta}_n(x)-u(x)|^2}\right)^{\!\nicefrac12}}{V_{\varepsilon}(x)} 
		\\& 
		\leq 
		\frac{		\left[ \sup_{y\in\R^d} \left( \frac{|g(y,0)|}{V_{\varepsilon}(y)} \right) \right]}{(\sqrt{\lambda}-\sqrt{L})\sqrt{\lambda-(L+2(\varepsilon^2+\varepsilon d)[\sup_{y\in\R^d\setminus\{0\}}(\frac{\norm{By}}{\norm{y}})]^2)}}
		\left[ \frac{1}{\sqrt{M}}\left( 1 + (1+\sqrt{M})\sqrt{\frac{L}{\lambda}} \right) \right]^n . 
		\end{split} 
		\end{equation} 
	This and the fact that for all 
		$ y \in \R^d $ 
	it holds that 
		$ -f(y,0) = g(y,0) $ 
	establish \eqref{error_estimate_semilinear_elliptic:claim}. The proof of \cref{cor:error_estimate_semilinear_elliptic} is thus completed. 
\end{proof}

\subsection{Computational cost analysis for MLP approximations} \label{subsec:computational_effort}

\begin{lemma} \label{lem:comp_cost_gronwall}
	Let $ \alpha,\beta \in [0,\infty) $, 
	$ M \in [1,\infty) $, 
	$ (C_n)_{n\in\N_0} \subseteq [0,\infty) $ 
	satisfy for all 
	$ n \in \N $ 
	that 
	\begin{equation} \label{comp_cost_gronwall:recursion_relation}
	C_n \leq \alpha M^n + \sum_{k=1}^{n-1} M^{(n-k)} ( \beta + C_k + C_{k-1} ).  
	\end{equation}
	Then it holds for all 
	$ n \in \N $ 
	that 
	\begin{equation} \label{comp_cost_gronwall:claim}
		C_n \leq \left[\frac{\alpha + \beta + C_0 }{2}\right] (2M+1)^n 
			\leq \left[\frac{\alpha + \beta + C_0 }{2}\right] (3M)^n. 
	\end{equation} 
\end{lemma} 

\begin{proof}[Proof of \cref{lem:comp_cost_gronwall}] 
	Throughout this proof let 
		$ c_n \in [0,\infty) $, $ n \in \N_0 $, 
	satisfy for all 
		$ n \in \N_0 $ 
	that 
		$ c_n = \frac{C_n}{M^n} $
	and let 
		$ S_n \in [0,\infty) $, $ n \in \N_0 $, 
	satisfy for all 
		$ n \in \N_0 $ 
	that 
		\begin{equation} \label{comp_cost_gronwall:S_n_definition}
		S_n = 
		\frac{\beta n}{1+\nicefrac{1}{M}} 
		+ \frac{\beta}{(1+\nicefrac{1}{M})^2}
		+ \frac{\alpha-c_0}{1+\nicefrac{1}{M}} 
		+ \sum_{k=0}^n c_k. 
		\end{equation} 
	Note that \eqref{comp_cost_gronwall:recursion_relation}, the fact that $ (c_n)_{n\in\N_0} \subseteq [0,\infty) $, and the assumption that $ M \geq 1 $ ensure for all 
		$ n \in \N $ 
	that 
		\begin{equation} \label{comp_cost_gronwall:transform}
		\begin{split}
		c_n 
		& \leq \alpha + \sum_{k=1}^{n-1} M^{-k}(\beta + C_{k} + C_{k-1}) 
		\leq \alpha + (n-1) \beta + \sum_{k=1}^{n-1} ( c_k + \frac{c_{k-1}}{M} )
		\\
		& = \alpha + (n-1) \beta + \sum_{k=1}^{n-1} c_k + \frac{1}{M} \sum_{k=0}^{n-2} c_{k} 
		\leq 
		\alpha - c_0 + (n-1) \beta + (1+\frac{1}{M})\sum_{k=0}^{n-1} c_k 
		\\
		& = (1+\frac{1}{M}) S_{n-1} - \frac{\beta}{1+\nicefrac{1}{M}} . 
		\end{split} 
		\end{equation}
	Combining this with \eqref{comp_cost_gronwall:S_n_definition} proves for all 
		$ n \in \N \cap [2,\infty) $ 
	that 
		\begin{align} 
		& S_n 
		= S_{n-1} + \frac{\beta}{1+\nicefrac{1}{M}} + c_n  
		\leq (2+\frac{1}{M}) S_{n-1} . 
		\end{align}
	This implies for all 
		$ n \in \N $ 
	that 
		$ S_{n} \leq (2+\frac{1}{M})^{n-1} S_1 $. 
	Combining this with
	\eqref{comp_cost_gronwall:transform} and the assumption that $ \beta \geq 0 $ demonstrates for all 
		$ n \in \N \cap [2,\infty) $ 
	that 
		\begin{equation} \label{comp_cost_gronwall:case_geq_two}
		\begin{split} 
		c_n
		& 
		\leq (1+\nicefrac{1}{M}) S_{n-1} - \frac{\beta}{1+\nicefrac{1}{M}}
		\leq (1+\nicefrac{1}{M}) S_{n-1}
		\leq (1+\nicefrac{1}{M})(2+\nicefrac{1}{M})^{n-2} S_1 . 
		\end{split}
		\end{equation} 
	Next observe that \eqref{comp_cost_gronwall:transform} implies that 
		$ c_1 \leq \alpha $. 
	Combining this with \eqref{comp_cost_gronwall:S_n_definition} ensures that 
		\begin{equation} 
		\begin{split}
		S_1 
		& = \frac{\beta}{1+\nicefrac{1}{M}} + \frac{\beta}{(1+\nicefrac{1}{M})^2} + \frac{\alpha-c_0}{1+\nicefrac{1}{M}} 
		+ c_0 + c_1
		\\
		& \leq \frac{\beta}{1+\nicefrac{1}{M}} + \frac{\beta}{(1+\nicefrac{1}{M})^2} + \frac{\alpha-c_0}{1+\nicefrac{1}{M}}
		+ c_0 + \alpha 
		\\
		& = \frac{2+\nicefrac{1}{M}}{1+\nicefrac{1}{M}}\left[ \frac{\beta}{1+\nicefrac{1}{M}} + \alpha \right] 
		+ \frac{\nicefrac{1}{M}}{1+\nicefrac{1}{M}} C_0 
		\leq \frac{2+\nicefrac{1}{M}}{1+\nicefrac{1}{M}} ( \alpha + \beta + c_0 ). 
		\end{split}
		\end{equation} 
	This and \eqref{comp_cost_gronwall:case_geq_two} show for all 
		$ n \in \N \cap [2,\infty) $
	that 
		\begin{equation} 
		\begin{split} 
		c_n 
		& 
		\leq (2+\nicefrac{1}{M})^{n-2} (1+\nicefrac{1}{M}) \left[\frac{2+\nicefrac{1}{M}}{1+\nicefrac{1}{M}}\right](\alpha + \beta + c_0 ) 
		= (2+\nicefrac{1}{M})^{n-1} ( \alpha + \beta + c_0 )
		\\
		& \leq (2+\nicefrac{1}{M})^n \left[\frac{\alpha+\beta+c_0}{2}\right]\!.
		\end{split} 
		\end{equation} 
	Combining this with the fact that $ c_1 \leq \alpha \leq (2+\nicefrac{1}{M}) \frac{\alpha+\beta+c_0}{2} $ and the fact that for all 
		$ n \in \N_0 $ 
	it holds that $ C_n = c_n M^n $ establishes \eqref{comp_cost_gronwall:claim}. The proof of \cref{lem:comp_cost_gronwall} is thus completed. 
\end{proof} 

\begin{lemma} \label{lem:elementary_comp_cost_estimate}
	Let $ m \in \N_0 $, 
		$ \alpha,\beta,\kappa_1,\kappa_2 \in \R $ 
	satisfy  	
		$ 0 < \alpha < 1 < \beta $, 
	let $ c_n, e_n \in [0,\infty) $, $ n \in \N_0 \cap [m,\infty) $, 
	assume for all $ n \in \N_0 \cap [m,\infty) $ that 
		$ e_n \leq \kappa_1 \alpha^n $
	and  
		$ c_n \leq \kappa_2 \beta^n $, 
	and let 
		$ N \colon (0,1] \to \N_0 \cap [m,\infty) $ 
	satisfy for all 
		$ \varepsilon \in (0,1] $ 
	that 
		\begin{equation} \label{elementary:N_epsilon}
		N_{\varepsilon} = \min\{n\in\N_0\cap [m,\infty) \colon \kappa_1 \alpha^n \leq \varepsilon \}. 
		\end{equation} 
	Then it holds for all 
		$ \varepsilon \in (0,1] $ 
	that $ e_{N_{\varepsilon}} \leq \varepsilon $ and  
		\begin{equation} \label{elementary:claim}
		c_{N_{\varepsilon}} \leq \kappa_2 \max\left\{ \beta^m, \beta \kappa_1^{(\frac{\ln(\beta)}{\ln(1/\alpha)})} \right\} \left[ \frac{1}{\varepsilon} \right]^{(\frac{\ln(\beta)}{\ln(1/\alpha)})}. 
		\end{equation} 
\end{lemma}

\begin{proof}[Proof of \cref{lem:elementary_comp_cost_estimate}]
	First, observe that \eqref{elementary:N_epsilon} ensures for all 
		$ \varepsilon \in (0,1] $ 
	that 
		$ e_{N_{\varepsilon}} \leq \varepsilon $. 
	Moreover, note that \eqref{elementary:N_epsilon} guarantees for all 
		$ \varepsilon \in (0,1] $ 
	with $ N_{\varepsilon} \geq m + 1 $ that $ \kappa_1 \alpha^{N_{\varepsilon}-1} > \varepsilon $. 
	Hence, we obtain for all $ \varepsilon \in (0,1] $ 
	with $ N_{\varepsilon} \geq m + 1 $ that 
		\begin{equation} 
		(N_{\varepsilon}-1) \ln(\nicefrac{1}{\alpha}) < \ln(\kappa_1)-\ln(\varepsilon). 
		\end{equation}
	This implies for all 
		$ \varepsilon \in (0,1] $ 
	with $ N_{\varepsilon} \geq m + 1 $ that
		\begin{equation} \label{elementary:eq01}
		\begin{split}
		c_{N_{\varepsilon}} 
		& \leq \kappa_2 \beta^{N_{\varepsilon}} 
		= \kappa_2 \beta \exp((N_{\varepsilon}-1)\ln(\beta)) 
		\leq 
		\kappa_2 \beta \exp\!\left((\ln(\kappa_1)-\ln(\varepsilon))\frac{\ln(\beta)}{\ln(\nicefrac{1}{\alpha})}\right)
		\\
		& = \left[ \kappa_2 \beta \kappa_1^{(\frac{\ln(\beta)}{\ln(1/\alpha)})} \right] \left[ \frac{1}{\varepsilon} \right]^{(\frac{\ln(\beta)}{\ln(1/\alpha)})}. 
		\end{split} 
		\end{equation} 
	Next note that the assumption that for all 
		$ n \in \N_0 \cap [m,\infty) $ 
	it holds that 
		$ c_n \leq \kappa_2 \beta^n $ 
	and the fact that 
		$ \frac{\ln(\beta)}{\ln(\nicefrac{1}{\alpha})} \in (0,\infty) $
	ensure that for all 
		$ \varepsilon \in (0,1] $ 
	with $ N_{\varepsilon} = m $ it holds that
		\begin{equation} 
		c_{N_{\varepsilon}} = c_m \leq \kappa_2 \beta^m \leq \kappa_2 \beta^m \left[ \frac{1}{\varepsilon} \right]^{(\frac{\ln(\beta)}{\ln(1/\alpha)})}. 
		\end{equation} 
	This and \eqref{elementary:eq01} establish \eqref{elementary:claim}. The proof of \cref{lem:elementary_comp_cost_estimate} is thus completed. 
\end{proof} 

\subsection{Overall complexity analysis for MLP approximations}
\label{subsec:complexity_analysis}

\begin{theorem} \label{prop:overcoming_the_curse}
	Let $ \kappa,L,\mf n,p,q,r,s \in [0,\infty) $, 
		$ \lambda \in (L,\infty) $,
		$ M \in \N \cap ( (\sqrt{\lambda}+\sqrt{L})^2 (\sqrt{\lambda}-\sqrt{L})^{-2} , \infty ) $, 
		$ \alpha = -\ln(3M)[\ln( \nicefrac{1}{\sqrt{M}} + (1+\nicefrac{1}{\sqrt{M}})\sqrt{\nicefrac{L}{\lambda}} )]^{-1} $, 
		$ \Theta = \cup_{n\in\N} \Z^n $, 
	let $ u_d \in C(\R^d,\R) $, $ d \in \N $, 
	let $ B_d \in \R^{d\times d} $, $ d \in \N $, 
	let $ f_d \in C(\R^d\times\R,\R) $, $ d \in \N $,  
	assume for every $ d \in \N $ that $ u_d $ is a viscosity solution of 
		\begin{equation} 
		\tfrac12\operatorname{Trace}(B_d(B_d)^{*}(\operatorname{Hess} u_d)(x)) = f_d(x,u_d(x)) 
		\end{equation} 
	for $ x \in \R^d $, 
	let $ ( \Omega, \mc F, \P ) $ be a probability space, 
	let $ W^{d,\theta} \colon [0,\infty) \times \Omega \to \R^d $, $ \theta\in\Theta$, $ d \in \N $, be i.i.d.~standard Brownian motions, 
	let $ R^{\theta} \colon \Omega \to [0,\infty) $, $ \theta \in \Theta $, be i.i.d.~random variables, 
	assume that 
		$ (R^{\theta})_{\theta\in\Theta} $ and $ (W^{d,\theta})_{(d,\theta)\in\N\times\Theta}$ are independent, 
	let	$ \norm{\cdot} \colon (\cup_{d\in\N} \R^d) \to [0,\infty) $ satisfy for all 
		$ d \in \N $, 
		$ x = (x_1,\ldots,x_d) \in \R^d $ 
	that 
		$ \norm{x} = [ \sum_{i=1}^d |x_i|^2 ]^{\nicefrac12} $, 
	assume for all 
		$ d \in \N $, 
		$ \varepsilon \in (0,\infty) $, 
		$ x \in \R^d $, 
		$ v,w \in \R $ 
	that 
		$ \Norm{B_d x} \leq \kappa d^r \Norm{x} $, 
		$ | f_d(x,0) | \leq \kappa d^{s}( 1 + \norm{x}^p ) $, 	
		$ | f_d(x,v) - f_d(x,w) - \lambda ( v - w ) | \leq L | v-w | $, 
		$ \sup_{y\in\R^d} [|u_d(y)|\exp(-\varepsilon \Norm{y})] < \infty $,  
	and  	
		$ \P(R^0 \geq \varepsilon) = e^{-\lambda \varepsilon} $, 
	let $ U^{d,\theta}_{n} = (U^{d,\theta}_{n}(x))_{x\in\R^d}\colon\R^d\times\Omega\to\R$, $\theta\in\Theta$, $ d, n \in \N_0 $, 
	satisfy for all 
		$ d,n \in \N $, 
		$ \theta\in\Theta $, 
		$ x \in \R^d $
	that 
		$ U^{d,\theta}_{0}(x) = 0 $ 
	and 
		\begin{equation} 
		\begin{split}
		&U^{d,\theta}_{n}(x) 
		=  
		\frac{-1}{\lambda M^n} \left[ 
		\sum_{m=1}^{M^n} f_d(x+B_dW^{d,(\theta,0,m)}_{R^{(\theta,0,m)}},0) 
		\right]
		\\
		& 
		+ \sum_{k=1}^{n-1} \frac{1}{\lambda M^{(n-k)}} 
		\Bigg[ 
		\sum_{m=1}^ {M^{(n-k)}} 
		\bigg( 
		\lambda \Big[  U^{d,(\theta,k,m)}_{k}(x+B_dW^{d,(\theta,k,m)}_{R^{(\theta,k,m)}})
		- 
		U^{d,(\theta,k,-m)}_{k-1}(x+B_dW^{d,(\theta,k,m)}_{R^{(\theta,k,m)}}) \Big]
		\\& 
		\qquad \qquad \qquad \qquad \qquad - 
		\Big[   
		f_d\!\left( x+B_dW^{d,(\theta,k,m)}_{R^{(\theta,k,m)}}, 
		U^{d,(\theta,k,m)}_{k}(x+B_dW^{d,(\theta,k,m)}_{R^{(\theta,k,m)}}) \right)
		\\
		& \qquad \qquad \qquad \qquad \qquad \qquad  
		- f_d\!\left(  
		x+B_dW^{d,(\theta,k,m)}_{R^{(\theta,k,m)}}, 
		U^{d,(\theta,k,-m)}_{k-1}(x+B_dW^{d,(\theta,k,m)}_{R^{(\theta,k,m)}})
		\right) 
		\Big]
		\bigg)
		\Bigg], 
		\end{split} 
		\end{equation}
	and let $ \Cost_{d,n} \in \R $, $ d,n\in\N_0$, 
	satisfy for all 
		$ d,n \in \N_0 $ 
	that 
		$ \Cost_{d,n} \leq (d+1) M^n + \sum_{k=1}^{n-1} M^{(n-k)} ( d + 1 + \Cost_{d,k} + \Cost_{d,k-1} ) $.
	Then there exist $ c \in \R $ and $ \mf N \colon (0,1] \times \N \to (\Z \cap [\mf n,\infty)) $
	such that for all 
		$ \varepsilon \in (0,1] $, 
		$ d \in \N $
	it holds that 	
		\begin{equation} \label{overcoming_the_curse:claim}
		\begin{split}
		& 
		\Cost_{d,\mf N_{\varepsilon,d}} \leq c d^{1 + \alpha ( s + p \max\{q,2r+1\} )}
		{\varepsilon}^{-\alpha} 
		\quad \text{and} \quad
		\sup\nolimits_{ x \in \R^d, \norm{x} \leq \kappa d^q}
		\big( \EXPP{ | u_d(x) - U^{d,0}_{\mf N_{\varepsilon,d}}(x) |^2 }\big)^{\!\nicefrac12} 
		\leq \varepsilon . 
		\end{split}
		\end{equation} 
\end{theorem} 

\begin{proof}[Proof of \cref{prop:overcoming_the_curse}] 
	Throughout this proof let 
		$ \lceil \cdot \rceil \colon \R \to \Z $ 
	satisfy for all 
		$ x \in \R $ 
	that 
		$ \lceil x \rceil = \min( \Z \cap [x,\infty) ) $, 		
	let $ c_1, c_2, c_3 \in \R $ 
	satisfy for all 
		$ d \in \N $, 
		$ x \in \R^d $ 
	that 
		$ \norm{B_d x} \leq c_1 d^{r} \Norm{x} $,
		$ |f_d(x,0)| \leq c_2 d^{s} (1+\norm{x}^p) $, 	
	and
		\begin{equation} 
		c_3 = \frac{4c_2((\lceil p\rceil)!)}{(\sqrt{\lambda}-\sqrt{L})^2}
		\max\left\{1,\left[\frac{8|c_1|^2}{\lambda-L}\right]^{p} \right\}
		\exp( 1 + \kappa ),  
		\end{equation} 
	and let 
		$ \eta_d \in (0,\infty) $, $ d \in \N $, 
	satisfy for all 
		$ d \in \N $ 
	that 
		$ \eta_d = \sup\{ t \in [0,\infty) \colon (8|c_1|^2 d^{2r+1} t \leq \lambda - L~\text{and}~ t d^q \leq 1) \} $. 
	Observe that the fact that for all 
		$ d \in \N $ 
	it holds that 
		$ \eta_d \leq 1 $, 
	the fact that for all 
		$ d \in \N $, 
		$ x \in \R^d $ 
	it holds that 
		$ \norm{ B_d x } \leq c_1 d^r \norm{x} $,  
	and the fact that for all 
		$ d \in \N $ 
	it holds that 
		$ 8 \eta_d |c_1|^2 d^{2r+1} \leq \lambda - L $
	ensure that for all 
		$ d \in \N $
	it holds that 
		\begin{equation} 
		|\eta_d|^2 \left[ \sup_{x\in\R^d\setminus\{0\}} \left(\frac{\norm{B_dx}}{\norm{x}}\right) \right]^2 
		\leq 
		d \eta_d \left[ \sup_{x\in\R^d\setminus\{0\}} \left(\frac{\norm{B_dx}}{\norm{x}}\right) \right]^2 
		\leq 
		\eta_d |c_1|^2 d^{2r+1} \leq \frac{\lambda - L}{8}. 
		\end{equation} 
	This implies for all 
		$ d \in \N $ 
	that 
		\begin{equation} 
		\lambda - L - 2(|\eta_d|^2+d\eta_d) \left[\sup_{x\in\R^d\setminus\{0\}} \left(\frac{\norm{B_d x}}{\norm{x}}\right)\right]^2
		\geq
		\frac{\lambda - L}{2} . 
		\end{equation}  
	Next note that the fact that for all 
		$ d \in \N $, 
		$ x = (x_1,\ldots,x_d) \in \R^d $ 
	it holds that 
		$ \norm{x} \leq \sum_{i=1}^d |x_i| $	
	and the assumption that for all 
		$ d \in \N $, 
		$ \varepsilon \in (0,\infty) $ 
	it holds that 
		$ \sup_{ x \in \R^d } [ |u_d(x)| \exp(-\varepsilon\norm{x})] < \infty $ 
	prove that for all 
		$ d \in \N $,
		$ \varepsilon \in (0,\infty) $
	it holds that 
		\begin{equation} 
		\sup_{x=(x_1,\ldots,x_d)\in\R^d} \left[ \frac{|u_d(x)|}{\exp(\varepsilon \sum_{i=1}^d |x_i| )} \right] 
		\leq 
		\sup_{x\in\R^d} \left[ \frac{|u_d(x)|}{\exp(\varepsilon \Norm{x})} \right] < \infty. 
		\end{equation} 
	\cref{cor:error_estimate_semilinear_elliptic} (with 
		$ d \is d $, 
		$ M \is M $, 
		$ B \is B_d $, 
		$ L \is L $, 
		$ \lambda \is \lambda $, 
		$ \Theta \is \Theta $, 
		$ f \is f_d $, 
		$ (\Omega,\mc F,\P) \is (\Omega,\mc F,\P) $, 
		$ (W^{\theta})_{\theta\in\Theta} \is (W^{d,\theta})_{\theta\in\Theta} $, 
		$ (R^{\theta})_{\theta\in\Theta} \is (R^{\theta})_{\theta\in\Theta} $, 
		$ u \is u_d $
	for $ d \in \N $
	in the notation of \cref{cor:error_estimate_semilinear_elliptic}) therefore ensures for all 
		$ d \in \N $, 
		$ n \in \N_0 $, 
		$ x \in \R^d $
	that 
		\begin{equation} 
		\begin{split}
		&
		\left( 
		\Exp{|U^{d,0}_{n}(x) - u_d(x)|^2}
		\right)^{\!\nicefrac12}
		\\
		& \leq 
		\left[ \frac{1}{\sqrt{M}}\left( 1 + ( 1 + \sqrt{M} ) \sqrt{\nicefrac{L}{\lambda}} \right) \right]^n 
		\left[ \sup_{y\in\R^d} \left( \frac{|f_d(y,0)|}{\exp(\eta_d ( 1 + \norm{y}^2 )^{\nicefrac12})} \right) \right] 
		\frac{\sqrt2\exp(\eta_d(1+\norm{x}^2)^{\nicefrac12})}{(\sqrt{\lambda}-\sqrt{L})\sqrt{\lambda - L}}
		\\
		& \leq 
		2c_2 \left[ \frac{1}{\sqrt{M}}\left( 1 + ( 1 + \sqrt{M} )\sqrt{\nicefrac{L}{\lambda}} \right) \right]^n 
		\left[ \sup_{y\in\R^d} \left( \frac{ d^{s}(1 + \norm{y}^p)  }{\exp(\eta_d(1+\norm{y}^2)^{\nicefrac12})} \right) \right]\frac{\exp(\eta_d(1+\norm{x}^2)^{\nicefrac{1}{2}})}{(\sqrt{\lambda}-\sqrt{L})^2}
		\\
		& \leq 
		4c_2 \left[ \frac{1}{\sqrt{M}}\left( 1 + ( 1 + \sqrt{M} )\sqrt{\nicefrac{L}{\lambda}} \right) \right]^n 
		\left[ \sup_{y\in\R^d} \left( \frac{ d^s(1 + \norm{y}^2)^{\nicefrac{p}{2}}  }{\exp(\eta_d(1+\norm{y}^2)^{\nicefrac12})} \right) \right]\frac{\exp(\eta_d(1+\norm{x}^2)^{\nicefrac{1}{2}})}{(\sqrt{\lambda}-\sqrt{L})^2}
		. 
		\end{split} 
		\end{equation} 
	The fact that for all 
		$ x \in (0,\infty) $
	it holds that 
		$ x^p \leq ((\lceil p \rceil)!)\exp(x) $
	and the fact that for all 
		$ d \in \N $ 
	it holds that 
		$ \eta_d d^q \leq 1 $
	hence imply that for all 
		$ d \in \N $, 
		$ n \in \N_0 $
	it holds that 
		\begin{equation}  \label{overcoming_the_curse:error_estimate}
		\begin{split}
		& 
		\sup_{x \in \R^d, \norm{x} \leq \kappa d^q}
		\left( 
		\Exp{|U^{d,0}_{n}(x) - u_d(x)|^2}
		\right)^{\!\nicefrac12} 
		\\&
		\leq 
		\frac{4c_2((\lceil p \rceil)!)}{(\sqrt{\lambda}-\sqrt{L})^2}
		\left[ \frac{1}{\sqrt{M}}\left( 1 + ( 1 + \sqrt{M} ) \sqrt{\nicefrac{L}{\lambda}} \right) \right]^n 
		\frac{d^s}{\eta_d^{p}}
		 \exp\!\left( 1 + \kappa \right)  
		\\& 
		\leq 
		c_3 \max\left\{ d^{s+pq}, d^{s+p(2r+1)} \right\} \left[ \frac{1}{\sqrt{M}}\left( 1 + ( 1 + \sqrt{M} ) \sqrt{\nicefrac{L}{\lambda}} \right) \right]^n . 
		\end{split}
		\end{equation}
	Moreover, observe that \cref{lem:comp_cost_gronwall} (with 
		$ \alpha \is d+1 $, 
		$ \beta \is d+1 $, 
		$ M \is M $,
		$ (C_n)_{ n \in \N_0 } \is (\max\{\Cost_{d,n},0\})_{ n \in \N_0 } $ 
	for $ d \in \N $ 
	in the notation of \cref{lem:comp_cost_gronwall}) ensures for all 
		 $ d,n \in \N $ 
	that 
		$ \Cost_{d,n} \leq 3d(3M)^n $. 
	Combining this with the fact that for all 
		$ d \in \N $ 
	it holds that 
		$ \Cost_{d,0} \leq d+1 \leq 3d $ 
	implies that for all 
		$ d \in \N $, 
		$ n \in \N_0 $ 
	it holds that 
		$ \Cost_{d,n} \leq 3d(3M)^n $. 
	Next let 
		$ N \colon (0,1] \times \N \to \N_0 $ 
	satisfy for all 
		$ d \in \N $, 
		$ \varepsilon \in (0,1] $
	that 
		\begin{equation}
		\begin{split}
		N_{\varepsilon,d} 
		= 
		\min\left\{ n \in \N_0 \cap [ \mf n, \infty) \colon 
		c_3 d^{s+p\max\{q,2r+1\}}\left[ \frac{1}{\sqrt{M}}\left( 1 + ( 1 + \sqrt{M} ) \sqrt{\nicefrac{L}{\lambda}} \right) \right]^n 
		\leq 
		\varepsilon 
		\right\}\!. 
		\end{split}
		\end{equation} 
	\cref{lem:elementary_comp_cost_estimate} (with 
		$ \alpha \is \frac{1}{\sqrt{M}}[1+(1+\sqrt{M})\sqrt{\nicefrac{L}{\lambda}}] $, 
		$ \beta \is 3M $, 
		$ \kappa_1 \is c_3 d^{s+p\max\{q,2r+1\}} $, 
		$ \kappa_2 \is 3d $, 
		$ m \is  \lceil\mf n\rceil $, 
		$ (e_n)_{n\in\Z\cap[ \mf n ,\infty) } \is (\sup_{x\in\R^d,\norm{x}\leq \kappa d^q} (\EXP{|U^{d,0}_{n}(x)-u_d(x)|^2})^{\nicefrac12})_{n\in\Z\cap [ \mf n , \infty) } $, 
		$ (c_n)_{n\in\Z\cap[ \mf n ,\infty)} \is (\max\{ \Cost_{d,n}, 0\} )_{n\in\Z\cap[ \mf n ,\infty)}  $
	for $ d \in \N $
	in the notation of \cref{lem:elementary_comp_cost_estimate}), the fact that for all 
		$ d \in \N $,
		$ n \in \N_0 $
	it holds that $ \Cost_{d,n} \leq 3d(3M)^n $, 
	and \eqref{overcoming_the_curse:error_estimate} therefore imply that for all 
		$ d \in \N $, 
		$ \varepsilon \in (0,1] $
	it holds that 
		$ \sup_{x \in \R^d, \norm{x}\leq \kappa d^q} (\EXP{|U^{d,0}_{N_{\varepsilon,d}}(x)-u_d(x)|^2})^{\nicefrac{1}{2}} \leq \varepsilon $ 
	and 
		\begin{equation} 
		\begin{split}
		\Cost_{d,N_{\varepsilon, d}}
		& 
		\leq 
		3 d
		\max\left\{ (3M)^{\lceil \mf n \rceil}, 3M (c_3  d^{ s + p \max\{ q , 2r+1 \} })^{\alpha} \right\}
		\left[ \frac{1}{\varepsilon} \right]^{\alpha}
		\\
		& \leq  
		\max\left\{ 3 (3M)^{\lceil \mf n \rceil}, 9M c_3^{\alpha} \right\}
		d^{ 1 + ( s + p \max\{ q , 2r+1 \} ) \,\alpha }
		\left[ \frac{1}{\varepsilon} \right]^{\alpha}. 
		\end{split}
		\end{equation} 
	This establishes \eqref{overcoming_the_curse:claim}. The proof of \cref{prop:overcoming_the_curse} is thus completed. 
\end{proof}

\begin{cor} \label{prop:overcoming_the_curse_2}
	Let $ c,L \in [0,\infty) $, 
		$ \lambda \in (L,\infty) $,
		$ M \in \N \cap ( (\sqrt{\lambda}+\sqrt{L})^2(\sqrt{\lambda}-\sqrt{L})^{-2}, \infty ) $, 
		$ \Theta = \cup_{n\in\N} \Z^n $, 
	let $ u_d \in C(\R^d,\R) $, $ d \in \N $, 
	let $ B_d \in \R^{d\times d} $, $ d \in \N $, 
	let $ f_d \in C(\R^d\times\R,\R) $, $ d \in \N $, 
	assume for every $ d \in \N $ that $ u_d $ is a viscosity solution of 
		\begin{equation} 
		\tfrac12\operatorname{Trace}(B_d(B_d)^{*}(\operatorname{Hess} u_d)(x)) = f_d(x,u_d(x)) 
		\end{equation} 
	for $ x \in \R^d $, 
	let $ ( \Omega, \mc F, \P ) $ be a probability space, 
	let $ W^{d,\theta} \colon [0,\infty) \times \Omega \to \R^d $, $ \theta\in\Theta$, $ d \in \N $, be i.i.d.~standard Brownian motions, 
	let $ R^{\theta} \colon \Omega \to [0,\infty) $, $ \theta \in \Theta $, be i.i.d.~random variables, 
	assume that 
		$ (R^{\theta})_{\theta\in\Theta} $ and $ (W^{d,\theta})_{(d,\theta)\in\N\times\Theta}$ are independent, 
	let	$ \norm{\cdot} \colon (\cup_{d\in\N} \R^d) \to [0,\infty) $ satisfy for all 
		$ d \in \N $, 
		$ x = (x_1,\ldots,x_d) \in \R^d $ 
	that 
		$ \norm{x} = [ \sum_{i=1}^d |x_i|^2 ]^{\nicefrac12} $, 
	assume for all 
		$ d \in \N $, 
		$ x \in \R^d $, 
		$ v,w \in \R $, 
		$ \varepsilon \in (0,\infty) $ 
	that 
		$ \norm{B_dx} \leq c d^{c} \Norm{x} $,  
		$ | f_d(x,v) - f_d(x,w) - \lambda ( v - w ) | \leq L | v-w | $, 
		$ |f_d(x,0)| \leq c d^{c} (1+\norm{x}^c) $,  
		$ \sup_{y\in\R^d} [|u_d(y)|\exp(-\varepsilon \Norm{y})] < \infty $,  
	and 
		$ \P(R^0 \geq \varepsilon) = e^{-\lambda \varepsilon} $, 
	let $ U^{d,\theta}_{n} = (U^{d,\theta}_{n}(x))_{x\in\R^d}\colon\R^d\times\Omega\to\R$, $\theta\in\Theta$, $d,n\in\N_0$, 
	satisfy for all 
		$ d,n \in \N $, 
		$ \theta\in\Theta $, 
		$ x \in \R^d $
	that 
		$ U^{d,\theta}_{0}(x) = 0 $ 
	and 
		\begin{equation} 
		\begin{split}
		&U^{d,\theta}_{n}(x) 
		=  
		\frac{-1}{\lambda M^n} \left[ 
		\sum_{m=1}^{M^n} f_d(x+B_dW^{d,(\theta,0,m)}_{R^{(\theta,0,m)}},0) 
		\right]
		\\
		& 
		+ \sum_{k=1}^{n-1} \frac{1}{\lambda M^{(n-k)}} 
		\Bigg[ 
		\sum_{m=1}^ {M^{(n-k)}} 
		\bigg( 
		\lambda \Big[  U^{d,(\theta,k,m)}_{k}(x+B_dW^{d,(\theta,k,m)}_{R^{(\theta,k,m)}})
		- 
		U^{d,(\theta,k,-m)}_{k-1}(x+B_dW^{d,(\theta,k,m)}_{R^{(\theta,k,m)}}) \Big]
		\\& 
		\qquad \qquad \qquad \qquad \qquad - 
		\Big[   
		f_d\!\left( x+B_dW^{d,(\theta,k,m)}_{R^{(\theta,k,m)}}, 
		U^{d,(\theta,k,m)}_{k}(x+B_dW^{d,(\theta,k,m)}_{R^{(\theta,k,m)}}) \right)
		\\
		& \qquad \qquad \qquad \qquad \qquad \qquad  
		- f_d\!\left(  
		x+B_dW^{d,(\theta,k,m)}_{R^{(\theta,k,m)}}, 
		U^{d,(\theta,k,-m)}_{k-1}(x+B_dW^{d,(\theta,k,m)}_{R^{(\theta,k,m)}})
		\right) 
		\Big]
		\bigg)
		\Bigg], 
		\end{split} 
		\end{equation}
	and let $ \Cost_{d,n} \in \R $, $ d,n\in\N_0$, 
	satisfy for all 
		$ d,n \in \N_0 $ 
	that 
		$ \Cost_{d,n} \leq (d+1) M^n + \sum_{k=1}^{n-1} M^{(n-k)} ( d + 1 + \Cost_{d,k} + \Cost_{d,k-1} ) $. 
	Then there exist $ \kappa \in \R $ and $ \mf N \colon (0,1] \times \N \to \N $
	such that for all 
		$ \varepsilon \in (0,1] $, 
		$ d \in \N $
	it holds that 
		\begin{equation} \label{overcoming_the_curse2:claim}
		\begin{split}
		& 
		\Cost_{d,\mf N_{\varepsilon,d}} \leq \kappa d^{\kappa}
		{\varepsilon}^{-\kappa} 
		\qquad \text{and} \qquad
		\sup\nolimits_{ x \in [-cd^c,cd^c]^d }
		\big( \EXPP{ | u_d(x) - U^{d,0}_{\mf N_{\varepsilon,d}}(x) |^2 }\big)^{\!\nicefrac12} 
		\leq \varepsilon . 
		\end{split}
		\end{equation} 
\end{cor} 

\begin{proof}[Proof of \cref{prop:overcoming_the_curse_2}]
	Throughout this proof let $ \alpha \in \R $ satisfy	 
		$ \alpha = -\ln(3M)[\ln(\nicefrac{1}{\sqrt{M}}+(1+\nicefrac{1}{\sqrt{M}})\sqrt{\nicefrac{L}{\lambda}})]^{-1} $. 
	Note that \cref{prop:overcoming_the_curse} (with 
		$ \kappa \is c $, 
		$ L \is L $, 
		$ \mf n \is 1 $,
		$ p \is c $, 
		$ q \is c+\nicefrac12 $, 
		$ r \is c $, 
		$ s \is c $,
		$ \lambda \is \lambda $, 
		$ M \is M $, 
		$ \Theta \is \Theta $, 
		$ (u_d)_{ d \in \N } \is (u_d)_{ d \in \N } $, 
		$ (B_d)_{ d \in \N } \is (B_d)_{ d \in \N } $, 
		$ (f_d)_{ d \in \N } \is (f_d)_{ d \in \N } $, 
		$ ( \Omega, \mc F, \P ) \is ( \Omega, \mc F, \P ) $, 
		$ (W^{d,\theta})_{(d,\theta)\in\N\times\Theta} \is (W^{d,\theta})_{(d,\theta)\in\N\times\Theta} $, 
		$ (R^{\theta})_{ \theta \in \Theta } \is (R^{\theta})_{ \theta \in \Theta } $, 
		$ (\Cost_{d,n})_{ (d,n) \in \N_0\times\N_0 } \is (\Cost_{d,n})_{ (d,n) \in \N_0\times\N_0 } $ 		
	in the notation of \cref{prop:overcoming_the_curse}) ensures that there exist $ \gamma \in (0,\infty) $ and $ \mf N \colon (0,1] \times \N \to \N $ which satisfy for all 
		$ \varepsilon \in (0,1] $, 
		$ d \in \N $
	that 
		\begin{equation} 
		\Cost_{d,\mf N_{\varepsilon,d}} \leq \gamma d^{1+2(c^2+c)\alpha} \varepsilon^{-\alpha} 
		\qquad
		\text{and} 
		\,\,\,
		\sup_{ x \in \R^d, \norm{x} \leq c d^{(c+\nicefrac12)}} 
		\left( \Exp{ \big| U^{d,0}_{\mf N_{\varepsilon,d}}(x) - u_d(x) \big|^2 }\right)^{\!\nicefrac12} 
		\leq \varepsilon . 
		\end{equation} 
	This, the fact that for all 
		$ d \in \N $, $ x \in [-cd^c,cd^c]^d $ 
	it holds that 
		$ \norm{x} \leq c d^{c+\nicefrac12} $, 
	and   
	the fact that for all 
		$ d \in \N $, $ \varepsilon \in (0,1] $ 
	it holds that 
		$ \gamma d^{1+2(c^2+c)\alpha} \varepsilon^{-\alpha} \leq \max\{1,\gamma,\alpha,1+2(c^2+c)\alpha\} (\nicefrac{d}{\varepsilon})^{\max\{1,\gamma,\alpha,1+2(c^2+c)\alpha\}} $ 
	establish \eqref{overcoming_the_curse2:claim}. The proof of \cref{prop:overcoming_the_curse_2} is thus completed. 
\end{proof} 

\subsection*{Acknowledgements}

The third author acknowledges funding by the Deutsche Forschungsgemeinschaft (DFG, German Research Foundation) under Germany's Excellence Strategy EXC 2044-390685587, Mathematics Muenster: Dynamics-Geometry-Structure. 

\bibliographystyle{acm}
\bibliography{References}

\end{document}